\documentclass[10pt]{amsart}
\usepackage{amsfonts}
\usepackage{amsmath,amsthm} 
\usepackage{hyperref} 
\usepackage{latexsym}
\usepackage{array}
\usepackage{amssymb}
\usepackage{enumerate}
\usepackage[francais,english]{babel}

\usepackage{color}
\usepackage[latin1]{inputenc}

\theoremstyle{plain} 
\newtheorem{theorem}{Theorem}[section]
\newtheorem*{theorem*}{Theorem}
\newtheorem{lemma}[theorem]{Lemma}
\newtheorem{proposition}[theorem]{Proposition}
\newtheorem{corollary}[theorem]{Corollary} 
\newtheorem{definition}[theorem]{Definition} 

\theoremstyle{remark}
\newtheorem{remark}[theorem]{Remark}

\newtheorem*{lem*}{Lemma}
\newtheorem*{sublem*}{Sublemma}
\newtheorem*{remark*}{Remark}
\newtheorem*{NB*}{NB}

\newcommand{\R}{  \mathbb{R}   }
\newcommand{\C}{  \mathbb{C}   }
\newcommand{\Z}{  \mathbb{Z}   }
\newcommand{\N}{  \mathbb{N}   }

\newcommand{\T}{  \mathbb{T}   }
\newcommand{\A}{  \mathbb{A}   }

\newcommand{\cA}{  \mathcal{A}   }
\newcommand{\cB}{  \mathcal{B}   }
\newcommand{\cC}{  \mathcal{C}   }
\newcommand{\Ca}{  \mathcal{C}   }

\newcommand{\D}{  \mathcal{D}   }

\newcommand{\E}{  \mathcal{E}   }
\newcommand{\F}{  \mathcal{F}   }

\renewcommand{\L}{  \mathcal{L}   }

\newcommand{\cM}{  \mathcal{M}   }

\newcommand{\NF}{  \mathcal{NF}   }
\renewcommand{\O}{  \mathcal{O}   }
\renewcommand{\P}{  \mathcal{P}   }

\newcommand{\Tc}{  \mathcal{T}   }
\newcommand{\cT}{  \mathcal{T}   }

\newcommand{\om}{  \omega   }
\newcommand{\Om}{  \Omega   }

\newcommand{\ga}{\gamma   }
\newcommand{\s}{  \sigma   }
\newcommand{\De}{  (-\Delta) }
\newcommand{\ka}{  \kappa   }
\renewcommand{\r}{  \rho   }

\renewcommand{\phi}{  \varphi  }
\newcommand{\eps}{\varepsilon}
\newcommand{\vark}{  \varkappa   }
\newcommand{\de}{  \delta   }

\newcommand{\LL}{  \Lambda^{\#}}

\newcommand{\la}{  \lambda_a   }
\newcommand{\lb}{  \lambda_b   }

\newcommand{\diag}{\operatorname{diag}}
\newcommand{\Leb}{\operatorname{Leb}}
\newcommand{\Mat}{\operatorname{Mat}}
\newcommand{\meas}{\operatorname{meas}}

\newcommand{\dist}{\underline{\operatorname{dist}}}
\newcommand{\cAd}{\operatorname{ad}}
\newcommand{\diam}{\operatorname{diam}}
\newcommand{\Id}{\operatorname{Id}}

\newcommand{\cte}{ {\operatorname{ct.}  } }
\newcommand{\Cte}{ {\operatorname{Ct.}  } }
\newcommand{\Haus}{ {\operatorname{Hausdorff}  } }

\newcommand{\lsim}{  \lesssim   }
\newcommand{\gsim}{  \gtrsim   }

\def\ab#1{\left|#1\right|}
\def\aa#1{\left\Vert#1\right\Vert}

\newcommand{\be}{\begin{equation}}
\newcommand{\ee}{\end{equation}}
\newcommand{\ben}{\begin{equation*}}
\newcommand{\een}{\end{equation*}}
\newcommand{\ban}{\begin{align*}}
\newcommand{\ean}{\end{align*}}
\numberwithin{equation}{section}

\newcommand{\dd}{  \text{d}   }

\newcommand{\p}{ \partial}

\author{L. Hakan Eliasson}
\address{Univ. Paris Diderot, Sorbonne Paris Cit\'e\\
Institut de Math\'emathiques de Jussieu-Paris rive gauche, UMR 7586\\
CNRS\\
Sorbonne Universit\'es, UPMC Univ. Paris 06\\
F-75013, Paris, France} 
\email{hakan.eliasson@imj-prg.fr}

 \author{ Beno\^it Gr\'ebert}
\address{Laboratoire de Math\'ematiques Jean Leray, Universit\'e de Nantes, UMR CNRS 6629\\
2, rue de la Houssini\`ere \\
44322 Nantes Cedex 03, France}
\email{benoit.grebert@univ-nantes.fr}

\author{ Serge\"i B. Kuksin }
\address{CNRS\\
Institut de Math\'emathiques de Jussieu-Paris rive gauche, UMR 7586\\
Univ. Paris Diderot, Sorbonne Paris Cit\'e\\
Sorbonne Universit\'es, UPMC Univ. Paris 06\\
F-75013, Paris, France}
\email{sergei.kuksin@imj-prg.fr}

\title[KAM for the non-linear Beam equation 2]
{KAM for the non-linear Beam equation 2:\\
A normal form theorem}

\begin{document}

\begin{abstract}
We prove an abstract KAM theorem adapted to 
space-multidimensional hamiltonian 
 PDEs with regularizing nonlinearities. It applies
 in particular to the singular perturbation problem
 studied in the first part of this work.
 
  \begin{center} {\bf \large 8/2/ 2015}\end{center}
 
 \end{abstract}

\subjclass{ }
\thanks{
}

\maketitle

\tableofcontents

\section{Introduction}

\subsubsection{The phase space}\label{ssThePhaseSpace}

Let $\cA$ and $\F$ be two finite sets in $\Z^{d_*}$ and let $\L_\infty$ be an infinite subset of $\Z^{d_*}$. 
Let $\L$ be the disjoint union $\cA\sqcup \F\sqcup \L_{\infty}$ and consider $(\C^2)^{\L}$.

 For any subset $X$ of $\L$, consider the projection
$$\pi_X:(\C^2)^{\L}\to (\C^2)^{X}=\{\zeta\in (\C^2)^{\L}: \zeta_a=0\ \forall a\notin X\}.$$
We can thus write $(\C^2)^{\L}=(\C^2)^{X}\times (\C^2)^{\L\setminus X}$,
$\zeta=(\zeta_X,\zeta_{\L\setminus X})$,
and when $X$ is finite this gives an injection
$$\iota_X:(\C^2)^{\#X}\hookrightarrow (\C^2)^{\L}$$
whose image is $ (\C^2)^{X}$.

\medskip

Let $\ga=(\ga_1,\ga_2)\in\R^2$ and  let
$Y_\ga$  be the space of sequences $\zeta\in (\C^2)^{\L}  $ such that
$$
||\zeta||_{\ga}=  \sqrt{\sum_{a\in\L} |\zeta_a|^2e^{2\ga_1|a|}\langle a\rangle^{2 \ga_2}}<\infty$$
-- here $\langle a\rangle= \max (|a|,1)$ and $|\cdot |$ is the standard Hermitian norm on $\C^n$ associated with 
the standard scalar product $\langle \cdot,\cdot \rangle_{\C^n}$.

\medskip

Write $\zeta_a=(p_a,q_a)$ and let
$$\Om(\zeta,\zeta')= \sum_{a\in\L} p_a q'_a-q_a p'_a.$$
$\Om$ is  an anti-symmetric bi-linear form which is  continuous on
$$Y_\ga\times Y_{-\ga}\cup Y_{-\ga}\times Y_{\ga}\to \C$$ 
with norm $\aa{\Om}= 1$.  The subspaces $(\C^2)^{\{a\}}$  are symplectic subspaces of two (complex) dimensions carrying the canonical symplectic structure.

$\Om$  defines as usual (by contraction on the first factor) a bounded bijective operator
$$Y_\ga\ni \zeta\mapsto \Om( \zeta,\cdot) \in Y^*_{-\ga}.
\footnote{\ $Y^*_{\ga}$ denote the Banach space dual of $Y_{\ga}$}
$$
We shall denote its inverse by
$$J:   Y^*_{-\ga}\to Y_\ga.$$ 

\begin{NB*} 
There is another common way to identify $ Y^*_{-\ga}$ with $Y_\ga$, the $L^2$-pairing. This pairing defines
an isomorphism $\nabla:   Y^*_{-\ga}\to Y_\ga$ such that
$$J\circ\nabla^{-1}\zeta=\{ \left(\begin{array}{cc} 0& -1\\ 1& 0\end{array}\right)\zeta_a: a\in\L\}.$$
The operator $J\circ\nabla^{-1}$ is a complex structure compatible with $\Om$ which is customarily denoted by $J$, and we shall follow this tradition. This abuse of notation
will cause no confusion since the two $J$'s  act on different objects: one acts on one-forms and the other on vectors, and which is the case will be clear from the context.

\end{NB*}

A bounded map $A:Y_\ga\to Y_\ga$, $\ga\ge(0,0)$,
\footnote{\ $(\ga_1',\ga_2')\le (\ga_1,\ga_2)$ if, and only if $\ga_1'\le\ga_1$ and $\ga_1'\le \ga_2'$}
is {\it symplectic} if, and only if,
it extends to a bounded map $A:Y_{-\ga}\to Y_{-\ga}$
and verifies
$$\Om(A\zeta,A\zeta')=\Om(\zeta,\zeta'),\quad \zeta\in Y_{\ga},\zeta'\in Y_{-\ga},$$
or, equivalently, $A^*\circ J^{-1}\circ A=J^{-1}$ on  $Y_{\ga}$ and on $Y_{-\ga}$. If
$A$ is bijective, then it is symplectic if, and only if, 
$A^*\circ J^{-1}\circ A=J^{-1}$ on  $Y_{\ga}$ (see \cite{K00}).

\medskip
Let 
$$\A^{\cA} =\C^{\cA}\times(\C/2\pi \Z)^{\cA}$$
and consider the Banach manifold
$\A^{\cA} \times \pi_{\L\setminus \cA} Y_\ga$ whose elements are denoted $x=(r,\theta= [z],w)$. 
\footnote{\ $[z]$ being the class of $z\in \C^{\cA}$}

We provide 
this manifold with the metric 
$$\aa{x-x'}_\ga=
\inf_{p\in\Z^{d_*}}  ||(r, z+2\pi p, w)-(r', z',w')||_\ga.$$

We provide $\A^{\cA} \times \pi_{\L\setminus \cA} Y_\ga$ with the  symplectic structure $\Om$. To any 
$C^{1}$-function  $f(r,\theta,w)$ on (some open set in)  $\A^{\cA}\times \pi_{\L\setminus \cA} Y_{\ga}$ it associates
a  vector field $X_f=-J(df)$  --  the Hamiltonian vector field of $f$
\footnote{\ there is no agreement as to the sign of the Hamiltonian vectorfield -  we've used the choice of Arnold 
\cite{Arn}}
-- which in the coordinates $(r,\theta,w)$ takes the form
$$
\left(\begin{array}{c} \dot r_a \\ \dot \theta_a\end{array}\right)=J
\left(\begin{array}{c} \frac{\p}{\p r_a} f(r,\theta,w)\\
\frac{\p}{\p \theta_a} f(r,\theta,w)\end{array}\right)
\qquad
\left(\begin{array}{c} \dot p_a \\ \dot q_a\end{array}\right)=J
\left(\begin{array}{c} \frac{\p}{\p p_a} f(r,\theta,w)\\
 \frac{\p}{\p q_a} f(r,\theta,w)\end{array}\right).
$$

\subsubsection{An integrable Hamiltonian system in $\infty$ many dimensions}

In this paper we are considering an infinite dimensional Hamiltonian system given by a function
$h(r,w,\r)$ of the form
\be\label{equation1.1}  \langle r,\om(\r)\rangle +\frac12\langle w,A(\r)w \rangle=
\langle r,\om(\r)\rangle +\frac12\langle w_{\F},H(\r)w_{ \F} \rangle
+\frac12\sum_{a\in \L_{\infty}} \la(p_a^2+q_a^2),\ee
where $w_a=(p_a,q_a)$ and
\be\label{properties}\left\{\begin{array}{ll}
\om:\D\to\R^{\cA}&\\
\la:\D\to \R,&\quad  a\in \L_\infty\\
H:\D\to gl(\R^{\F}\times \R^{\F}),&\quad {}^t\! H=H
\end{array}\right.\ee
are $\cC^{{s_*}}$, ${s_*}\ge 1$,   functions of $\r\in \D$, the unit ball in $\R^{\P}$, parametrized by  some finite subset $\P$
of $\Z^{d_*}$. 

The Hamiltonian vector field of  $h$ is not $\cC^{1}$ on $\A^{\cA}\times  \pi_{\L\setminus \cA} Y_{\ga}$, but its Hamiltonian system still has a well defined flow with a {\it finite-dimensional invariant torus}
$$\{0\}\times\T^{\cA}\times\{0,0\}$$
which is {\it reducible}, i.e. the linearized equation on this torus (is conjugated to a system that) does not depend on 
the angles $\theta$. This linearized equation
has infinitely many elliptic directions with purely imaginary eigenvalues
$$\{{\mathbf i}\la(\r) : a\in \L_\infty\}$$
and finitely many other directions given by the system
$$
\dot \zeta_\F= JH(\r)\zeta_\F.$$

\subsubsection{A perturbation problem}

The question here is if this invariant torus for $h$ persists under perturbations $h+f$, and, if so, if the persisted torus
is reducible.

In finite dimension the answer is yes under very general conditions  --  for the first proof in the purely elliptic case see \cite{E88}, and for a more general case see \cite{Y99}. These statements say that, under general conditions, the invariant torus persists and remains reducible under sufficiently small perturbations for a subset of parameters $\r$ of large Lebesgue measure. Since the unperturbed problem is linear, parameter selection can not be avoided here.

In infinite dimension the situation is more delicate, and results can only be proven under quite severe restrictions
on the normal frequencies (i.e. the eigenvalues ${\mathbf i}\la$). Such restrictions
are fulfilled for many PDE's in one space dimension --  the first such result was obtained in  \cite{K87}. 

For PDE's in higher space dimension the behavior of the normal frequencies is much more complicated and the results
are more sparse. A result for the Beam equation (which is simpler model than the Schr\"odinger equation and the Wave equation, and frequently considered in works on nonlinear PDE's)
was first obtained in \cite{GY06a} and  \cite{GY06b}. For other results on PDE's in higher space dimension see the discussion in the first part of this work \cite{EGK}.

\subsubsection{Conditions on the unperturbed Hamiltonian}

The function $h$ we shall consider will verify several assumptions.
\begin{itemize}
\item[A1] {\it -- spectral asymptotics.} There exist  constants  
$0< c',c\le 1$  and exponents $\beta_1= 2$, $\beta_2\ge 0,  \beta_3>0$  such that for all $\r\in\D$:

\be\label{la-lb-ter} 
|\la(\r)-\ab{a}^{\beta_1} |\leq c \frac1{\langle a\rangle^{\beta_2}}\quad a\in \L_{\infty};
\ee
\begin{multline}\label{la-lb}
|(\la(\r )-\lb(\r))-(\ab{a}^{\beta_1}-\ab{b}^{\beta_1}) |\leq \\
\le c'c\max( \frac1{\langle a\rangle^{\beta_3}},\frac1{\langle b\rangle^{\beta_3}}),
\quad a,b\in \L_{\infty}\,;
\end{multline}
\be\label{laequiv}
\left\{\begin{array}{l}
 \la(\r )\geq  c' \ \quad a\in\L_\infty\\
||(JH(\rho))^{-1}||\leq \frac1{c'};
\end{array}\right.
\ee
\be\label{la-lb-bis}
\left\{\begin{array}{ll}
|(\la(\r )-\lb(\r))) |\ge c' & a,b\in\L_\infty,\ \ab{a}\not=\ab{b}\\
||(\la(\r)I-{\mathbf i}JH(\rho))^{-1}||\leq \frac1{c'} & a\in \L_{\infty}.
\end{array}\right.
\ee

\item[A2] {\it  -- transversality.} We refer to section \ref{ssUnperturbed} for the precise formulation of this
condition which describes how the eigenvalues vary with the parameter $\r$.

\end{itemize}

\subsubsection{Conditions on the perturbation}
For $\sigma>0$ we let $\O_\ga(\s,\mu)$ be the set
$$\{x=(r ,\theta,w)\in \A^{\cA} \times\pi_{\L\setminus \cA}  Y_{\ga}: 
\aa{(\frac r\mu,{\mathbf  i}\frac{\Im\theta}\s,\frac w\mu)-0}_\ga<1\}.$$
It is often useful to scale the action variables by $\mu^2$ and not by $\mu$, but in our case  $\mu$ will
be $\approx 1$, and then there is no difference.

We shall consider perturbations 
$$f:\O_{\ga_*}(\s, \mu)\to \C,\quad \ga_*=(0,m_*)\ge(0,0),$$
that are  real holomorphic up to the boundary (rhb).
This means that it gives real values to real arguments and extends holomorphically
to a neighborhood of the closure of  $\O_{\ga_*}(\s, \mu)$. 
$f$ is clearly also rhb on $\O_{\ga'}(\s, \mu)$ for any $\ga'\ge\ga_*$, and
$$Jd f:\O_{\ga'}(\s, \mu)\to Y_{-\ga'} $$
is rhb. But we shall require more:

\begin{itemize}

\item[R1] {\it  -- first differential} 
$$Jd f:\O_{\ga'}(\s, \mu)\to Y_{\ga'} $$
is rhb for any $\ga_*\le \ga'\le\ga$.

\end{itemize}
This is a natural smoothness condition on the space of holomorphic functions on $\O_{\ga_*}(\s, \mu)$, and it  implies, in particular, that   $Jd^2 f(x)\in\cB(Y_{\ga};Y_{\ga})$ for any $x\in \O_{\ga}(\s, \mu)$.  

That $Jd^2 f(x)\in\cB(Y_{\ga};Y_{\ga})$ implies in turn that
$$
\ab{Jd^2f(x)[e_a,e_b]}\le\Cte e^{-\ga_1\ab{ \ab{a}-\ab{b}}}\min(\frac{\langle a\rangle}{\langle b\rangle},
\frac{\langle b\rangle}{\langle a\rangle})^{\ga_2}$$
for any two  unit vectors $e_a\in(\C^2)^{\{a\}}$ and $e_b\in(\C^2)^{\{b\}}$. 
But many Hamiltonian PDE's verify other, and  stronger, 
decay conditions  in terms of
$$\min(\ab{a-b},\ab{a+b}).$$
Such decay conditions do not seem to be  naturally related to any smoothness condition of $f$, but they may be instrumental in the KAM-theory for multidimensional PDE's:  see for example \cite{EK10} where such conditions were used to build a KAM-theory for some multidimensional non-linear Schr\"odinger equations.

The  decay condition needed in this work depends on a parameter $0\le\vark \le m_*$  and defines
a Banach sub-algebra $ \cM_{\ga,\vark}^b$ of $\cB(Y_{\ga};Y_{\ga})$, with norm 
$\aa{\cdot}_{\ga,\vark}$  --   its precise definition will be given in section \ref{ssMatrixAlgebra}. We shall require:

\begin{itemize}

\item[R2] {\it  -- second differential}
$$Jd^2f:\O_{\ga'}(\s, \mu)\to  \cM_{\ga',\vark}^b$$
is rhb for any $\ga_*\le \ga'\le\ga$s.

\end{itemize}

Denote by $\cT_{\ga,\vark}(\s,\mu) $ the space of functions
$$f:\O_{\ga_*}(\s, \mu)\to \C,$$
real holomorphic  up to the boundary, verifying  R1 and R2.
We provide  $\cT_{\ga,\vark} (\s, \mu)$ with the norm 
\be\label{norm}
|f|_{\begin{subarray}{c}\s,\mu\\ \ga, \vark  \end{subarray}}=
\max\left\{\begin{array}{l}
\sup_{x\in \O_{\ga_*}(\s, \mu)}|f(x)|\\ 
\sup_{\ga_*\le \ga'\le\ga}\sup_{x\in \O_{\ga'}(\s, \mu)} || Jd f(x)||_{\ga'}\\
\sup_{\ga_*\le  \ga'\le\ga}\sup_{x\in \O_{\ga'}(\s, \mu)}||Jd^2f(x)||_{\ga',\vark}
\end{array}\right.
\ee
making it  into a Banach space.  Notice that the first two ``components'' of this norm are related to the smoothness
of $f$, while the third ``component'' imposes a further decay condition on $Jd^2f(x)$.

\subsubsection{The normal form theorem}
For any $a\in\L_\infty$, let
$$
[a]=\{b\in\L_\infty: |b|=|a|\}.$$

\begin{theorem*}
Let  $h$ be a Hamiltonian defined by \eqref{equation1.1} and verifying Assumptions A1-2.

Let $f:\O_{\ga_*}(\s,\mu)\to \C$ be real holomorphic and verifying Assumptions R1-2 with
$$
\ga=(\ga_1,m_*)>\ga_*=(0,m_*)\quad\textrm{and}\quad 0<\vark\le m_*.$$
 
If $\eps=|f|_{\begin{subarray}{c}\s,\mu\\ \ga, \vark \end{subarray}}$ is sufficiently small, then there is a  set 
$\D'\subset \D$ with 
$$\Leb(\D\setminus \D')\to 0,\quad \eps\to 0$$
and a $\cC^{{s_*}}$ mapping
$$\Phi:\O_{\ga_*}(\s/2,\mu/2)\times\D\to \O_{\ga_*}(\s,\mu),$$
real holomorphic and symplectic for each parameter $\r\in\D$,
such that 
$$
(h+ f)\circ \Phi= h'+f'\in \cT_{\ga_*,\vark} ,$$
and
\begin{itemize}
\item[(i)] for $\r\in\D'$ and $\zeta=r=0$
$$d_r f'=d_\theta f'=
d_{\zeta} f'=d^2_{\zeta} f'=0;$$
\item[(ii)] $h':\O_{\ga_*}(\s/2,\mu/2)\times\D\to\C$ is a $\cC^{{s_*}}$-function, real holomorphic 
for each parameter $\r\in\D$, and
$$[\p_\r^j(h'(\cdot,\r)-h(\cdot,\r))]_{\begin{subarray}{c}\s,\mu\\ \ga_* , \vark  \end{subarray}}< C\eps, \quad |j|\le {s_*}-1;$$
\item[(iii)] $h'$ has the form
\begin{multline*}
 \langle r,\om'(\r)\rangle +\frac12\langle \zeta_{\F},H'(\r)\zeta_{ \F} \rangle+\\
+\frac12\sum_{[a]} \big(\langle p_{[a]},A'_{[a]}(\r)p_{[ a]} \rangle+ \langle q_{[a]},A'_{[a]}(\r)q_{[ a]} \rangle\big)
\end{multline*}
where the matrix $A'_{[a]}$ is symmetric;

\item[(iv)] for any $x\in \O_{\ga_*}(\s/2,\mu/2)$
$$||\p_\r^j(\Phi(x,\r)-x  )||_{\ga_*}\le C\eps,\quad |j|\le {s_*}-1.$$
\end{itemize}

The constant $C$ depends on $\#\cA,\#\F, \#\P,d_*, {{s_*}},m_*,\vark$ and $h$, 
but not on $\eps$.
\end{theorem*}

We shall give a more precise formulation of this result in Theorem \ref{main} and its Corollary \ref{cMain}.

\subsubsection{A singular perturbation problem}

We want to apply this theorem to construct  small-amplitude solutions of 
 the multi-dimensional beam equation on the torus:
$$ u_{tt}+\Delta^2 u+m  u =   -  g(x,u)\,,\quad u=u(t,x), \quad  \ x\in \T^{d_*}.$$
Here $g$ is a real analytic function  satisfying 
 $$ g(x,u)=4u^3+ O(u^4).$$
Writing it in Fourier components, and introducing action-angle variables for the modes in (an arbitrary finite subset)
$\cA$, the linear part becomes a Hamiltonian system with a Hamiltonian $h$ of the form \eqref{equation1.1}, with $\F$ void.
$h$ satisfies (for all $m>0$) condition A1, but not condition A2. 

The way to improve on $h$ is to use a (partial) Birkhoff normal form around $u=0$ in order to extract a piece from the non-linear part which improves on $h$. This leads to a situation where the assumptions A1 and A2 and the size of the perturbation are linked -- a singular perturbation problem.

In order to apply the theorem to such a singular situation one needs a careful and precise description of how the smallness
requirement depends on the (parameters determining) assumptions A1 and A2. This is quite a serious complication which is carried out in this paper and the precise description of the smallness requirement  is given in Theorem \ref{main}.

This normal form theorem improves on the result in \cite{GY06a} and \cite{GY06b} in two respects. 
\begin{itemize}
\item
We have imposed no ``conservation of momentum''  on the perturbation --  this  has the effect that our normal form is not diagonal in the purely elliptic directions. In this respect it resembles the normal form obtained for the non-linear 
Schr\"odinger equation obtained in \cite{EK10} and the block diagonal form is the same.
\item
We have a finite-dimensional, possibly hyperbolic, component, whose treatment requires higher smoothness in the parameters.
\end{itemize}

The proof has no real surprises. It is a classical KAM-theorem carried out in a complex situation. The main part is,
as usual, the solution of the homological equation with reasonable estimates. The fact that the block structure is not
diagonal complicates, but this was also studied in for example \cite{EK10}. The iteration combines a finite linear iteration
with a ``super-quadratic'' infinite iteration. This has become quite common in KAM and was also used for example in
\cite{EK10}.


  \subsubsection{Notation and agreements.} 
 $\langle a\rangle =\max(|a|,1)$.
${\mathbf i}$ will denote the complex imaginary unit. ${\overline z}$ is the complex-conjugate of $z\in\C$. By $\langle\zeta,\zeta'\rangle_{\C^n}$ we denote the standard Hermitian scalar product in $\C^n$,
conjugate-linear in the first variable and linear in the second variable. 


All Euclidean spaces, unless otherwise stated, are provided with the Euclidean norm denoted by $|\cdot|$. For two
subsets $X$ and $Y$ of a Euclidean space we denote
$$\dist(X,Y)=\inf_{x\in X,\ y\in Y}|x-y|$$
and
$$\diam(X)=\sup_{x,y\in X}|x-y|.$$


The space of bounded  linear operators between two Banach spaces $X$  and $Y$ is denoted  $\cB(X;Y)$. Its operator norm will usually be denoted $\|\cdot\|$ without specification of the spaces. Our complex Banach spaces will be complexifications of some ``natural'' real Banach spaces which in general are implicit.
An analytic function between domains of two complex Banach spaces is called real holomorphic if it gives real values
to real arguments.

The  sets $\cA, \F,\L_\infty, \P$, as well as  the ``starred'' constants $d_*, {s_*},m_*$ will be fixed in this paper  --  the dependence of them is usually not indicated. 

Constants depending only on  the dimensions $\#\cA, \#\F,\#\P$, on  ${d_*},{s_*},m_*$
and on the choice of finite-dimensional norms
are regarded as {\it absolute constants}. An absolute constant only depending on $\beta$ is thus a constant that only depends on  $\beta$, besides these factor.
Arbitrary constants will often be denoted by $\Cte$, $\cte$ and, when they occur as an exponent,  $\exp$. Their values may change from line to line. For example we allow ourselves to write $2\Cte\le \Cte$.

\subsubsection{Acknowledgement} The authors acknowledge the support from the project ANR-10-BLAN 0102
of the Agence Nationale de la Recherche.

\section{Preliminaries.}

\subsection{A matrix algebra}\label{ssMatrixAlgebra}
\ 
The mapping
\be\label{pdist}
(a,b)\mapsto [a-b]=\min (|a-b|,|a+b|)\ee
is a pseudo-metric on $\Z^{d_*}$, i.e. verifying all the relations of a metric with the only exception that
$[a-b]$ is $=0$ for some $a\not=b$.
This is most easily seen by observing that $[a-b]=\textrm{d}_{\Haus}(\{\pm a\},\{\pm b\})$.
We have $ [a-0]=|a|$.

Define, for any $\ga=(\ga_1,\ga_2)\ge(0,0)$ and $\vark\ge0$,
$$e_{\ga,\vark}(a,b)=Ce^{\ga_1[a-b]}\max([a-b],1)^{\ga_2}\min(\langle a\rangle,\langle b\rangle)^\vark.$$

\begin{lemma}\label{lWeights}
\ 
\begin{itemize}
\item[(i)] If $\ga_1, \ga_2- \vark\ge0$, then 
$$e_{\ga, \vark}(a,b)\le   e_{\ga,0}(a,c) e_{\ga,\vark}(c,b),\quad \forall a,b,c,$$
if $C$ is sufficiently large (bounded with  $\ga_2,\vark$).

\item[(ii )] If $-\ga\le\tilde \ga\le \ga$, then
$$e_{\tilde\ga, \vark}(a,0)\le   e_{\ga,\vark}(a,b) e_{\tilde\ga,\vark}(b,0),\quad \forall a,b$$
if $C$ is sufficiently large  (bounded with  $\ga_2,\vark$).
\end{itemize}
\end{lemma}

\begin{proof} 
(i). Since $[a-b]\le [a-c]+[c-b]$ it is sufficient to prove this for $\ga_1=0$. If $\ga_2=0$ then the statement holds for any
$C\ge1$, so it is sufficient to consider $\ga_2>0$ and, hence $\ga_2=1$. Then we want to prove
\begin{multline*}
\max([a-b],1)\min(\langle a\rangle,\langle b\rangle)^\vark\\ \le C \max([a-c],1)\max([c-b],1)
\min(\langle c \rangle,\langle b\rangle)^\vark.\end{multline*}
Now $\max([a-b],1)\le \max([a-c],1)+\max([c-b],1)$,
$$
\max([c-b],1)
\min(\langle c \rangle,\langle b\rangle)^\vark\gsim  \langle b\rangle^\vark,
$$
and
$$
\max([a-c],1)
\min(\langle c \rangle,\langle b\rangle)^\vark\gsim  \min(\langle a \rangle,\langle b\rangle)^\vark.$$
This gives the estimate.

(ii) Again it suffices to prove this for $\ga_1=0$ and $\ga_2=1$. Then we want to prove
$$
\max(\ab{a},1)^{\tilde \ga_2}  \le C \max([a-b],1)\min(\langle a \rangle,\langle b\rangle)^\vark \max(\ab{b},1)^{\tilde \ga_2}.$$
The inequality is fulfilled with $C\ge1$ if $a$ or $b$ equal $0$. Hence we need to prove
$$
\ab{a}^{\tilde \ga_2}  \le C \max([a-b],1)\min(\langle a \rangle,\langle b\rangle)^\vark \ab{b}^{\tilde \ga_2}.$$
Suppose $\tilde\ga_2\ge0$. If $\ab{a}\le 2 \ab{b}$ then this holds for any  $C\ge2$. If $\ab{a}\ge2\ab{b}$ then
$[a-b]\ge \frac12\ab{a}$ and the statement holds again  for any  $C\ge2$.

If instead $\tilde\ga_2<0$, then we get the same result with $a$ and $b$ interchanged.

\end{proof} 

\subsubsection {The space $\cM_{\ga,\vark}$}\label{sM2}

We shall consider matrices $A:\L\times\L\to gl(2,\C)$, formed by $2\times2$-blocs,
(each $A_a^{b}$ is a $2\times2$-matrix). Define
$$|A|_{\ga,\vark}=\max\left\{\begin{array}{l}
\sup_a\sum_{b} \ab{A_a^b} e_{\ga,\vark}(a,b)\\
\sup_b\sum_{a} \ab{A_a^b} e_{\ga,\vark}(a,b),
\end{array}\right.$$
where the norm on $A_a^{b}$ is the matrix operator norm.

Let $\cM_{\ga,\vark}$ denote the space of all matrices $A$ such that
$\ab{A}_{\ga,\vark}<\infty$.  Clearly $\ab{\cdot}_{\ga,\vark}$ is a norm
on $\cM_{\ga,\vark}$. It follows by well-known results that $\cM_{\ga,\vark}$,
provided with this norm, is a Banach space.

Transposition --  $({}^tA)_a^b={}^t\!A_b^a$ --  and $\C$-conjugation --
$(\overline{A})_a^b={\overline{A_a^b}})$  -- do not change 
this norm.The identity matrix is in $\cM_{\ga,\vark}$ if, and only if, $\vark=0$, and then $|I|_{\ga,0}=C$.

\subsubsection{Matrix multiplication}

We define (formally) the {\it matrix product}
$$(AB)_a^b=\sum_{c} A_a^cB_c^b.$$
Notice that complex conjugation, transposition and taking the adjoint behave in the usual
way under this formal matrix product.

\begin{proposition}\label{pMatrixProduct} Let $\ga_2\ge\vark$. 
If $A\in \cM_{\ga,0}$ and $B\in \cM_{\ga,\vark}$,   then
$AB$ and $BA\in \cM_{\ga,\vark}$ and
$$
\ab{{AB} }_{\ga,\vark}\ \textrm{and}\ 
\ab{{BA} }_{\ga,\vark}\le \ab{A}_{\ga,0} \ab{B}_{\ga,\vark}.$$ 

\end{proposition}

\begin{proof}
(i) We have, by Lemma \ref{lWeights}(i),
$$\sum_{b} \ab{( AB)_a^b} e_{\ga,\vark}(a,b)\le
\sum_{b,c} \ab{ A_a^c} \ab{ B_c^b} e_{\ga,\vark}(a,b)\le $$
$$\le\sum_{b,c}\ab{ A_a^c} \ab{ B_c^b} 
e_{\ga,0}(a,c)e_{\ga,\vark}(c,b)  $$
which is $\le \ab{A}_{\ga,0} \ab{B}_{\ga,\vark}$.
This implies in particular the existence of $(AB)_a^b$.

The sum over $a$ is shown to be  $\le \ab{A}_{\ga,0} \ab{B}_{\ga,\vark}$ in a similar way.
The estimate of $BA$ is the same.
\end{proof}

Hence $\cM_{\ga,0}$ is a Banach algebra, and $\cM_{\ga,\vark}$ is a closed ideal in $\cM_{\ga,0}$.
 
\subsubsection{The space $\cM_{\ga,\vark}^b$}
We define (formally) on $Y_\ga$ (see section \ref{ssThePhaseSpace})
$$(A\zeta )_a=\sum_{b} A_a^b \zeta _b.$$

\begin{proposition}\label{pMatrixBddOp}
Let $-\ga\le\tilde \ga\le\ga$.
If $A\in\cM_{\ga,\vark}$ and 
$\zeta \in Y_{\tilde \ga}$,  then $A\zeta \in Y _{\tilde \ga}$ and
$$
\aa{A\zeta }_{\tilde \ga}\le \ab{A}_{\ga,\vark}\aa{\zeta }_{\tilde\ga}.$$
\end{proposition}

\begin{proof}
Let $\zeta '=A\zeta $. We have
$$
\sum_{a} \ab{\zeta '_a}^2e_{\tilde \ga,0}(a,0)^2\le\sum_{a} 
\big(\sum_{b} \ab{A_a^b}\ab{\zeta _{b}} e_{\tilde \ga,0 }(a,0)\big)^2.$$
Write
$$\ab{A_a^b}\ab{\zeta _{b}}e_{\tilde \ga,0 }(a,0)
= I\times (I\ab{ \zeta _{b}}e_{\tilde \ga,0}(b,0))\times  J$$
where
$$I=I_{a,b}=\sqrt{\ab{A_a^b}    e_{\ga,\vark}(a,b)}.$$
Since,  by Lemma \ref{lWeights}(ii),
$$J=
\frac{e_{\tilde \ga,0}(a,0)}{e_{\ga,\vark}(a,b)e_{\tilde\ga,0}(b,0)}\le 1,$$
we get,  by H\"older,
\begin{multline*}
\sum_{a} \ab{\zeta' _a}^2e_{\tilde \ga,0}(a,0)^2\le  
\sum_{a} (\sum_{b} I_{a,b}^2)( \sum_{b} I_{a,b}^2\ab{ \zeta _{b}}^2e_{\tilde\ga,0}(b,0)^2)\\
\le \ab{A}_{\ga,\vark}    \sum_{a,b} I_{a,b}^2\ab{\zeta _{b}}^2e_{\tilde\ga,0}(b,0)^2
\le  \ab{A}_{\ga,\vark}    \sum_{b} \ab{ \zeta _{b}}^2e_{\tilde \ga,0}(b,0)^2\sum_{a}I_{a,b}^2\le
\end{multline*}
$$\le \ab{A}_{\ga,\vark}^2      \aa{\zeta }_{\tilde \ga}^2. $$
This shows that $y_a$ exists  for all $a$, and it also proves the estimate. 
\end{proof}

We have thus, for any $-\ga\le \tilde\ga\le\ga$, a continuous embedding of $\cM_{\ga,\vark}$,
$$\cM_{\ga,\vark}\hookrightarrow \cM_{\ga,0}\to
 \cB(Y_{\tilde \ga};Y_{\tilde \ga}),$$
into the space of bounded linear operators on $Y_{\tilde \ga}$. Matrix multiplication in $\cM_{\ga,\vark}$
corresponds to composition of operators.

\smallskip

For our applications we must consider a larger sub algebra with somewhat weaker decay properties.
For $\ga=(\ga_1,\ga_2)\ge \ga_*=(0,m_*)$, let
$$\cM_{\ga,\vark}^b= \cB(Y_{\ga};Y_{\ga})\cap \cM_{(\ga_1,\ga_2+\vark-m_*),\vark}$$
which we provide with the norm
$$
\aa{A}_{\ga,\vark }=\aa{A}_{ \cB(Y_{\ga};Y_{\ga})}+ \ab{A}_{(\ga_1,\ga_2+\vark-m_*),\vark }.$$
This norm makes $\cM_{\ga,0}^b$ into a Banach sub-algebra of $\cB(Y_{\ga};Y_{\ga})$
and $\cM_{\ga,\vark}^b$ becomes a closed ideal in $\cM_{\ga,0}^b$.

\subsection{Functions}
Let
$$0\le  \vark\le m_*$$
and let
$$\ga=(\ga_1,m_*)\ge \ga_*=(0,m_*).$$
 
\subsubsection{The function space  $\cT_{\ga,\vark} $ }
Consider the space of  functions
$f:\O_{\ga_*}(\s, \mu)\to \C$
which are  real holomorphic   up to the boundary (rhb) 
of $\O_{\ga_*}(\s, \mu)$. This implies that for any $\ga\ge\ga_*$
$$f:\O_{\ga}(\s, \mu)\to \C$$
$$Jd f:\O_\ga(\s, \mu)\to Y_{-\ga}$$
and
$$Jd^2f:\O_\ga(\s, \mu)\to \cB(Y_\ga,Y_{-\ga}) $$
are also rhb. We define  $\cT_{\ga,\vark}(\s,\mu) $ to be the space of such functions such that
 for any $\ga_*\le \ga'\le\ga$,
$$\begin{array}{lll}
R_1  & -  &  Jd f:\O_{\ga'}(\s, \mu)\to Y_{\ga'}  \\
R_2  & -  &  Jd^2f:\O_{\ga'}(\s, \mu)\to  \cM_{\ga',\vark}^b
\end{array}$$
are rhb. 

We provide  $\cT_{\ga,\vark} (\s, \mu)$ with the norm \eqref{norm}.  The higher differentials $d^{k+2}f$  can be estimated by Cauchy estimates on some smaller domain in terms of this norm.
This norm makes $\cT_{\ga,\vark}(\s, \mu) $ into a Banach space and a Banach algebra with the constant function
$f=1$ as unit.

\begin{remark*}
The differential forms  $d^{k+2}f(x)$, $x\in \O_{\ga'}(\s, \mu)$, are canonically identified with three bounded 
linear maps
$$
\begin{array}{lll}
R_0  & -  & \bigotimes_{k+2} Y_{\ga'}=\underbrace{Y_{\ga'}\otimes\dots\otimes Y_{\ga'}}_{k+2\ times}\longrightarrow \C\\
R_1  & -  & \bigotimes_{k+1} Y_{\ga'}\longrightarrow  Y_{-\ga'}^*\overset{J}{\longrightarrow} Y_{\ga'}\\
R_2  & -  &\bigotimes_{k} Y_{\ga'} \longrightarrow J^{-1}\cM_{\ga',\vark}^b
\overset{J}{\longrightarrow} \cM_{\ga',\vark}^b,
\end{array}$$
where $\otimes$ is the symmetric tensor product (over $\C$).

\end{remark*}

 \subsubsection{The function space  $\cT_{\ga,\vark,\D} $ }
  
Let $\D$ be the unit ball in $\R^{\P}$.  We shall consider functions
$$f:\O_{\ga^*}(\s, \mu)\times \D\to \C$$
which are of class $\cC^{{s_*}}$.  We say that $f\in \cT_{\ga,\vark,\D}(\s,\mu)  $ if, and only if,  
$$\frac{\p^j f}{\p \r^j}(\cdot,\r)\in \cT_{\ga,\vark}(\s,\mu) $$
for any $\r\in\D$ and any $\ab{j}\le {s_*}$. We provide this space by the norm
$$
|f|_{\begin{subarray}{c}\s,\mu\ \ \\ \ga, \vark,\D  \end{subarray}}=
\max_{\ab{j}\le {s_*}}\sup_{\r\in\D}
|\frac{\p^j f}{\p \r^j}(\cdot,\r)|_{\begin{subarray}{c}\s,\mu\\ \ga, \vark  \end{subarray}}$$
This norm makes $\cT_{\ga,\vark,\D}(\s,\mu)  $ into a Banach space and a Banach algebra.

\subsubsection{ Jets of functions.} \label{ss5.1}

For any function $f\in \Tc_{\ga,\vark,\D}(\s,\mu)$ we define its jet $f^T(x)$, $x=(r,\theta,w)$, as the following 
 Taylor polynomial of $f$ at $r=0$ and $w=0$
\be
\label{jet}
f(0,\theta,0)+d_r f(0,\theta,0)[r] +d_w f(0,\theta,0)[w]+\frac 1 2 d^2_wf(0,\theta,0)[w,w] 
\ee
Functions of the form  $f^T$ will be called {\it jet-functions.}

\begin{proposition}\label{lemma:jet}
Let $f\in \Tc_{\ga,\vark,\D}(\s,\mu)$. Then
$f^T\in \Tc_{\ga,\vark,\D}(\s,\mu)$ and 
$$
|f^T|_{\begin{subarray}{c}\s,\mu \  \\ \ga, \vark, \D \end{subarray}}
\leq C |f|_{\begin{subarray}{c}\s,\mu\ \ \\ \ga, \vark,\D  \end{subarray}}.$$

$C$ is an absolute constant. 
\end{proposition}

\begin{proof} The first part follows by general arguments. Look for example on
$$g(x)=   d^2_wf\circ p(x)[ w,w],\quad x=(r,\theta,w),$$
where $p(x)$ is the projection onto $(0,\theta,0)$.
This function $g$  is real holomorphic   up to the boundary (rhb) on $\O_{\ga_*}(\s, \mu)$, 
being a composition of such functions.  Its sup-norm is obtained by a Cauchy estimate of $f$:
$$\aa{d^2_wf(p(x))}_{ \cB(Y_{\ga_*},Y_{\ga_*};\C)}\aa{w}^2_{\ga'}
\le\Cte\frac1{\mu^2} \sup_{\O_{\ga_*}(\s, \mu)}\ab{f(y)}\aa{w}_{\ga_*}^2\le 
\Cte \sup_{y\in \O_{\ga_*}(\s, \mu)}\ab{f(y)}.$$

Since $Jd g(x)[\cdot]$ equals
$$\big(J dd^2_w f\circ p(x)[w,w]\big)[dp[\cdot]]+ 2\big(Jd^2_w f\circ p(x)[ w]\big)[\cdot], $$
and
$$Jd^2_w f:\O_{\ga'}(\s, \mu)\to \cB(Y_{\ga'};Y_{\ga'})$$
and
$$Jd d^2_w  f=J d^2_w  df:\O_{\ga'}(\s, \mu)\to \cB(Y_{\ga'},Y_{\ga'};Y_{\ga'})$$
are rhb, it follows that $d g$ verifies R1 and is rhb.  The norm $\aa{Jd g(x)}_{\ga'}$ is less than
$$\aa{J   d^2_w d f(p(x))}_{ \cB(Y_{\ga'},Y_{\ga'};Y_{\ga'})}\aa{w}^2_{\ga'}
+2\aa{Jd^2_wf(p(x))}_{ \cB(Y_{\ga'};Y_{\ga'})}\aa{w}_{\ga'},$$
which is 
$\le \Cte \sup_{y\in \O_{\ga'}(\s, \mu)}\aa{Jd f(y)}_{\ga'}$ --  this follows by  Cauchy estimates of 
derivatives of $Jd f$.

Since $Jd^2 g(x)[\cdot,\cdot]$ equals
$$\big(J d^2 d_w^2f\circ p(x)[w,w]\big)[dp[\cdot],dp[\cdot]]+ 
2J\big(  d d^2_w f\circ p(x)[w]\big)[\cdot,dp[\cdot]+ 2Jd_w^2 f\circ p(x)[\cdot,\cdot],$$
and
$$Jd^2_w f:\O_{\ga'}(\s, \mu)\to \cM_{\ga',\vark}^b,$$
$$J d d^2_w f=J d^2_w df:\O_{\ga'}(\s, \mu)\to \cB(Y_{\ga'};\cM_{\ga',\vark}^b)$$
and
$$J d^2 d_w^2f=J d_w^2 d^2f :\O_{\ga'}(\s, \mu)\to \cB(Y_{\ga'},Y_{\ga'};\cM_{\ga',\vark}^b)$$
are rhb,  it follows that $Jd g^2$ verifies R2   and is rhb. The norm $\aa{Jd^2 g}_{\ga',\vark}$ is less than
$$\aa{J d_w^2 d^2f(p(x))}_{\cB(Y_{\ga'},Y_{\ga'};\cM_{\ga',\vark}^b)}\aa{w}^2_{\ga'}
+2\aa{J d_w d^2f(p(x))}_{\cB(Y_{\ga'};\cM_{\ga',\vark}^b)}\aa{w}_{\ga'}+$$
$$
+2\aa{Jd^2f(x)}_{\ga',\vark},$$
which is 
$\le \Cte \sup_{y\in \O_{\ga'}(\s, \mu)}\aa{Jd^2 f(y)}_{\ga',\vark}$ --  this follows by a Cauchy estimate of $Jd^2 f$.

The derivatives with respect to $\r$ are treated alike. 
\end{proof}

\subsection{Flows}

\subsubsection{ Poisson brackets.}  \label{ss5.2}
The Poisson bracket $\{f,g\}$  of two $\cC^1$-functions $f$ and $g$
is (formally) defined by
$$\Om(Jdf,Jdg)=df[Jdg]=-dg[Jdf]$$
If one of the two functions verify condition R1, this product is well-defined.
Moreover, if both $f$a nd $g$ are jet-function, then $\{f,g\}$ is also a jet-function.

\begin{proposition}\label{lemma:poisson}
Let $f,g\in  \Tc_{\ga,\vark,\D}(\s,\mu)$, and let $\s'<\s$ and $\mu'<\mu\le1$. Then
\begin{itemize}
\item[(i)] $\{g,f\}\in  \Tc_{\ga,\vark,\D}(\s,\mu)$ and
$$
\ab{\{g,f\}}_{\begin{subarray}{c}\s',\mu' \  \\ \ga, \vark, \D \end{subarray}}\leq 
C_{\s-\s'}^{\mu-\mu'}\ab{g}_{\begin{subarray}{c}\s,\mu\ \  \\ \ga, \vark, \D \end{subarray}}
\ab{f}_{\begin{subarray}{c}\s,\mu \ \ \\ \ga, \vark, \D \end{subarray}}$$
for 
$$C_{\s-\s'}^{\mu-\mu'}=C \big(\frac1{(\sigma-\sigma')}  +   \frac1{ (\mu-\mu') }\big).$$

\item[(ii)] the n-fold Poisson bracket
$ P_g^n f \in  \Tc_{\ga,\vark,\D}(\s,\mu)$ and
$$\ab{ P_g^n f }_{\begin{subarray}{c}\s',\mu' \ \\ \ga, \vark, \D \end{subarray}}\leq
\big(C_{\s-\s'}^{\mu-\mu'}\ab{g}_{\begin{subarray}{c}\s,\mu \ \ \\ \ga, 0, \D \end{subarray}}\big)^n
\ab{f}_{\begin{subarray}{c}\s,\mu \ \ \\ \ga, 0, \D \end{subarray}}$$
where $P_g f=\{g,f\}$.
\end{itemize}
$C$ is an absolute constant.

\end{proposition}

\begin{proof} (i)
We must first consider the function $h=\Om(Jdg,Jdf)$ on $\O_{\ga_*}(\s, \mu)$
Since $J dg,\ J df:\O_{\ga_*}(\s, \mu)\to Y_{\ga_*}$ are real holomorphic  up to the boundary (rhb), it follows that
$h:\O_{\ga_*}(\s, \mu)\to \C$ is rhb, and 
$$\ab{h(x)}\le \aa{J dg (x)}_{\ga_*}\aa{J df(x)}_{\ga_*}.$$

The vector $J d h(x)$ is a sum of
$$J \Om(Jd^2g(x),Jdf(x))=Jd^2g(x)[Jdf(x)]$$
and another term with $g$ and $f$ interchanged.
Since $Jd^2g:\O_{\ga'}(\s, \mu)\to \cB(Y_{\ga'};Y_{\ga'})$ and 
$Jdg,\ Jdf:\O_{\ga'}(\s, \mu)\to Y_{\ga'}$ are rhb, it follows that $Jd h$
verifies R1 and is rhb. Moreover 
$$\aa{Jd^2g(x)[Jdf(x),\cdot]}_{\ga'}\le
\aa{J d^2g (x)}_{\cB(Y_{\ga'},Y_{\ga'}) }\aa{J df(x)}_{\ga'}$$
and, by definition of $ \cM_{\ga,\vark}^b$ ,
$$\aa{J d^2g (x)}_{\cB(Y_{\ga'},Y_{\ga'}) }\le \aa{J d^2g (x)}_{\ga',0}.$$

The operator $Jd^2h(x)=d(Jdh) (x)$ is a sum of 
$$Jd^3g(x)[Jdf(x)]$$
and
$$Jd^2g(x)[Jd^2f(x)]$$
and two  other terms with $g$ and $f$ interchanged.

Since $Jd^3g:\O_{\ga'}(\s, \mu)\to \cB(Y_{\ga'};\cM_{\ga',\vark}^b)$ and 
$Jdf:\O_{\ga'}(\s, \mu)\to Y_{\ga'}$ are rhb, it follows that the first function
$\O_{\ga'}(\s, \mu)\to \cB(Y_{\ga'};\cM_{\ga',\vark}^b)$ is rhb. It can be estimated
on a smaller domain using a Cauchy estimate for $Jd^3g(x)$

The second term is treated  differently. Since
$$Jd^2f,\ Jd^2g:\O_{\ga'}(\s, \mu)\to \cM_{\ga,\vark}^b$$ 
are rhb, and since, by Proposition \ref{pMatrixProduct}, taking products
is a bounded bi-linear maps with norm $\le 1$, it follows that the second function
$\O_{\ga'}(\s, \mu)\to \cM_{\ga',\vark}^b$ is rhb and
$$\aa{Jd^2g(x)[Jd^2f(x)]}_{\ga',\vark}\le  \aa{Jd^2g(x)}_{\ga',\vark}\aa{Jd^2f(x)}_{\ga',\vark}.$$

The derivatives with respect to $\r$ are treated alike.

(ii) That $ P_g^n f \in  \Tc_{\ga,\vark,\D}(\s,\mu)$ follows from (ii), but the estimate does not follow from
the estimate in (ii). The estimate follows instead from Cauchy estimates of  $n$-fold product $P_g^n f $.
\end{proof}

\begin{remark}
The proof shows that the assumptions can be relaxed when $g$ is a jet function: it suffices then to assume that
 $g\in  \Tc_{\ga,0,\D}(\s,\mu)$ and
$g-\hat g(\cdot,0,\cdot)
\footnote{\ $\hat g(\cdot,0,\cdot)$ this is the $0$:th Fourier coefficent of the function $\theta\mapsto g(\cdot,\theta,\cdot)$}
\in  \Tc_{\ga,\vark,\D}(\s,\mu)$.
Then $\{g,f\}$ will still be in $\Tc_{\ga,\vark,\D}(\s,\mu)$ but with the bound
$$
\ab{\{g,f\}}_{\begin{subarray}{c}\s',\mu' \  \\  \ga, \vark, \D \end{subarray}}\leq 
C_{\s-\s'}^{\mu-\mu'} 
\big(\ab{g}_{\begin{subarray}{c}\s,\mu \ \ \\ \ga, 0, \D \end{subarray}}
+\ab{g-\hat g(\cdot,0,\cdot)}_{\begin{subarray}{c}\s,\mu \ \ \\ \ga, \vark, \D \end{subarray}}\big)
\ab{f}_{\begin{subarray}{c}\s,\mu \ \ \\ \ga, \vark, \D \end{subarray}}.$$

To see this it is enough to consider a jet-function $g$ which does not depend on $\theta$. The only difference
with respect to case  (i) is for the second differential. The second term is fine
since, by Proposition \ref{pMatrixProduct},  $\cM_{\ga',\vark}^b$ is a two-sided ideal in  $\cM_{\ga',0}^b$ and
$$\aa{Jd^2g(x)[Jd^2f(x)]}_{\ga',\vark}\le  \aa{Jd^2g(x)}_{\ga',0}\aa{Jd^2f(x)}_{\ga',\vark}.$$
For the first term we must consider $Jd^3g(x)[Jdf(x)]$ which, a priori, takes its values in $\cM_{\ga',0}^b$ and not in
$\cM_{\ga',\vark}^b$. But since $g$ is a jet-function independent of $\theta$ this term is $=0$.

\end{remark}

 \subsubsection {Hamiltonian flows} \label{ss5.3}
 
 The Hamiltonian vector field of a $\cC^1$-function $g$ on (some open set in) $Y_\ga$  is $-Jdg$. 
 Without further assumptions it is an element in $Y_{-\ga}$, but if $g\in\cT_{\ga,\vark}$, then it is an element
 in $Y_\ga$ and has a well-defined local flow $\Phi_g$.

\begin{proposition}\label{Summarize}
Let $g\in  \Tc_{\ga,\vark,\D}(\s,\mu)$, and let $\s'<\s$ and $\mu'<\mu\le1$.
If 
$$
\ab{g}_{\begin{subarray}{c}\s,\mu\ \ \\ \ga, \vark, \D  \end{subarray}}\leq\\
 \frac1{C}\min( \s-\s',\mu-\mu'),$$
 then
\begin{itemize}
\item[(i)]  the Hamiltonian flow map  $\Phi^t=\Phi^t_g$ is,
  for  all $\ab{t}\le1$ and all $\ga_*\le \ga'\le\ga$,   a $\cC^{{s_*}}$-map
$$\O_{\ga'}(\s',\mu')\times\D \to\O_{\ga'}(\s,\mu)$$
which is real holomorphic and symplectic for any fixed  $\rho\in \D$. 

Moreover,
$$\aa{ \p_\r^j (\Phi^t(x,\r)-x)}_{\ga'}\le 
C\ab{g}_{\begin{subarray}{c}\s,\mu\ \ \\ \ga, \vark, \D  \end{subarray}},$$
and
$$
\aa{ \p_\r^j(d\Phi^t(x)-I)}_{\ga',\vark}\le
C\ab{g}_{\begin{subarray}{c}\s,\mu\ \ \\ \ga, \vark, \D  \end{subarray}},$$
for any $x\in \O_{\ga'}(\s',\mu')$, $\ga_*\le \ga'\le\ga$,  and $0\le \ab{j}\le {s_*}$.

\item[(ii)] $f\circ \Phi_g^t\in \Tc_{\ga,\vark}(\s',\mu',\D)$ for $\ab{t}\le1$ and
$$
\ab{ f\circ \Phi_g^t }_{\begin{subarray}{c}\s',\mu' \ \ \\ \ga, \vark, \D \end{subarray}}\leq C
\ab{f}_{\begin{subarray}{c}\s,\mu \ \ \\ \ga, \vark, \D \end{subarray}}.$$
\end{itemize}

$C$ is an absolute constant.  
\end{proposition}

\begin{proof}
It follows by general arguments that  $\Phi=\Phi_g: U\to \O_{\ga}(\s,\mu)$
is real holomorphic in $(t,\zeta)\in U\subset \C\times \O_{\ga}(\s,\mu)$ and depends smoothly on
any smooth parameter in the vector field.
Clearly,  for $\ab{t}\le 1$ and $x\in \O_{\ga}(\s',\mu')$
$$\aa{\Phi^t(x,\r)-x}_{\ga}\le  \sup_{x\in \O_{\ga}(\s,\mu)}\aa{Jdg(x)}_\ga 
\le \ab{g}_{\begin{subarray}{c}\s,\mu\ \ \\ \ga, 0, \D  \end{subarray}}$$
as long as $\Phi^t(x)$ stays in the domain $ \O_{\ga}(\s,\mu)$. It follows by classical arguments that this is the case if
$$\ab{g}_{\begin{subarray}{c}\s,\mu\ \ \\ \ga, 0, \D  \end{subarray}}
\le\cte \min(\s-\s',\mu-\mu').$$

{\it The differential}. We have
$$\frac{d}{dt} d\Phi^t(x)=-Jd^2g(\Phi^t(x))d\Phi^t(x)=B(t)d\Phi^t(x),$$
where $B(t)\in \cM_{\ga,\vark}^b$.
By re-writing this equation in  the integral form 
 $d\Phi^t(x)=\Id+\int_0^t B(s) d\Phi^s(x)\dd s$ and iterating this relation, we get that 
$ d\Phi^t(x)-\Id= B^{\infty}(t)$  with
 $$B^{\infty}(t)
 =\sum_{k\geq 1}\int_{0}^{t}\int_{0}^{t_{1}}\cdots \int_{0}^{t_{k-1}} \prod_{j=1}^{k}B(t_{j})\text{d}t_{k}  \cdots \text{d}t_{2}\,\text{d}t_{1}.$$

We get, by Proposition \ref{pMatrixProduct}, that $d\Phi^t(x)-\Id\in\cM_{\ga,\vark}^b$ and, for $\ab{t}\le1$ ,
$$
\aa{ d\Phi^t(x)-\Id}_{\ga,\vark}\le \sum_{k\geq 1} \aa{Jd^2g(\Phi^t(x))}^k_{\ga,\vark}\frac{t^k}{k!}\le
\aa{Jd^2g(\Phi^t(x))}_{\ga,\vark}.$$

In particular, $A=d\Phi^t(x)$ is a bounded bijective operator
on $Y_\ga$.  Since $Jd^2g$ is a Hamiltonian vector field we clearly have that
$$\Om(A\zeta,A\zeta')=\Om(\zeta,\zeta'),\quad\forall  \zeta,\zeta'\in Y_\ga,$$
so $A$ is symplective.
 
 {\it Parameter dependence}. For $\ab{j}=1$, we have
$$\frac{d}{dt} Z(t)=\frac{d}{dt} \frac{\p^j\Phi^t(x,\r)}{\p\r^j}=
B(t,\r)Z(t) -\frac{\p^j Jdg(x,\r)}{\p\r^j}  =B(t)Z(t)+A(t).$$

Since
$$\aa{A(t)}_\ga+ \aa{B(t)}_{\ga,\vark}\le \Cte \ab{g}_{\begin{subarray}{c}\s,\mu\ \ \\ \ga, \vark, \D  \end{subarray}},$$
it  follows by classical arguments, using Gronwall, that
$$\aa{Z(t)}_{\ga,0}\le \Cte \ab{g}_{\begin{subarray}{c}\s,\mu\ \ \\ \ga, \vark, \D  \end{subarray}}\ab{t}.$$
The higher order derivatives (with respect tp $\rho$) of   $\Phi^t(x,\r )$,  and the derivatives of $d\Phi^t(x,\r )$  are treated in the same way.

The same argument applies to any $\ga_*\le\ga'\le\ga$.

Since
$$
 f\circ \Phi_g^t =
\sum_{n\ge0}\frac1{n!}t^nP^n_{-g}f  ,$$
(ii) is a consequence of Proposition \ref{lemma:poisson}(ii).
\end{proof}

\section{Normal Form Hamiltonians and the KAM theorem}

\subsection{Block decomposition, normal form matrices.}
In this subsection we recall two notions introduced in 
\cite{EK10} for the nonlinear Schr\"odinger equation. 
They are essential  to overcome the problems of small divisors 
in a multidimensional context. Since the  structure of the spectrum for the beam equation,
 $\{\sqrt{|a|^4+m},\ a\in \Z^{d_*}\}$, is similar to that  for the NLS equation,
 $\{|a|^2+\hat{V}_a,\ a\in \Z^{d_*}\}$, then to study the beam equation we will use 
 tools, similar to those used to study the NLS equation.
\medskip

\noindent{\bf Block decomposition:} 
 For any $\Delta\in\N\cup \{\infty\}$ 
we define an equivalence relation on $\Z^{d_*}$, generated by the pre-equivalence relation
$$ a\sim b \Longleftrightarrow \left\{\begin{array}{l} |a|=|b| \\   {[a-b]}
 \leq \Delta. \end{array}\right.$$
(see \eqref{pdist}). 
Let $[a]_\Delta$ denote the equivalence class of $a$  -- the {\it block} of $a$. 
For further references we note that 
\be\label{a-b} 
|a|=|b| \text{ and }   [a]_{\Delta}\neq [b]_{\Delta}
\Rightarrow [a-b]\geq \Delta
\ee
The crucial fact is that the blocks have a finite maximal ``diameter''
$$d_\Delta=\max_{[a]=[b]} [a-b]$$
which do not depend on $a$ but only on $\Delta$.

\begin{proposition}\label{blocks}
\be\label{block}
d_\Delta\leq C \Delta^{\frac{(d_*+1)!}2}. 
\ee
The constant $C$ only depends on $d_*$.
\end{proposition}

\begin{proof} In \cite{EK10} it was considered the equivalence relation on $\Z^{d_*}$, generated by the 
pre-equivalence 
$$
a\approx b\quad\text{if}\quad |a|=|b|\quad \text{and}\quad |a-b|\le\Delta. 
$$
Denote by $[a]^o_\delta$ and $d^o_\Delta$ the corresponding equivalence class and its diameter (with respect to the 
usual distance). Since $a\sim b$ if and only if $a\approx b$ or $a\approx -b$, then 
\be\label{union}
[a]_\Delta = [a]^o_\Delta \cup -[a]^o_\Delta,
\ee 
provided that the union in the r.h.s. is disjoint. It is proved in \cite{EK10} that 
$d_\Delta^o\le D_\Delta=:C \Delta^{\frac{(d_*+1)!}2}$. Accordingly,
if $|a|\ge D_\Delta$, then the union above is disjoint, \eqref{union} holds and diameter of $[a]_\Delta$ satisfies 
\eqref{block}. If  $|a|< D_\Delta$, then $[a]_\Delta$ is contained in a sphere of radius $< D_\Delta$. So the block's
  diameter is at most 
$2D_\Delta$. This proves \eqref{block} if we replace there $C_{d_*}$ by $2C_{d_*}$.
\end{proof}

If $\Delta=\infty$ then the block of $a$ is the sphere $\{b: |b|=|a|\}$. 
Each block decomposition is a sub-decomposition of the trivial decomposition formed by the spheres $\{|a|=\text{const}\}$.

\medskip

\noindent
{\bf Normal form matrices.} 
Let $\E_\Delta$  be the decomposition of $\L=\F\sqcup\L_{\infty}$
into the subsets
$$[a]_\Delta=
\left\{\begin{array}{ll} [a]_\Delta\cap\L_{\infty} & a\in \L_{\infty}\subset\Z^{d_*}\\
\F & a\in \F.\end{array}\right.$$

\begin{remark}\label{remark-blocks}
Now the diameter of each block $[a]_\Delta$ is bounded by 
$$d_\Delta\leq C \Delta^{\frac{(d_*+1)!}2}$$
if moreover we let $C\ge \#\F$.
\end{remark}

On the space of $2\times 2$ complex matrices we introduce  a projection 
$$
\Pi: \Mat(2\times 2, \C)\to \C I+\C J,
$$
 orthogonal with respect to the Hilbert-Schmidt  scalar product. Note that 
$\C I+\C J$ is the space  of matrices, commuting with the symplectic matrix $J$. 
\begin{definition}\label{d_31}
We say that a matrix $A:\ \L\times \L\to \Mat(2\times 2, \C)$ is on normal form with respect to 
$\Delta$, $\Delta\in\N\cup \{\infty\}$,  and  write  $A\in  \NF_\Delta$, if
 \begin{itemize}
 \item[(i)] $A$ is real valued,
 \item[(ii)] $A$ is symmetric, i.e. $A_b^a\equiv {}^t\hspace{-0,1cm}A_a^b$,
 \item[(iii)] $A$ is block diagonal over $\E_\Delta$, i.e. $A_b^a=0$ if $[a]_\Delta\neq [b]_\Delta$,
 \item[(iv)] $A$ satisfies $\Pi A^a_b\equiv A^a_b$ for all $a,b\in\L_{\infty}$.

 \end{itemize}
 \end{definition}
 
Any quadratic form ${\mathbf q}(w)= \frac 1 2\langle w,Aw \rangle$, $w=(p,q)$, 
can be written as
$$
\frac 1 2\langle p,A_{11}p \rangle+\langle p,A_{12}q \rangle+\frac 1 2\langle q,A_{22}q \rangle 
+\frac 1 2\langle w_{\F}, H(\r) w_{\F}\rangle 
$$
where $A_{11},\ A_{22}$ and $H$ are real symmetric matrices and $A_{12}$ is a real matrix.

We now pass from the real variables $w_a=(p_a,q_a)$ to the 
complex variables $z_a=(\xi_a,\eta_a)$ by $w=U z$ defined through
\be\label{transf}
\xi_a=\frac 1 {\sqrt 2} (p_a+iq_a),\quad \eta_a =\frac 1 {\sqrt 2} (p_a-iq_a),\ee
for $a\in\L_\infty$, and acting like the identity on $  (\C^2)^\F$.
Then we have
$$
{\mathbf q}(Uz)=\frac 1 2\langle \xi,P\xi\rangle+ \frac 1 2\langle \eta,{\overline P}\eta\rangle+\langle \xi,Q\eta\rangle
+\frac 1 2\langle z_{\F}, H(\r) z_{\F}\rangle ,$$
where
$$P=\frac12\Big( (A_{11}-A_{22})-{\mathbf i}(A_{12}+{}^t A_{12})\big)$$
and
$$Q=\frac12\Big( (A_{11}+A_{22})+{\mathbf i}(A_{12}-{}^t A_{12})\big).$$
Hence $P$ is a complex symmetric matrix and $Q$ is a Hermitian matrix. If $A$ is on normal form, then $P=0$. 

Notice that this change of variables is not symplectic but
$$
U\left(\begin{array}{cc} J_\infty & 0\\ 0 & J_{\F}\end{array}\right) {}^t\! U=
\left(\begin{array}{cc} {\mathbf i}J_\infty & 0\\ 0 & J_{\F}\end{array}\right).$$

\subsection{The unperturbed Hamiltonian}\label{ssUnperturbed}

Let $h$ be a function as in \eqref{equation1.1}, i.e.
\be\label{unperturbed}
h (r,w,\r) =  \langle r,\om (\r)\rangle  + \frac12\langle w,A (\r)w \rangle,\ee
where
$$\langle w,A(\r)w\rangle =\langle w_{\F},H (\r)w_{\F} \rangle
+\frac12\big( \langle p_{\infty},Q (\r)p_{\infty} \rangle + \langle q_{\infty},Q (\r)q_{\infty} \rangle  \big)$$
and
$$Q (\r)=\diag\{\la(\r): a\in\L_\infty\}.$$
Assume $\om,\ \la,\ H$ verify \eqref{properties}. We shall denote by 
$$ \chi=
 |\nabla_\r \omega |_{\cC^{ {{s_*}}-1 } (\D)}+\sup_{a\in\L} |\nabla_\r \lambda_a|_{\cC^{ {{s_*}}-1 } (\D)}
 + ||\nabla_\r H ||_{\cC^{ {{s_*}}-1 } (\D) }$$
since this quantity will play an important role in our analysis.

\subsubsection{Assumption A1 -- spectral asymptotic}
\   

There exist  constants  $0< c,c'\le 1$  and exponents $\beta_1=2$, $\beta_2\ge0$, $\beta_3>0$  such that for all $\r\in\D$
the relations 
\eqref{laequiv}, \eqref{la-lb}, \eqref{la-lb-bis} and \eqref{la-lb-ter} hold.
\medskip



\subsubsection{Assumption  A2 -- transversality.}
Let  $[a]=[a]_\infty$ so that $[a]$ equals $\{b: \ab{b}=\ab{a}\}$ when $a\in\L_\infty$ and equals $\F$
when $a\in\F$ .  Denote by $Q_{[a]}$  the restriction of the matrix $Q$
to $[a]\times [a]$ and let $Q_{[\emptyset]}=0$. 
Let also $JH(\r)_{[\emptyset]}=0$ and $m=2\#\F$.

\medskip

There exists a  $1\ge\delta_0>0$ such that 
for all $\cC^{{s_*}}$-functions
$$\omega':\D\to \R^n,\quad |\omega'-\omega|_{\cC^{{s_*}}(\D)}<\delta_0,$$
the following hold for each  $k\in\Z^n\setminus 0$:

\begin{itemize}

\item[$(i)$] for any $a,b\in\L_\infty\cup \{\emptyset\} $ let
$$L(\r):X\mapsto \langle k,\om'(\r) \rangle X+Q_{[a]}(\r)X\pm XQ_{[b]};$$ 
then either $L(\r) $ is {\it $\de_0$-invertible} for all $ \r\in\D$  , i.e.
$$\aa{L(\r)^{-1}}\le \frac1{\delta_0}, \quad \forall \r\in\D,$$
or there exists a unit vector ${\mathfrak z}$ such that
$$\ab{\langle  v,\p_{\mathfrak z}  L(\r) v\rangle}  \ge \delta_0, \quad \forall \r\in\D
\footnote{\ $\p_{\mathfrak z} $ denotes here the directional derivative in the direction ${\mathfrak z}\in\R^p$}
$$
and for any unit-vector $v$ in the domain of $L(\r)$;

\item[$(ii)$]   let 
$$L(\r,\lambda):X\mapsto \langle k,\om'(\r) \rangle X+\lambda X+ {\mathbf i}XJH(\r)$$
and
$$P(\r,\lambda)=  \det L(\r,\lambda):$$
then  either $L(\r,\la(\r))$ is $\de_0$-invertible for all $ \r\in\D$  and $a\in[a]_\infty$, or
there exists a unit vector ${\mathfrak z}$ such that
$$
\ab{\p_{\mathfrak z}P(\r,\la(\r))+\p_\lambda P(\r,\la(\r))
\langle v,\p_{\mathfrak z} Q(\r) v\rangle}
\ge \delta_0\aa{L(\cdot,\la(\cdot))}_{\cC^{1}(\D)}^{m-1}$$
for all $\r\in \D$,  $a\in[a]_\infty$
and for any unit-vector $v\in (\C^2)^{[a]}$;

\item[$(iii)$] for any $a,b\in \F\cup \{\emptyset\} $
let
$$L(\r):X\mapsto \langle k,\om'(\r) \rangle X-{\mathbf i}JH(\r)_{[a]}X+ {\mathbf i}XJH(\r)_{[b]};$$ 
then  either $L(\r)$ is $\de_0$-invertible for all $ \r\in\D$, or
 there exists a unit vector ${\mathfrak z}$ and an integer  $1\le j\le {s_*}$ such that
$$\ab{ \p_{\mathfrak z}^j \det L(\r) }\ge \delta_0 \aa{L(\r)}_{\cC^{j}(\D)}^{m-1}, \quad \forall \r\in \D.$$
\end{itemize}

\medskip

 \begin{remark}\label{remro} The dichotomy 
in A2 is imposed not only on  $\omega$ but also on
$\Ca^1$-perturbations of $\omega'$, because, in general, the 
dichotomy for $\omega'$ does not imply that for perturbations.

If, however,  any $\Ca^{{s_*}}$ perturbation $\omega'$
of $\omega$ can be written as $\omega'=\omega\circ  f$
for some diffeomorphism $ f=id+\O(\delta_0)$ --   this is for example the case when $\omega(\r)=\r$,  -- 
then the dichotomy on $\omega$ implies a dichotomy on 
$\Ca^{{s_*}}$-perturbations.
\end{remark}

\subsection{Normal form Hamiltonians}
Consider now  the function \eqref{unperturbed} defined on the set $\D$. This function will
be fixed throughout this paper and we shall denote it and its ``components'' by 
$$h_{\textrm up},\ \om_{\textrm up},\ A_{\textrm up},\ Q_{\textrm up},\ H_{\textrm up}.$$							

 \begin{remark}\label{rConvention}
 The essential properties of  $h_{\textrm up}$ are given by  the constants 
 $$\chi,c',c, \beta=(\beta_1,\beta_2,\beta_3), \de_0.$$
 These will be fixed now, once and for all. All estimates will depend on $h_{\textrm up}$ only through these 
 constants. Since it will be important for our analysis of the Beam equation we shall track this dependence with respect to $c', \de_0,\chi$.
 In order to simplify the estimates a little we shall assume that  
 \be\label{Conv}0<c'\le \de_0\le\chi\le  c.\ee

\end{remark}

We shall consider functions of
the form
\be\label{normform}
h(r,w,\r)=\langle \om(\r), r\rangle +\frac 1 2\langle w, A(\r)w\rangle\ee  
which satisfies

\medskip

\noindent 
{\bf Hypothesis $\omega$:}
$\omega$ is of class $\Ca^{{s_*}}$ on $\D$ and
\be\label{hyp-omega}
|\omega-\omega_{\textrm up}|_{\Ca^{{s_*}}(\D)}\le \delta.
\ee

\medskip

\noindent
{\bf Hypothesis B: }   
$A-A_{\textrm up}:\D\to \cM_{0,\varkappa}^b$ is of class $\Ca^{{s_*}}$, 
$A(\r)$ is on normal form $\in \NF_{\Delta}$ for all $\r\in\D$
 and
  \be\label{hypoB}
 || \p_\r^j (A(\r)-A_{\textrm up}(\r))_{[a]} || \le  \delta\frac{1}{\langle a \rangle^\varkappa} \ee
for $ |j| \le {{s_*}}$, $a\in\L$  and $\r\in \D$.
\footnote{\ here it is important that $||\ ||$ is the matrix operator norm} We also require that
\be\label{varkappa}
\varkappa>0.\ee

\medskip

A function verifying these assumptions is said to be on
{\it normal form}. and we shall denote this by
$$h\in \NF_{\varkappa,h_{\textrm up} }(\Delta,\delta).$$
Since the unperturbed Hamiltonian $h_{\textrm up}$ will be fixed in this paper
we shall often suppress is in this notation writing simply  $h\in \NF_{\varkappa }(\Delta,\delta)$.

\subsection{The normal form  theorem}
In this section we state an abstract KAM result for perturbations of normal form Hamiltonians by a
function in $ \Tc_{\ga,\varkappa,\D }(\s,\mu)$, $0<\s,\mu\le 1$.
Let
$$\ga=(\ga_1,m_*)> \ga_*=(0,m_*)\quad\textrm{and}\quad  \vark> 0.$$

\begin{theorem}\label{main}
There exist positive constants $C$, $\alpha$, $\exp$  and  $\exp_3$ such that, for any
$h\in\NF_{\varkappa,h_{\textrm up}}(\Delta,\de)$ and for any 
$f\in \Tc_{\ga,\varkappa,\D }(\s,\mu)$,
$$ \eps=\ab{f^T}_{\begin{subarray}{c}\s,\mu\ \ \\ \ga, \varkappa,\D  \end{subarray}}\ \textrm{and}\  
 \xi=\ab{f}_{\begin{subarray}{c}\s,\mu\ \ \\ \ga, \varkappa,\D  \end{subarray}},$$
if
$$\delta \le \frac1{2C} c'$$
and
\be\label{epsi}
\eps(\log \frac1\eps)^{\exp}\le
\frac1{C}\big( \frac{ \max(\ga_1^{-1} ,d_{\Delta})}{\s\mu}\big)^{-\exp}
\big(\frac{c'}{\chi+\xi}\big)^{\exp_3} c',
\footnote{\ remember the convention \eqref{Conv} }
\ee
then there exist a set $\D'=\D'(h, f)\subset \D$,
$$\Leb (\D\setminus \D')\leq 
C\big(\log\frac1{\eps} \frac{ \max(\ga_1^{-1} ,d_{\Delta})}{\s\mu} \big)^{\exp}
( \frac{\chi+\xi}{\delta_0})^{1+\alpha}
(\frac{\eps}{\de_0})^{\alpha},$$
and a $\cC^{{s_*}}$ mapping
$$\Phi:\O_{ \ga_*}(\s/2,\mu/2)\times\D\to \O_{ \ga_*}(\s,\mu),$$
real holomorphic and symplectic for each parameter $\r\in\D$, such that 
$$(h+ f)\circ \Phi= h'+f'$$
and
\begin{itemize}
\item[(i)] for $\r\in\D'$ and $\zeta=r=0$
$$d_r f'=d_\theta f'=
d_{\zeta} f'=d^2_{\zeta} f'=0;$$
\item[(ii)] $ h'\in\NF_{\vark}(\infty,\de')$, $\de'\le \frac{c'}2$,  and
$$\ab{ h'- h}_{\begin{subarray}{c}\s/2,\mu/2\ \ \\ \ga_*, \varkappa,\D  \end{subarray}}\le C;$$

\item[(iii)] 
for any $x\in \O_{ \ga_*}(\s/2,\mu/2)$ and $\ab{j}\le{s_*}$
$$|| \p_\r^j (\Phi(x,\cdot)-x)||_{ \ga_*}+ \aa{ \p_\r^j (d\Phi(x,\cdot)-I)}_{ \ga_*,\vark} \le C;$$

\item[(iv)]  if  $\tilde \r=(0,\r_2,\dots,\r_p)$ and $f^T(\cdot,\tilde \r)=0$ for all $\tilde \r$,  then $h'=h$ and $\Phi(x,\cdot)=x$
 for all $\tilde \r$.
\end{itemize}

 $C$ is an absolut constant that only depends on  $\beta, \vark,c$ and $\sup_\D\ab{\om}$.
The exponent $\exp'$ is an absolute constant that only depends  on $\beta$ and $\vark$. 
The exponent $\exp_3$ only depends on $s_*$.
The exponent $\alpha'$ is a positive constant only depending on  $s_*,\frac{d_*}{\vark} ,\frac{d_*}{\beta_3} $.

\end{theorem}

The condition on $\Phi$ and $h'-h$ may look bad but it is not. 

\begin{corollary}\label{cMain} Under the assumption of the theorem, let $\eps_*$ be the largest positive number such that \eqref{epsi}
holds. Then,
for any  $\r\in\D$ and $\ab{j}\le{s_*}-1$,
\begin{itemize}
\item[$(ii)'$] 
$$\ab{  \p_\r^j (h'(\cdot,\r)- h(\cdot,\r))}_{\begin{subarray}{c}\s/2,\mu/2\ \ \\ \ga_*, \varkappa,\ \ \  \ \end{subarray}}\le 
\frac{C}{\eps_*}\ab{f^T}_{\begin{subarray}{c}\s,\mu\ \ \\ \ga, \varkappa,\D  \end{subarray}};$$

\item[$(iii)'$] 
$$|| \p_\r^j (\Phi(x,\r)-x)||_{\ga_*}+ \aa{ \p_\r^j (d\Phi(x,r)-I)}_{\ga^*,\vark} \le \frac{C}{\eps_*}\ab{f^T}_{\begin{subarray}{c}\s,\mu\ \ \\ \ga, \varkappa,\D  \end{subarray}},$$
for any $x\in \O_{\ga_*}(\s/2,\mu/2)$.
\end{itemize}
The constant $C$ is an absolute constant that also depends on $\beta$. 
\end{corollary}

\begin{proof} Let us denote $\r$ here by  $\tilde \r$.
If  $|f^T|_{\begin{subarray}{c}\s,\mu\ \ \\ \ga, \vark,\D  \end{subarray}}\le \eps_*$, then we can apply the theorem to $\eps f$ for any $|\eps|\le 1$. Let now
$\r=(\eps,\tilde \r)$ and consider $h$ and $h_{\mathrm up}$ as functions depending on 
this new parameter $\r$ --  they will still verify the assumptions of the theorem, which will provide us with a mapping $\Phi$ with a $\cC^{{s_*}}$ dependence in 
$\r=(\eps,\tilde \r)$ and equal to the identity when $\eps=0$. The bound on the derivative now implies that
$$|| \Phi(x,\eps,\tilde\r)-x  ||_{\ga_*}\le C\eps\le
\frac C{\eps_*}|f^T|_{\begin{subarray}{c}\s,\mu\ \ \\ \ga, \vark,\D  \end{subarray}}
$$
for any $x\in \O_{\ga_*}(\s/2,\mu/2)$. The same estimate holds for all derivatives with respect to 
$\tilde \r$ up to order ${{s_*}}-1$.

The argument for $h'-h$ is the same.
\end{proof}

We  can take $\de=0$ here, in which case $h$ equals the unperturbed Hamiltonian $h_{\textrm up}$  -  this is the case described in theorem of the Introduction.

\begin{remark} 
The assumption \eqref{epsi} on $\epsilon$  involves many constants and parameters. 
Then  \eqref{epsi} takes the form
$$\eps(\log \frac1\eps)^{\exp}\le C'\big(\frac{c'}{\chi+\xi}\big)^{\exp_3} c'$$
where $C'$ depends on $C\ga,\s,\mu,\Delta$.
If we assume that
$$\xi, \chi =O(\delta_0^{1-\aleph}) \quad \text{and} \quad c'=O(\delta_0^{1+\aleph})$$
for some $\aleph>0$, then assumption \eqref{epsi} reduces to 
$$\eps(\log \frac1\eps)^{\exp}\lsim C'
\delta_0^{1+\aleph+2 \aleph \exp_3},$$
which implies
$$\eps_*\gsim C'
\delta_0^{1+2\aleph+2 \aleph \exp_3}$$
when $\de_0$ is sufficiently small.

Actually in paper \cite{EGK} Theorem \ref{main} this is used in this context.
\end{remark}

\section{Small Divisors}\label{s6}

For a mapping $L:\D\to gl(\dim,\R)$ define, for any $\ka>0$,
$$\Sigma(L,\kappa)=\{\r\in \D:   ||L^{-1}| |>\frac1\ka\}.$$
Let  
$$h(r,w,\r)=\langle r,\om(\r)\rangle+ \frac12\langle w,A(\r) w \rangle$$
be a  normal form Hamiltonian in $\NF_{\varkappa}(\Delta,\delta)$.
Recall the convention in Remark \ref{rConvention} and assume $\vark>0$ and
\be\label{ass} \delta \le \frac{1}{C}c' ,\ee
where $C$ is to be determined.

\begin{lemma}\label{lSmallDiv1}
Let
$$L_{k}=\langle k,\om(\r)\rangle.$$
There exists  a constant $C$ such that if \eqref{ass} holds, then
$$\Leb\big(\bigcup_{0<\ab{k}\le N} \Sigma(L_k,\ka)\big)
\le C N^{\exp} \frac{\chi+\delta}{\delta_0}\frac{\ka}{\de_0}$$
and
$$\dist(\D\setminus \Sigma(L_k,\ka),\Sigma(L_k,\frac\ka2))>\frac{1}{C}\frac{\ka}{N(\chi+\delta)}
 \footnote{\ this is assumed to be fulfilled if $\Sigma_{L_j}(\frac\ka2)=\emptyset$}
 $$
for any $\ka>0$.

The exponent $\exp$ only depends on $\#\cA$. $C$ is an absolut constant.
\end{lemma}

\begin{proof} 
Since $\delta\le\delta_0$, 
using  Assumption A2$(i)$, with $a=b=\emptyset$,  we have, for each $k\not=0$, either that
$$ |\langle\om(\r), k\rangle|\ge \de_0\ge \ka\quad \forall \r\in \D$$
or that
$$ \p_{\mathfrak z} \langle\om(\r), k\rangle  \ \geq \delta_0\quad \forall \r\in \D
$$
(for some suitable choice of a unit vector $\mathfrak z$). The first case implies  
$\Sigma(L_{k},\ka)=\emptyset$. The second case implies  that 
 $\Sigma(L_k,\kappa)$
 has Lebesgue measure 
 $$\lsim \frac{N(\chi+\delta)}{\delta_0} \frac{\ka}{\de_0}.$$
 Summing up over all $0<\ab{k}\le N$ gives the first statement. The second statement
 follows from the mean value theorem and the bound
$$\ab{\nabla_\r L_k(\r)}\le N(\chi+\delta).$$
\end{proof}

\begin{lemma}\label{lSmallDiv2}
Let 
$$L_{k,[a]}=\big(  \langle k,\om \rangle I - J A \big)_{[a]}.$$
There exists  a constant $C$ such that if \eqref{ass} holds, then,
$$\Leb\big(\bigcup_{\begin{subarray}{c} 0<\ab{k}\le N\\  [a] \end{subarray}} \Sigma(L_{k,[a]}(\ka)\big)
\le C N^{\exp} \big(\frac{\chi+\delta}{\delta_0}\big) (\frac{\ka}{\delta_0})^{\frac1{{s_*}}}$$
and
$$\dist(\D\setminus \Sigma(L_{k,[a]},\ka),\Sigma(L_{k,[a]},\frac\ka2))>\frac{1}{C}\frac\ka{N(\chi+\delta)},$$
for any $\ka>0$.

$\exp$ only depends  on $\frac{d_*}{\beta_1}$ and $\#\cA$. $C$ is an absolut constant that depends
on $c$ and $\sup_\D\ab{\om}$.

\end{lemma}

\begin{proof}  
Consider first $a\in\L_\infty$. Then  $L_{k,[a]}$ decouples into to a sum of two Hermitian
operators of the form
$$\langle k,\om\rangle I + Q_{[a]}$$
-- denoted $L=L_{k,[a]}$  --  where $Q_{[a]}$  is  the restriction of $Q$ to $[a]\times [a]$. Since $L$ is Hermitian:

\begin{itemize}
\item
 $$ ||L(\r)^{-1}|| \leq \max  ||\big(\langle k,\om (\r)\rangle+\lambda(\r) \big)^{-1}|| ,$$
 where the maximum is taken over all eigenvalues $\lambda(\r)$ of $Q(\r)$;
\item for any $\r_0\in\D$,
 $$\p_{\mathfrak z} \langle v(\r)  L(\r), v(\r)\rangle \slash_ {\r=\r_0} \ =\   
 \p_{\mathfrak z} \langle v(\r_0)  L(\r), v(\r_0)\rangle \slash_ {\r=\r_0} $$
for any  eigenvector  $v(\r)$  of $L(\r)$ (associated to an eigenvalue which is $\cC^1$ in the direction $\mathfrak z$ ).
\end{itemize}

If we let
$$L_{\textrm up}= \langle k,\om\rangle I + (Q_{\textrm up}) _{[a]},$$
where $ Q_{\textrm up}$ comes from the unperturbed Hamiltinonian,
then it follows, from \eqref{hypoB} and  \eqref{ass},   that   
$$\aa{L-L_{\textrm up}}_{\cC^{{s_*}}(\D)} \leq  \de\leq \frac{\delta_0}2$$
and, hence,
$$\textrm{d}_{\Haus}(\s(L),\s(L_{\textrm up}))<\delta.$$
If now $L_{\textrm up}$ is $\delta_0$-invertible, then  this  implies that
$L$ is $\frac{\delta_0}2$-invertible. 

Otherwise
$$\p_{\mathfrak z}\big( \langle k,\om(\r)\rangle+\lambda(\r)\big)\slash_{\r=\r_0}=
\langle v(\r),\p_{\mathfrak z}  L(\r) v(\r)\rangle\slash_{\r=\r_0} =$$
$$=
\langle v(\r),\p_{\mathfrak z}  L_{\textrm up}(\r) v(\r)\rangle\slash_{\r=\r_0} +\O( \delta),$$
where $v(\r)$ is a unit eigenvector of $L(\r)$ associated with the eigenvalue $\lambda(\r)$.
If $L_{\textrm up}$ is not $\delta_0$-invertible, then, by  assumption A2$(i)$
there exists a unit vector ${\mathfrak z}$ such that 
$$\ab{\p_{\mathfrak z}\big( \langle k,\om(\r)\rangle+\lambda(\r)\big)}\slash_{\r=\r_0}\ge \delta_0-\delta \ge \frac{\delta_0}2.$$
Hence, the Lebesgue measure of $\Sigma(L _{k,[a]},\ka)$ is $\lsim  d_{\Delta}^{d_*} \frac\ka{\delta_0}$  --
recall that, by Remark  \ref{remark-blocks}, the operator is of dimension $\lsim d_{\Delta}^{d_*}$.
(This argument is valid  if $\lambda(\r)$ is $\cC^1$ in the direction $\mathfrak z$ which can always be assumed
when $Q$ is analytic in $\r$. The non-analytic case follows by analytical approximation.)

Since $|\langle k, \om(\r)\rangle |\lsim \ab{k}\lsim N$, it follows,  by \eqref{la-lb-ter},  that
$$|\langle k, \om(\r)\rangle  + \lambda(\r)|\geq \ab{\la(\r)}-\delta-\Cte\ab{k}
\ge \ab{a}^{\beta_1}- c \langle a \rangle ^{-\beta_2} -\delta-\Cte\ab{k}$$
 for some appropriate $a\in[a]$.
Hence, $\Sigma(L_{k,[a]}, \ka)=\emptyset$ for
 $ |a |\gsim N^{\frac1{\beta_1}}$.

Summing up  over all $0<\ab{k}\le N$ and 
all  $ |a |\lsim N^{\frac1{\beta_1}}$ gives the first estimate.

\medskip

Consider now $a\in\F$.  Then $L(\r)=\big(\langle k,\om\rangle I- {\mathbf i}JH\big)$ and
it follows, by \eqref{hypoB} and  \eqref{ass}, that
$$\aa{L-L_{\textrm up}}_{\cC^{{s_*}}}\le   \de\leq  \frac12 \delta_0,$$ 
where $L_{\textrm up}(\r)=\big(\langle k,\om \rangle I- {\mathbf i}JH_{\textrm up}\big)$.
If now $L_{\textrm up}$ is $\delta_0$-invertible, then
$L$ will be  $\frac{\delta_0}2$-invertible.

Otherwise,
$$\ab{\det L-\det L_{\textrm up}}_{\cC^{j}} \le \Cte  \delta\big(\aa{L_{\textrm up}}_{\cC^{j}}+\delta \big)^{m-1}$$
and, by  assumption A2(iii),
there exists a unit vector ${\mathfrak z}$ and an integer  $1\le j\le {s_*}$ such that
$$\ab{ \p_{\mathfrak z}^j \det L_{\textrm up}(\r) }\ge \delta_0\aa{ L_{\textrm up}}_{\cC^{j}(\D)}^{m-1}, \quad \forall \r\in \D.$$
This implies that $|L_{\textrm up}|_{\cC^{j}}\ge \cte \delta_0$ and,
hence,
$$\ab{\det L-\det L_{\textrm up}}_{\cC^{j}}\le \Cte  \delta \aa{L_{\textrm up}}_{\cC^{j}}^{m-1}.$$
Thus
$$\ab{ \p_{\mathfrak z}^j \det L(\r) }\ge (\delta_0-\Cte\delta)\aa{ L}_{\cC^{j}(\D)}^{m-1}, \quad \forall \r\in \D,$$
and $\delta_0-\Cte\delta\ge\frac{\delta_0}2$.

Then, by Lemma  \ref{lTransv1},
$$\frac{\det L(\r)} { \aa{ L }_{\cC^{j}}^{m-1} } \ge \ka,$$
outside a set $\Sigma'$ of Lebesgue measure
$$\le \Cte \frac{|\nabla_\r L|_{\cC^{{j}-1}(\D)}}{\de_0}(\frac\ka{\de_0})^{\frac1j}.$$
Hence, by Cramer's rule, 
$$
\Leb \Sigma(L,\eps)\le \Cte \frac{|\nabla_\r L |_{\cC^{{j}-1}(\D)}}{\de_0}(\frac\ka{\de_0})^{\frac1j}
\le  \Cte \frac{N(\chi+\delta)}{\de_0}(\frac\ka{\de_0})^{\frac1j}.$$
Summing up  over all $\ab{k}\le N$  gives the first estimate.

\medskip

The second estimate follows from the mean value theorem and the bound
$$\ab{\nabla_\r L_{k,[a]}(\r)}\le N(\chi+\delta).$$

\end{proof}

\begin{lemma}\label{lSmallDiv3}
Let  
$$L_{k,[a],[b]}=(\langle k,\om \rangle I-{\mathbf i}\cAd_{JA})_{[a]}^{[b]}.$$
There exists  a constant $C$ such that if \eqref{ass} holds, then,
$$\bigcup_{\begin{subarray}{c}0<\ab{k}\le N\\ \ab{a-b}\le\Delta' \end{subarray}} \Sigma(L_{k,[a],[b]},\ka)
\le \Cte N^{\exp}\big(\frac{\chi+\delta}{\delta_0}\big)  (\frac{\ka}{\delta_0})^{\alpha}$$
and
$$\dist(\D\setminus \Sigma(L_{k,[a],[b]},\ka),\Sigma(L_{k,[a],[b]},\frac\ka2))>\frac{1}{C}\frac\ka{N(\chi+\delta)},$$
for any $\ka>0$.

The exponent $\exp$ only depends  on $\frac{d_*}{\beta_1} $ and $\#\cA$. The exponent
$\alpha$ is a positive constant only depending on  $s_*,\frac{d_*}{\vark} ,\frac{d_*}{\beta_3} $.
$C$ is an absolut constant that depends on $c$ and $\sup_\D\ab{\om}$.

\end{lemma}

\begin{proof} 
Consider first  $a,b \in\F$. This case is treated as the operator
$L(\r)=\big(\langle k,\om \rangle I- {\mathbf i}JH\big)$  in the previous lemma.

\medskip

Consider then  $a\in\L_\infty$ and $b \in\F$, so that
$$L_{k,[a]}(\r):X\mapsto \langle k,\om(\r) \rangle X + Q_{[a]}(\r)X+X{\mathbf i}JH(\r).$$
Let
$$L(\r,\lambda):X\mapsto \langle k,\om(\r) \rangle X+\lambda X+ {\mathbf i}XJH(\r)$$
and
$$P(\r,\lambda)=  \det L(\r,\lambda).$$

Since $L_{k,[a]}(\ka)$ is ``partially'' Hermitian,
 $$ ||L_{k,[a]}(\r)^{-1}|| \leq \max \aa{L(\r,\lambda(\r))^{-1}},$$
 where the maximum is taken over all eigenvalues $\lambda(\r)$ of $Q(\r)$.

If we let
$$L_{\textrm up}(\r,\lambda):X\mapsto \langle k,\om(\r) \rangle X + \lambda X+X{\mathbf i}JH_{\textrm up}(\r),$$
then it follows, from \eqref{hypoB} and  \eqref{ass}, that   
$$\aa{L-L_{\textrm up}}_{\cC^{{s_*}}(\D)} \leq  \de\leq \frac{\delta_0}2.$$
If $L_{\textrm up}$ is $\delta_0$-invertible, then this  implies that
$L$ is $\frac{\delta_0}2$-invertible.

Otherwise
$$\frac{d}{d_{\mathfrak z}}P(\r,\lambda(\r))=
\p_{\mathfrak z}P(\r,\lambda(\r))+\p_\lambda P(\r,\lambda(\r))\langle v(\r),\p_{\mathfrak z}Q(\r)v(\r)\rangle=$$
$$=\p_{\mathfrak z}P_{\textrm up}(\r,\lambda_a(\r))+\p_\lambda P_{\textrm up}(\r,\lambda_a(\r))
\langle v(\r),\p_{\mathfrak z}Q_{\textrm up}(\r)v(\r)\rangle
+\O(\delta || L||_{\cC^{1}(\D)}^{m-1}).$$
By  Assumption A2$(ii)$
there exists a unit vector ${\mathfrak z}$ such that
$$
\ab{\frac{d}{d_{\mathfrak z}}P(\r,\lambda(\r))}\ge\frac{\delta_0}2 || L||_{\cC^{1}(\D)}^{m-1}.$$
Hence, the Lebesgue measure of $\Sigma(L_{k,[a]} ,\ka)$ is $\lsim  d_{\Delta}^{d_*} \frac\ka{\delta_0}$  --
recall that, by Remark  \ref{remark-blocks}, the operator is of dimension $\lsim d_{\Delta}^{d_*}$.
(This argument is valid  if $\lambda(\r)$ is $\cC^1$ in the direction $\mathfrak z$ which can always be assumed
when $Q$ is analytic in $\r$. The non-analytic case follows by analytical approximation.)

Let $\lambda(\r)$ be an eigenvalue of $Q_{[a]}$. Since $|\langle k, \om(\r)\rangle |\lsim \ab{k}\lsim N$, it follows,  by \eqref{la-lb-ter},  that
$$|\langle k, \om(\r)\rangle  + \lambda(\r)|\geq \ab{\la(\r)}  -\delta-\Cte\ab{k}
\ge \ab{a}^{\beta_1}-c \langle a \rangle ^{-\beta_2} -\delta -\Cte\ab{k}$$
 for some appropriate $a\in[a]$. Hence, 
$\Sigma(L_{k,[a]},\ka)=\emptyset$ for $ |a |\gsim(N)^{\frac1{\beta_1}}$.

Summing up  over all $0<\ab{k}\le N$ and 
all  $ |a |\lsim(N)^{\frac1{\beta_1}}$ gives the first estimate.

\medskip

Consider finally $a,b \in\L_\infty$. Then  $L_{k,[a],[b]}$ decouples into a sum of four Hermitian
operators of the forms
$$L_{k,[a],[b]}(\r):X\mapsto \langle k,\om \rangle X + Q_{[a]}X+X{}^tQ_{[b]}$$
and
$$L_{k,[a],[b]}(\r):X\mapsto \langle k,\om \rangle X + Q_{[a]}X-XQ_{[b]}.$$
The first one is treated exactly as the operator $X\mapsto \langle k,\om \rangle X + Q_{[a]}X$
in the previous lemma, so let us concentrate on the second one. 
It follows as in the previous lemma
that  the Lebesgue measure of $\Sigma(L_{k,[a],[b]},\ka')$ is 
$\lsim  d_{\Delta}^{2d_*} \frac{\ka'}{\delta_0}$  --
recall that, by Remark  \ref{remark-blocks}, the operator is of dimension $\lsim d_{\Delta}^{2d_*}$.

The problem now is the measure estimate of $\bigcup\Sigma(L_{k,[a],[b]},\ka)$ since, a priori, there may be infinitely many 
$\Sigma(L_{k,[a],[b]},\ka)$ that are non-void. 
We can assume without restriction that $\ab{a}\le \ab{b}$.
Since $|\langle k, \om(\r)\rangle |\le\Cte \ab{k}\le\Cte N$, it is enough to 
consider  $\ab{b}-\ab{a} \le\Cte N$ (because $\beta_1\ge1$).
 
Since $\beta_1=2$,  $\ab{a}^{\beta_1}-\ab{b}^{\beta_1}$ is an integer, and outside a set $\Sigma(2\ka')$
of Lebesgue measure
$$\le\Cte N\frac{\ka'}{\delta_0}$$
we have
$$|\langle k, \om(\r)\rangle \ +\ab{a}^{\beta_1}-\ab{b}^{\beta_1}|\ge2\ka'.$$

 Then, by \eqref{la-lb}, 
  $$|\langle k,\om(\r)\rangle \ +\alpha(\r)-\beta(\r)|\ge 2\ka'-
  2\frac{\de}{\langle  a\rangle^\varkappa}-  2\frac{c'c}{\langle  a\rangle^{\beta_3}}$$
for any $\alpha(\r)$ and $\beta(\r)$, eigenvalues of $Q_{[a]}(\r)$ and $Q_{[b]}(\r)$, respectively. 
Now this is $\ge\ka'$ unless
$$\ab{a} \le \Cte \min\Big(  (\frac{\de}{\ka'})^{\frac1{\vark}},( \frac{c'}{\ka'})^{ \frac1{\beta_3}}\Big)=:M.$$
 Hence, if $\ka'\ge\ka$, then
$$\bigcup_{[a],[b],k}\Sigma(L_{k,[a],[b]},\ka)\subset 
\Sigma(2\ka')\cup 
\bigcup_{ \ab{a},\ab{b}\le M+\Cte N,k}
\Sigma(L_{k,[a],[b]},\ka).$$
This set has measure
$$\lsim N\frac{\ka'}{\delta_0}+ 
\Big(N\frac{\de_0}{\ka'}\Big)^{2d_*\max( \frac1{\beta_3},\frac1{\vark})}\frac\ka{\delta_0} .$$

By an appropriate choice of $\ka'\in[\ka, \delta_0]$, this becomes
$$\le\Cte N^{\exp} (\frac{\ka}{\delta_0})^{\alpha}$$
for some $\alpha>0$ depending on $\frac{d_*}{\beta_3},\frac{d_*}{\vark}$.

\end{proof}

\section{Homological equation}\label{s6}

Let $h$ be a  normal form Hamiltonian \eqref{normform},
$$
h(r,w,\r)=\langle \om(\r), r\rangle +\frac 1 2\langle w, A(\r) w\rangle\in\NF_{\varkappa}(\Delta,\delta)$$
and assume $\vark>0$ and 
 \be\label{ass1}
 \delta \le \frac{1}{C}c' ,\ee
 where $C$ is to be determined. Let
$$\ga=(\ga,m_*)\ge \ga_*=(0,m_*).$$

\begin{remark}\label{rAbuse}
Notice the abuse of notations here. It will be clear from the context when $\ga$ is a two-vector, like in $\aa{\cdot}_{\ga,\vark}$,
and when it is a scalar, like in $e^{\ga d}$.
\end{remark}

Let $f\in \Tc_{\ga,\varkappa,\D}(\s,\mu)$. In this section we shall construct a jet-function  $S$ that solves the 
{\it non-linear homological equation}
\be\label{eqNlHomEq}
\{ h,S \}+ \{ f-f^T,S \}^T+f^T=0\ee
as good as possible  -- the reason for this will be explained in the beginning of the next section. 
In order to do this we shall start by analyzing  the {\it homological equation}
\be \label{eqHomEq}
\{ h,S \}+f^T=0.
\ee
We shall solve this equation modulo some ``cokernel'' and modulo an ``error''.

\medskip 

\subsection{Three components of the homological equation}\label{ssFourComponents}

Let us write 
$$f^T(\theta,r,w)=f_r(r,\theta)+\langle f_w(\theta),w\rangle+\frac 1 2 \langle f_{ww}(\theta)w,w \rangle$$
and recall that, by Proposition \ref{lemma:jet}, $f^T\in \Tc_{\ga,\varkappa,\D}(\s,\mu)$.
Let
$$
S(\theta,r,w)=S_r(r,\theta)+\langle S_w(\theta),w\rangle+
\frac 1 2 \langle S_{ww}(\theta)w,w \rangle,$$
where  $f_r$ and $S_r$ are affine functions in $r$ -- here we have not indicated the dependence on $\r$.

Then the Poisson bracket $\{h, S\}$ equals
\begin{multline*}
-\big( \p_{\omega} S_r(r, \theta)  + \langle \p_{\omega} S_w(\theta), w\rangle
+ \frac12 \langle \p_{\omega} S_{ww}(\theta),w\rangle + \\
+\langle AJ S_w(\theta),w\rangle + \frac12\langle AJ S_{ww}(\theta)w,w\rangle  - \frac12\langle  S_{ww}(\theta)JAw,w\rangle
\end{multline*}
where $\p_{\omega}$ denotes the derivative of the angles $\theta$ in direction $\om$.
Accordingly the  homological equation \eqref{eqHomEq}  
 decomposes into three linear equations:
$$\left\{\begin{array}{l}      
\p_{\omega} S_r(r,\theta) =f_r(r,\theta),\\  
 \p_{\omega} S_w(\theta)   -AJ  S_w(\theta)= f_w(\theta),\\ 
 \p_{\omega} S_{ww}(\theta)   - AJS_{ww}(\theta)  +S_{ww}(\theta)JA=f_{ww}(\theta).
\end{array}\right.$$

\subsection{The first equation}\label{homogene}

\begin{lemma}\label{prop:homo12}
There exists   constant $C$ such that if \eqref{ass1} holds, then,
 for any  $N\ge1$ and $\ka>0$,
 there exists  a closed  set $\D_1= \D_1(h,\ka,N)\subset \D$,  satisfying
$$\Leb (\D\setminus \D_1)\leq C N^{\exp} \frac{\chi+\delta}{\delta_0}\frac{\ka}{\de_0}$$
and there exist $\Ca^{{s_*}}$ functions $S_r$ and $R_r$ on $\C^{\cA}\times\T^\cA\times \D\to\C$,  
real holomorphic in $r,\theta$, such that for all $\r\in\D_1$
 \be\label{homo1}
 \p_{\om(\r)}S_r (r,\theta,\r)  =f_r(r,\theta,\r)-\hat f_r(r,0,\r)
 \footnote{\ $\hat f_r(r,0,\r)$ is the $0$:th Fourier coefficent, or the mean value, of the function $\theta\mapsto  f_r(r,\theta,\r)$}
 -R_r(\theta,\r)\ee
and for all $(r,\theta,\r)\in \C^{\cA}\times \T^\cA_{\s'}\times \D$, $\ab{r}<\mu$, $\s'<\s$, and $|j|\le{{s_*}}$
 \begin{align} \label{homo1S}
 |\p_\r^jS_r(r,\theta,\r)|\leq &
C \frac{1}{\ka(\s-\s')^{n}}\big(N\frac{\chi+\de}{\ka}\big)^{|j|} 
  |f^T|_{\begin{subarray}{c}\s,\mu\ \ \\ \ga, \vark,\D  \end{subarray}} ,\\ \label{homo1R}
 |\p_\r^j R_r(r,\theta,\r)|\leq & C\frac{  e^{- (\s-\s')N}}  {  (\s-\s')^{n}}|f^T|_{\begin{subarray}{c}\s,\mu\ \ \\ \ga, \vark,\D  \end{subarray}}\,. 
\end{align}

The  exponent $\exp$  only depends on $n=\#\cA$, and $C$ is an absolut constant.
\end{lemma}

\begin{proof} Written in Fourier components the 
equation \eqref{homo1} then becomes, for $k\in \Z^{\cA}$,
$$L_k(\r)\hat S(k)=:
 \langle k, \om(\r) \rangle \hat S(k)=-{\mathbf i}(\hat F(k)-\hat R(k))$$
where we have written $S,F$ and $ R$ for $S_r, (f_r-\hat f_r)$ and $R_r$ respectively.
Therefore \eqref{homo1} has the (formal) solution 
$$S(r,\theta,\r)=\sum\hat S (r,k,\r) e^{{\mathbf i}\langle k,\theta\rangle}\quad\textrm{and}\quad
R(r,\theta,\r)=\sum\hat F (r,k,\r) e^{{\mathbf i}\langle k,\theta\rangle}$$
with
$$\hat S(r,k,\r)= 
\left\{\begin{array}{ll}
-L_k(\r)^{-1}{\mathbf i}\hat F(r,k,\r) & \textrm{ if }  0< |k|\le N\\
0 & \textrm{ if not}  
\end{array}\right.$$
and
$$\hat R(r,k,\r)= 
\left\{\begin{array}{ll}
\hat F_{a}(r,k,\r ) & \textrm{ if }   |k|> N\\
0& \textrm{ if not}.
\end{array}\right.$$

By Lemma \ref{lSmallDiv1}
$$ ||(L_k(\r))^{-1}||\le \frac1\ka\,
 $$
for all $\r$ outside some  set $\Sigma(L_k,\kappa)$ such that
$$\dist(\D\setminus \Sigma(L_k,\ka),\Sigma(L_k,\frac\ka2))\ge \cte\frac\ka{N(\chi+\delta)}$$
and
$$\D_1=\D\setminus \bigcup_{0<|k|\le N}\Sigma(L_k,\ka)$$
fulfills the estimate of the lemma.

 For $\r\notin \Sigma(L_k,\frac\kappa2)$ we get
 $$ |\hat S (r,k,\r)| \le \Cte\frac{1}{\ka}|\hat F(r,k,\r)|\,.$$
 Differentiating the formula for $\hat S(r,k,\r)$ once we obtain
 $$\p^j_\r \hat S(r,k,\r)=
 \Big(
 \ -\frac{{\mathbf i}}{ \langle\om, k\rangle}\p^j_\r \hat F(r,k,\r)+
 \ \frac{{\mathbf i}}{ \langle\om, k\rangle^2} \langle \p^j_\r\om, k\rangle\hat F(r,k,\r)\Big)
 $$
 which gives, for $\r\notin \Sigma(L_k,\frac\kappa2)$,
 $$
 |\p^j_\r\hat S(r,k,\r)|\le \Cte \frac1\kappa( N\frac{\chi+\de}{\kappa})\max_{0\le l\le j}|\p^l_\r \hat F(r,k,\r)|.
 $$
(Here we used that $|\p_\r \omega(\rho)|\le \chi+\de$. )
The higher order derivatives are estimated in the same way and this gives
 $$
 |\p_\r^j\hat S(r,k,\r)|\le \Cte \frac1\kappa(N\frac{\chi+\de}{\kappa})^{ |j |}\max_{0\le l\le j} |\p^l_\r\hat F(r,k,\r)|
 $$
 for any $ |j |\le{{s_*}}$,  where $\Cte$ is an absolute constant. 
 
 By Lemma  \ref{lExtension},
there exists a $\cC^\infty$-function $g_k:\D\to\R$, being $=1$ outside $\Sigma(L_k,\ka)$ and $=0$ on
$\Sigma(L_k,\frac\ka2)$ and such that for all $j\ge 0$
$$| g_k |_{\cC^j(\D)}\le (\Cte\frac{N(\chi+\delta)}{\ka})^j.$$
 Multiplying $\hat S(r,k,\r)$ with $g_k(\r)$ gives a $\cC^{{s_*}}$-extension of $\hat S(r,k,\r)$ from 
 $\D\setminus \Sigma(L_k,\ka)$  to  $\D$ satisfying the same bound as $\hat S(r,k,\r)$.

It follows now, by a classical argument,  that the formal solution converge and that
$| \p_\r^j  S(r,\theta,\r)|$ and $|\p_\r^j  R(r,\theta,\r)|$
fulfills the estimates of the lemma.
When summing up  the series for  $|\p_\r^j  R(r,\theta,\r)|$ we get a term $e^ {-\frac1C(\s-\s')N}$, but the factor $\frac1C$
disappears by replacing $N$ by $CN$.  

By construction $S$ and $R$ solve equation \eqref{homo1} for any $\r\in \D_1$.
\end{proof}

\subsection{The second equation}\label{s5.3} 
Concerning the second component  of the homological equation we have

\begin{lemma}\label{prop:homo3}
There exists  an absolut constant $C$ such that if \eqref{ass1} holds, then,
for any  $N\ge1$ and 
$$0<\ka\le c',$$
there exists  a closed set $\D_2=\D_2(h,\ka,N)\subset \D$,  satisfying
 $$
 \Leb (\D\setminus \D_2)\leq   C N^{\exp} 
 \big(\frac{\chi+\delta}{\delta_0}\big) (\frac{\ka}{\delta_0})^{\frac1{{s_*}}}, $$
and there exist $\Ca^{{s_*}}$-functions $S_w$ and $R_w$ on  $\T^\cA  \times \D\to Y_\ga$,  
real holomorphic in $\theta$, such that for $\r\in \D_2$
\be\label{homo2}\p_{\om(\r)} S_w(\theta,\r) -A(\r) JS_w(\theta,\r)=
f_w(\theta,\r)-R_w(\theta,\r)
\ee
and for all $(\theta,\r)\in \T^\cA_{\s'}\times \D$, $\s'<\s$, and $|j|\le {{s_*}}$
 \begin{align} \label{homo2S}
|| \p_\r^j S_w(\theta,\r)||_{\ga}\leq &
 C\frac{1}{\ka (\s-\s')^{n}}\big( N\frac{\chi+\de         }{\ka}\big)^{|j|}
  |f^T|_{\begin{subarray}{c}\s,\mu\ \ \\ \ga, \vark,\D  \end{subarray}} \\ \label{homo2R}
   ||  \p_\r^j R_w(\theta,\r)||_{\ga}\leq & C\frac{ e^{-(\s-\s')N} } {(\s-\s')^{n}}
 |f^T|_{\begin{subarray}{c}\s,\mu\ \ \\ \ga, \vark,\D  \end{subarray}}.
\end{align}

The  exponent $\exp$  only depends on  $\frac{d_*}{\beta_1}$ and $n=\#\cA$, and $C$ is an absolut constant that depends on $c$ and $\sup_\D\ab{\om}$.
\end{lemma}

\proof
Let us re-write \eqref{homo2} in the complex variables  $\xi$ and $\eta$ described in section
\ref{ssUnperturbed}. The quadratic form  $(1/2)\langle w, A(\rho) w\rangle\ $ gets transformed, by $w=Uz$,  to
$$ \langle\xi, Q(\rho)\eta\rangle + \frac12 \langle z_{\F}, H'(\rho) z_{\F}\rangle,$$ 
where $Q'$ is a Hermitian matrix and  $H'$ is a real symmetric matrix. Then we make in \eqref{homo2} the 
substitution   $ S={}^t\!US_w$, $R={}^t\! UR_w$ 
and $F={}^t\!Uf_w$, where $S={}^t(S^\xi, S^\eta,S^{\F})$, etc. 
In this notation eq.~\eqref{homo2}  decouples into the equations 
\begin{align*}
 \p_{\om}  S^\xi + {\mathbf i} QS^\xi= F^\xi-R^\xi,\\
 \p_{\om}  S^\eta-  {\mathbf i}{}^t\!QS^\eta= F^\eta-R^\eta\\
\p_{\om} S^{\F}-  HJS^{\F}= F^{\F}-R^{\F}.
\end{align*}

Let us consider the first equation. Written in the  Fourier components it becomes
\be\label{homo2.10}
( \langle k, \om(\r) \rangle I  + Q) \hat S^\xi (k)=-{\mathbf i}(\hat F^\xi(k)-\hat R^\xi(k)).\ee
This equation decomposes into its ``components'' over the  blocks $[a]=[a]_\Delta$ and takes the form
\be\label{homo2.1}L_{k,[a]}(\r)\hat S_{[a]}(k)=:
( \langle k, \om(\r) \rangle  + Q_{[a]}) \hat S_{[a]}(k)=-{\mathbf i}(\hat F_{[a]}(k)-\hat R_{[a]}(k))\ee
--  the matrix $Q_{[a]}$ being the restriction of $Q^\xi$ to $[a]\times [a]$, the vector $F_{[a]}$ being 
is the restriction of $F^\xi$ to $[a]$ etc. 

Equation  \eqref{homo2.10} has the (formal) solution
$$\hat S_{[a]}(k,\r)= 
\left\{\begin{array}{ll}
-(L_{k,[a]}(\r))^{-1}{\mathbf i}\hat F_{[a]}(k,\r) & \textrm{ if }   |k|\le N\\
0 & \textrm{ if not}  
\end{array}\right.$$
and 
$$\hat R_{a}(k,\r)= 
\left\{\begin{array}{ll}
\hat F_{a}(k,\r ) & \textrm{ if }   |k|> N\\
0& \textrm{ if not}.
\end{array}\right.$$

For $k\not=0$, by Lemma \ref{lSmallDiv2},
$$ ||(L_{k,[a]}(\r))^{-1}||\le \frac1\ka\,
 $$
for all $\r$ outside some  set $\Sigma(L_{k,[a]},\kappa)$ such that
$$\dist(\D\setminus \Sigma(L_{k,[a]},\ka),\Sigma(L_{k,[a]},\frac\ka2))\ge \cte\frac\ka{N(\chi+\delta)}$$
and
$$\D_2=\D\setminus \bigcup_{\begin{subarray}{c} 0< |k|\le N\\ [a]\end{subarray}}\Sigma_{k,[a]}(\ka),$$
fulfills the required estimate. For $k=0$, it follows by \ref{ass1} and \ref{laequiv} that
$$ ||(L_{k,[a]}(\r))^{-1}||\le \frac1{c'}\le \frac1{\ka}\, . $$

We then get, as in the proof of Lemma \ref{prop:homo12}, that $\hat S_{[a]}(k,\cdot)$ and 
$\hat R_{[a]}(k,\cdot)$ have  $\cC^{{s_*}}$-extension to $\D$ satisfying
$$|| \p_\r^j \hat S_{[a]} (k,\r)|| \leq \Cte\frac{1}{\ka}\big(N\frac{\chi+\de}{\ka}\big)^{|j|}
\max_{0\le l\le j} || \p_\r^l \hat F_{[a]}(k,\r)||
$$
and
$$||  \p_\r^j R_{[a]}(k,\r)||\leq  \Cte || \p_\r^j \hat F_{[a]}(k,\r)||,$$
and satisfying \eqref{homo2.1} for $\r\in\D_2$.

These estimates imply that
$$|| \p_\r^j \hat S^\xi(k,\r)||_\ga\leq \Cte\frac{1}{\ka}\big(N\frac{\chi+\de}{\ka}\big)^{|j|}
\max_{0\le l\le j} || \p_\r^l \hat F^\xi (k,\r)||_\ga
$$
and
$$||  \p_\r^j R^\xi (k,\r)||\leq  \Cte || \p_\r^j F^\xi(k,\r)||_\ga.$$
Summing up the Fourier series, as in Lemma \ref{prop:homo12}, we get
$$|| \p_\r^j S^\xi (\theta,\r)||_\ga\leq \Cte\frac{1}{\ka (\s-\s')^{n}}\big(N\frac{\chi+\de}{\ka}\big)^{|j|}
\max_{0\le l\le j} \sup_{|\Im\theta|<\s}|| \p_\r^l F^\xi(\cdot,\r)||_\ga
$$
and
$$||  \p_\r^j R^\xi(\theta,\r)||_\ga \leq  \Cte\frac{ e^{-\frac1{\Cte}(\s-\s')N} } {(\s-\s')^{n}}
\sup_{|\Im\theta|<\s}|| \p_\r^j F^\xi(\cdot,\r)||_\ga
$$
for $(\theta,\r)\in \T^\A_{\s'}\times \D$, $0<\s'<\s$, and $|j|\le{{s_*}}$. This implies the estimates \eqref{homo2S} and \eqref{homo2R}  --  the factor $\frac1{\Cte}$
disappears by replacing $N$ by $\Cte N$. 

The other two equations are treated in exactly the same way.

\endproof

\subsection{The third equation}\label{s5.4}
 Concerning the third component of the homological equation, \eqref{eqHomEq}, we have the following result, 
 where for a solution $S_{w w}(\theta,\rho)$ we estimate separately its mean-value
 $\hat S_{w w}(0,\rho)$ and the deviation from the mean-value 
 $S_{ww}(\theta,\rho) - \hat S_{ww}(0,\rho)$.
 
\begin{lemma}\label{prop:homo4}
There exists  an absolut constant $C$ such that if \eqref{ass1} holds, then,
for any  $N\ge1$, $\Delta'\ge \Delta\ge 1$,  and 
$$\ka\le\frac1C c',$$
there exist subsets $\D_3=\D_3(h, \ka,N,\Delta')\subset \D$, satisfying 
$$\Leb(\D\setminus {\D_3})\le C N^{\exp}\big(\frac{\chi+\delta}{\delta_0}\big) (\frac{\ka }{\de_0})^{\alpha}$$
and there exist real $\Ca^{{s_*}}$-functions 
$B_{ww}: \D\to \cM_{\ga,\varkappa} \cap \NF_{\Delta'} $ and 
$S_{ww}$, $R_{ww}:\T^{\cA}\times \D\to \cM_{\ga,\varkappa}$,
real holomorphic  in $\theta$, such that for all $\r\in\D_3$
\begin{multline}\label{homo3} 
\p_{\om(\r)} S_{ww}(\theta,\r) -A(\r)JS_{ww}(\theta,\r)+
S_{ww}(\theta,\r)JA(\r)=\\
f_{ww} (\theta,\r)-B_{ww}(\r)-R_{ww}(\theta,\r)
 \end{multline}
and for all $(\theta,\r)\in \T^\cA_{\s'}\times \D$, $\s'<\s$, and $|j|\le {{s_*}}$

\be\label{homo3S}
\aa{\p_\r^j S_{ww}(\theta,\r)}_{\ga',\varkappa}\leq 
C\Delta'\frac{\Delta^{\exp_2}e^{2\ga d_\Delta}}{\ka(\s-\s')^{n}}
\big(N\frac{\chi+\de}{\ka}\big)^{|j|}
\ab{f^T}_{\begin{subarray}{c}\s,\mu\ \ \\ \ga, \vark,\D  \end{subarray}},\ee


\be\label{homo3R}
\aa{ \p_\r^j  R_{ww}(\theta,\r)}_{\ga',\vark}\leq  
C\Delta' \Delta^{\exp_2}\left(\frac{e^{-(\s-\s')N}+e^{-(\ga-\ga')\Delta'}}{ (\s-\s')^{n}}\right)
\ab{f^T}_{\begin{subarray}{c}\s,\mu\ \ \\ \ga, \vark,\D  \end{subarray}},\ee

\be \label{homo3B}
\aa{\p_\r^j B_{ww}(\r)}_{\ga',\vark}\leq  C \Delta' \Delta^{\exp_2} \ab{f^T}_{\begin{subarray}{c}\s,\mu\ \ \\ \ga, \vark,\D  \end{subarray}},\ee
for any $\ga_*\le \ga'\le\ga$.

The exponent $\exp$ only depends  on $\frac{d_*}{\beta_1} $ and $n=\#\cA$. The exponent $\exp_2$ only depends  
on $d_*,m_*$. The exponent
$\alpha$ is a positive constant only depending on  $s_*,\frac{d_*}{\vark} ,\frac{d_*}{\beta_3} $.
$C$ is an absolut constant that depends on $c$ and $\sup_\D\ab{\om}$.

\end{lemma}

\proof It is also enough to find complex solutions $S_{ww}$, $R_{ww}$
and $B_{ww}$ verifying the estimates, because then their real parts will do the job.

As in the previous section, and using the same notation, we re-write \eqref{homo3} in complex variables. 
So we introduce 
$S={}^t\! US_{\zeta,\zeta} U$, $R={}^t\!  UR_{\zeta,\zeta} U$, $B={}^t\! UB_{\zeta,\zeta} U$ and $F={}^t\! UJf_{\zeta,\zeta} U$.
In appropriate notations  \eqref{homo3}  decouples into  the equations
\begin{align*}
&\p_{\om}  S^{\xi\xi} +{\mathbf i} Q S^{\xi\xi}+ {\mathbf i} S^{\xi\xi}\ {}^tQ= F^{\xi\xi}-B^{\xi\xi}- R^{\xi\xi},\\ 
&\p_{\om}  S^{\xi\eta}  + {\mathbf i}Q S^{\xi\eta} -  {\mathbf i}S^{\xi\eta} Q= F^{\xi\eta}-B^{\xi\eta}- R^{\xi\eta},\\
&\p_{\om}  S^{\xi z_{\F}} +{\mathbf i} Q S^{\xi z_{\F}}+  S^{\xi z_{\F}}\ JH
= F^{\xi z_{\F}}-B^{\xi\xi}- R^{\xi z_{\F}},\\ 
&\p_{\om}  S^{z_{\F}z_{\F}} 
+HJ S^{z_{\F}z_{\F}}- S^{z_{\F}z_{\F}}JH= F^{z_{\F}z_{\F}}-
B^{z_{\F}z_{\F}}- R^{z_{\F}z_{\F}},
\end{align*}
and  equations for $ S^{\eta\eta},S^{\eta\xi}, S^{z_{\F}\xi }, S^{\eta z_{\F}},S^{z_{\F} \eta}$. Since those latter equations are of the same type as the first four, we shall concentrate on these first.

\smallskip

{\it First equation. } 
 Written in the  Fourier components it becomes
\be\label{homo3.10}
( \langle k, \om(\r) \rangle I  + Q) \hat S^{\xi\xi} (k)+\hat S^{\xi\xi} (k){}^tQ
=-{\mathbf i}(\hat F^{\xi\xi}(k)-\de_{k,0} B-\hat R^{\xi\xi}(k)).\ee
This equation decomposes into its ``components'' over the  blocks $[a]\times[b]$, $[a]=[a]_\Delta$,
and takes the form
\begin{multline}\label{homo3.1}
L(k,[a],[b],\r)\hat S_{[a]}^{[b]}(k)=:\langle k, \om(\r)\rangle \ \hat S_{[a]}^{[b]}(k) + Q_{[a]}(\r) \hat S_{[a]}^{[b]}(k)+ \\
\hat S_{[a]}^{[b]}(k)\  {}^tQ_{[b]}(\r)=
-{\mathbf i}(\hat F_{[a]}^{[b]}(k,\r) -\hat R_{[a]}^{[b]}(k)-\de_{k,0} B_{[a]}^{[b]})
\end{multline}
--  the matrix $Q_{[a]}$ being the restriction of $Q^{\xi\xi}$ to $[a]\times [a]$, the vector $F_{[a]}^{[b]}$ being 
the restriction of $F^{\xi\xi}$ to $[a]\times[b]$ etc. 

Equation \eqref{homo3.10} has the (formal)  solution
$$\hat S_{[a]}^{[b]}(k,\r)=
\left\{\begin{array}{ll}
-L(k,[a],[b],\r)^{-1}{\mathbf i}\hat F_{[a]}^{[b]}(k,\r) &
\textrm{ if } \dist([a],[b])\le\Delta'\ \textrm{and}\ \  |k|\le N\\
0 &  \textrm{ if not }, 
\end{array}\right.
$$
$$\hat R_{a}^{b}(k,\r)=  
\left\{\begin{array}{ll}
\hat F_{a}^{b}(k,\r )& \textrm{ if } \dist([a],[b])\ge\Delta'\ \textrm{or}\ \ |k|>N\\
0 & \text{ if not}.
\end{array}\right.
$$

For $k\not=0$, by Lemma \ref{lSmallDiv3},
$$ ||(L_{k,[a],[b]}(\r))^{-1}||\le \frac1\ka\,
 $$
for all $\r$ outside some  set $\Sigma_{k,[a],[b]}(\kappa)$ such that
$$\dist(\D\setminus \Sigma_{k,[a],[b]}(\ka),\Sigma_{k,[a],[b]}(\frac\ka2))\ge \cte\frac\ka{N(\chi+\delta)},$$
and
$$\D_3=\D\setminus \bigcup_{\begin{subarray}{c} 0<|k|\le N\\ [a],[b]\end{subarray}}\Sigma_{k,[a],[b]}(\ka)$$
fulfills the required estimate.  For $k=0$, it follows by \ref{ass1} and \ref{laequiv} that
$$ ||(L_{k,[a],[b]}(\r))^{-1}||\le \frac1{c'}\le \frac1{\ka}\, . $$

We then get, as in the proof of Lemma \ref{prop:homo12}, that $\hat S_{[a]}^{[b]}(k,\cdot)$ and 
$\hat R_{[a]}^{[b]}(k,\cdot)$ have  $\cC^{{s_*}}$-extension to $\D$ satisfying
$$|| \p_\r^j \hat S_{[a]}^{[b]} (k,\r)|| \leq \Cte\frac{1}{\ka}\big(N\frac{\chi+\de}{\ka}\big)^{|j|}
\max_{0\le l\le j} || \p_\r^l \hat F_{[a]}(k,\r)||
$$
and
$$||  \p_\r^j R_{a}^{b}(k,\r)||\leq  \Cte || \p_\r^j \hat F_{a}(k,\r)||,$$
and satisfying \eqref{homo3.1} for $\r\in\D_3$.

These estimates imply that, for any $\ga_*\le \ga'\le\ga$,
$$|| \p_\r^j \hat S^{\xi\xi}(k,\r)||_{\cB(Y_{\ga'},Y_{\ga'})}\leq \Cte\Delta'\frac{\Delta^{\exp}e^{2\ga d_\Delta}}{\ka}\big(N\frac{\chi+\de}{\ka}\big)^{|j|}
\max_{0\le l\le j} || \p_\r^l \hat F^{\xi\xi} (k,\r)||_{\cB(Y_{\ga'},Y_{\ga'})} $$
and
$$|| \p_\r^j \hat R^{\xi\xi}(k,\r)||_{\cB(Y_{\ga'},Y_{\ga'})}\leq \Cte\Delta' \Delta^{\exp}  || \p_\r^j \hat F^{\xi\xi} (k,\r)||_{\cB(Y_{\ga'},Y_{\ga'})}.$$
The factor $\Delta^{\exp}e^{2\ga d_\Delta}$ occurs because the diameter of the blocks $\le d_{\Delta}$ interferes
with the exponential decay and influences the equivalence between the $l^1$-norm and the operator-norm.
The factor $\Delta' \Delta^{\exp} $ occurs because the truncation $\lsim \Delta'+ d_{\Delta}$ of diagonal 
influences the equivalence between the sup-norm and the operator-norm.

These estimates gives estimates for the matrix norms and, for any $\ga_*\le \ga'\le\ga$,
$$|| \p_\r^j \hat S^{\xi\xi}(k,\r)||_{\ga,\vark}\leq \Cte\frac{\Delta^{\exp}e^{2\ga d_\Delta}}{\ka}\big(N\frac{\chi+\de}{\ka}\big)^{|j|}
\max_{0\le l\le j} || \p_\r^l \hat F^{\xi\xi} (k,\r)||_{\ga,\vark} $$
and
$$||  \p_\r^j R^{\xi\xi} (k,\r)||_{\ga,\vark}\leq  \Cte || \p_\r^j F^{\xi\xi}(k,\r)||_{\ga,\vark}.$$

Summing up the Fourier series, as in Lemma \ref{prop:homo3}, we get that 
$S^{\xi\xi}(\theta,\r)$ satisfies the estimate \eqref{homo3S} and that
$R^{\xi\xi}(\theta,\r)$ satisfies the estimate \eqref{homo3R}. 

 \smallskip

{\it The third equation. } 
We  write the equation in  Fourier components and decompose it into its
 ``components'' on each product block $[a]\times[b]$, $[b]=\F$:
\begin{multline*}
L(k,[a],[b],\rho) \hat S_{[a]}^{[b]}(k) := \langle k, \om(\r)\rangle \ 
\hat S_{[a]}^{[b]}(k) +Q_{[a]}(\r)\hat S_{[a]}^{[b]}(k) -\\
 {\mathbf i}\hat S_{[a]}^{[b]}(k) JH(\r)= 
-{\mathbf i}( \hat F_{[a]}^{[b]}(k,\r)-\delta_{k,0}B_{[a]}^{[b]}-\hat R_{[a]}^{[b]}(k))
\end{multline*}
--  here we have suppressed the upper index ${\xi z_{\F}}$.

The formal solution is the same as in the previous case and it converges to functions verifying
\eqref{homo3S}, \eqref{homo3R}  and \eqref{homo3B}, by Lemma \ref{lSmallDiv3}, and by \eqref{la-lb-bis}.

 \smallskip

{\it The fourth equation. } 
We  write the equation in  Fourier components:
\begin{multline*}
L(k,[a],[b],\rho) \hat S_{[a]}^{[b]}(k) := \langle k, \om(\r)\rangle \ 
\hat S_{[a]}^{[b]}(k) -  {\mathbf i}HJ(\r)\hat S_{[a]}^{[b]}(k) +\\
 {\mathbf i}\hat S_{[a]}^{[b]}(k) JH(\r)= 
-{\mathbf i}( \hat F_{[a]}^{[b]}(k,\r)-\delta_{k,0}B_{[a]}^{[b]}-\hat R_{[a]}^{[b]}(k)),
\end{multline*}
where $[a]=[b]=\F$  --  here we have suppressed the upper index ${z_{\F} z_{\F}}$.

The equation is solved (formally) by 
$$\hat S_{[a]}^{[b]}(k,\r)= 
\left\{\begin{array}{ll}
-L(k,[a],[b],\r)^{-1} {\mathbf i}\hat F_{[a]}^{[b]}(k,\r) & \textrm{ if }   0<|k|\le N\\
0 & \textrm{ if not} ,
\end{array}\right.
$$
$$\hat R_{[a]}^{[b]}(k,\r)= 
\left\{\begin{array}{ll}
\hat F_{[a]}^{[b]}(k,\r ) & \textrm{ if }   |k|> N\\
0& \textrm{ if not};
\end{array}\right.
$$
and 
$$B^{[b]}_{[a]}(\r)=\hat F_{[a]}^{[b]}(0,\r).
$$

The formal solution now converges a  solution  verifying
\eqref{homo3S}, \eqref{homo3R} and \eqref{homo3B} by Lemma \ref{lSmallDiv3}.

 \smallskip

{\it The second equation. } 
We  write the equation in  Fourier components and decompose it into its
 ``components'' on each product block $[a]\times[b]$:
\begin{multline*}
L(k,[a],[b],\r)\hat S_{[a]}^{[b]}(k)=:\langle k, \om(\r)\rangle \ \hat S_{[a]}^{[b]}(k) + Q_{[a]}(\r) \hat S_{[a]}^{[b]}(k)- \\
\hat S_{[a]}^{[b]}(k)Q_{[b]}(\r)=
-{\mathbf i}(\hat F_{[a]}^{[b]}(k,\r) -\hat R_{[a]}^{[b]}(k)-\de_{k,0} B_{[a]}^{[b]})
\end{multline*}
--  here we have suppressed the upper index $\xi\eta$.
This equation is now  solved (formally) by 
$$
S_{[a]}^{[b]}(\theta,\r)=\sum\hat S_{[a]}^{[b]} (k,\r) e^{{\mathbf i}k\cdot \theta}
\quad\textrm{and}\quad
R_{[a]}^{[b]}(\theta,\r) =\sum\hat R_{[a]}^{[b]} (k,\r) e^{{\mathbf i}k\cdot \theta},$$
with

$$\hat S_{[a]}^{[b]}(k,\r)=
\left\{\begin{array}{ll}
L(k,[a],[b],\r)^{-1}{\mathbf i}\hat F_{[a]}^{[b]}(k,\r) &
\textrm{ if } \dist([a],[b])\le\Delta'\ \textrm{and}\ \ 0< |k|\le N\\
0 &  \textrm{ if not }, 
\end{array}\right.
$$
$$\hat R_{a}^{b}(k,\r)=  
\left\{\begin{array}{ll}
\hat F_{a}^{b}(k,\r )& \textrm{ if } \dist([a],[b])\ge\Delta'\ \textrm{or}\ \ |k|>N\\
0 & \text{ if not}.
\end{array}\right.
$$
 and
$$B_{a}^{b}(\r)=  
\left\{\begin{array}{ll}
\hat F_{a}^{b}(0,\r ) & \textrm{ if } \dist([a],[b])\le\Delta'\ \textrm{and}\ \ k=0\\
0& \text{ if not}.
\end{array}\right.
$$

We have to distinguish two cases, depending on when $k= 0$ or not.

\smallskip

{\it The case $k\not=0$. }  

We have, by Lemma \ref{lSmallDiv3},
$$ ||(L_{k,[a],[b]}(\r))^{-1}||\le \frac1\ka\,
 $$
for all $\r$ outside some  set $\Sigma_{k,[a],[b]}(\kappa)$ such that
$$\dist(\D\setminus \Sigma_{k,[a],[b]}(\ka),\Sigma_{k,[a],[b]}(\frac\ka2))\ge \cte\frac\ka{N(\chi+\delta)},
$$
and
$$\D_3=\D\setminus \bigcup_{\begin{subarray}{c} 0<|k|\le N\\ [a],[b]\end{subarray}}\Sigma_{k,[a],[b]}(\ka)$$
fulfills the required estimate.  

\smallskip

{\it The case $k=0$.}
In this case we consider the block decomposition $\E_{\Delta'}$  and  we distinguish  
whether $|a|=|b|$ or not.

If  $|a|>|b|$, we use   \eqref{ass1} and \eqref{la-lb-bis} to get 
$$|\alpha(\r)-\beta(\r)|\geq c'-\frac{\de}{\langle a\rangle^\varkappa}
-\frac{\de}{\langle b\rangle^\varkappa}\geq \frac{c'}{2}\ge \ka.$$

This estimate allows us to solve the equation by choosing 
$$B_{[a]}^{[b]}=\hat R_{[a]}^{[b]}(0) =0$$
and
$$\hat S_{[a]}^{[b]}(0,\r)= L(0,[a],[b],\r)^{-1}\hat F_{[a]}^{[b]}(0,\r)$$
with 
$$
||\p_\r^j\hat S_{[a]}^{[b]}(0,\r)||\le\Cte  \frac{1}{\ka} (N\frac{\chi+\de}{\kappa})^{ |j |}
\max_{0\le l\le j}\aa{ \p_\r^l\hat F_{[a]}^{[b]}(0,\r)},$$
which implies \eqref{homo3S}.

If $|a|=|b|$, we cannot control $|\alpha(\r)-\beta(\r)|$ from below, 
so then we define
$$\hat S_{[a]}^{[b]}(0)=0$$
and
\begin{align*}
B_{a}^{b}(\r)=\hat F_{a}^{b}(0,\r)) ,\quad \hat R_{a}^{b}(0)  =0\quad &\text{for }  [a]_{\Delta'}= [b]_{\Delta'}\\
 \hat R_{a}^{b}(0,\r)  =\hat F_{a}^{b}(0,\r)\quad B_{a}^{b}=0,\quad& \text{for }  [a]_{\Delta'}\not= [b]_{\Delta'}.
\end{align*}
Clearly $R$ and $B$ verifiy the estimates \eqref{homo3R} and \eqref{homo3B}. 

\smallskip

Hence, the formal solution converges to functions verifying
\eqref{homo3S}, \eqref{homo3R} and \eqref{homo3B} by Lemma \ref{lSmallDiv3}.
Moreover, for $\r\in \D'$, these functions are a solution of the fourth equation.

 \endproof
 
\subsection{The homological equation}

\begin{lemma}\label{thm-homo}
There exists  a constant $C$ such that if \eqref{ass1} holds, then,
for any  $N\ge1$, $\Delta'\ge \Delta\ge 1$  and 
$$\ka\le\frac1C c',$$
there exist subsets $\D'=\D(h, \ka, N)\subset \D$, satisfying 
$$\Leb(\D\setminus {\D'})\le C N^{\exp_1}\big(\frac{\chi+\delta}{\delta_0}\big) (\frac{\ka}{\de_0})^{\alpha}$$
and there exist real jet-functions 
$S, R\in \Tc_{\ga,\varkappa,\D}(\s,\mu)$ and $h_+$
verifying, for $\r\in\D'$,
\be\label{eqHomEqbis}
\{ h,S \}+ f^T= h_++R,\ee
and such that
$$h+h_+\in\NF_{\varkappa}(\Delta',\de_+)$$
and, for all  $0<\s'<\s$,

\be\label{estim-B}
\ab{ h_+}_{\begin{subarray}{c}\s',\mu\ \ \\ \ga, \varkappa,\D  \end{subarray}}\le  X
\ab{f^T}_{\begin{subarray}{c}\s,\mu\ \ \\ \ga, \varkappa,\D  \end{subarray}}
\ee

\begin{equation}\label{estim-S}
\ab{S}_{\begin{subarray}{c}\s',\mu\ \ \\ \ga, \varkappa,\D  \end{subarray}}\\
\leq \frac{1}\ka X (N\frac{\chi+\delta}{\ka})^{{{s_*}}}
\ab{f^T}_{\begin{subarray}{c}\s,\mu\ \ \\ \ga, \varkappa,\D  \end{subarray}}
\end{equation}
  
\be\label{estim-R}
\ab{R}_{\begin{subarray}{c}\s',\mu\ \ \\ \ga', \varkappa,\D  \end{subarray}}\leq 
X\left( e^{-(\s-\s')N}+e^{-(\ga-\ga')\Delta'}\right)
\ab{f^T}_{\begin{subarray}{c}\s,\mu\ \ \\ \ga, \varkappa,\D  \end{subarray}},\ee
for $\ga_*\le\ga'\le \ga$, 
where
$$
X=C\Delta'\big(\frac{\Delta}{\s-\s' }\big)^{\exp_2} e^{2\ga d_\Delta}\max(1,\mu^2).$$

The exponent $\exp_1$ only depends  on $\frac{d_*}{\beta_1} $ and $\#\cA$. The exponent $\exp_2$ only depends  
on $d_*,m_*$  and $\#\cA$. The exponent
$\alpha$ is a positive constant only depending on  $s_*,\frac{d_*}{\vark} ,\frac{d_*}{\beta_3} $.
$C$ is an absolut constant that depends on $c$ and $\sup_\D\ab{\om}$.

\end{lemma}

\begin{remark}
The estimates \eqref{estim-B} provides an estimate of $\de_+$. Indeed, for any $a,b\in [a]_{\Delta'}$
$$
\ab{\p_\r^j B_a^b}\le  \frac1 C||\p_\r^j  B||_{\ga,\vark}e_{\ga,\vark}(a,b)^{-1}\le \Cte  (\Delta')^{\vark} 
\ab{f^T}_{\begin{subarray}{c}\s,\mu\ \ \\ \ga, \varkappa,\D  \end{subarray}} \frac1{\langle a\rangle^{\vark}}.
$$
Since $\# [a]_{\Delta'}\le (\Delta')^{\exp}$ we get
$$
 || \p_\r^j B(\r)_{[a]_{\Delta'}} || \le  \Cte(\Delta')^{\exp} 
\ab{f^T}_{\begin{subarray}{c}\s,\mu\ \ \\ \ga, \varkappa,\D  \end{subarray}} \frac1{\langle a\rangle^{\vark}}.
$$
This gives the estimate of $\delta_+-\delta$.
\end{remark}

\begin{proof}
The set $\D'$ will now be given by the intersection of the sets in the three previous lemma of this section.
We set 
$$ h_+(r,w)=\hat f_r(r,0) + \frac 1 2 \langle w, Bw\rangle$$
$$
S(r,\theta,w)=S_r(\theta,r)+\langle S_w(\theta)w\rangle+
\frac 1 2 \langle S_{ww}(\theta)w,w \rangle$$ 
and
$$
R(r,\theta,w)= R_r(r,\theta)+\langle R_w(\theta),w\rangle+
\frac 1 2 \langle R_{ww}(\theta)w,w \rangle.$$
These functions also depend on $\r\in\D$ and they verify equation \eqref{eqHomEqbis}
for $\r\in\D'$.

If $x=(r,\theta,w)\in \O_{\ga_*}(\s,\mu)$, then 
$$| h_+(x)|\le \ab{f^T}_{\begin{subarray}{c}\s,\mu\ \ \\ \ga, \varkappa,\D  \end{subarray}}
+  \frac 1 2  || Bw ||_{\ga_*}||w ||_{\ga_*}.$$
Since
$$\aa{B}_{\ga,\vark} \ge \aa{B}_{\ga_*,\vark}\ge \aa{B}_{\cB(Y_{\ga_*},Y_{\ga_*})}$$
it follows that
$$| h_+(x)|\le  \Cte \ab{f^T}_{\begin{subarray}{c}\s,\mu\ \ \\ \ga, \varkappa,\D  \end{subarray}}\max(1,\mu^2).$$

We also have for any $x=(r,\theta,w)\in \O_{\ga'}(\s,\mu)$, $\ga_*\le \ga'\le\ga$, 
$$||Jd  h_+(x)||_{\ga'}\le 
 \Cte \ab{f^T}_{\begin{subarray}{c}\s,\mu\ \ \\ \ga, \varkappa,\D  \end{subarray}}+ 
||Bw||_{\ga'}.$$
Since
$$\aa{B}_{\ga,\vark} \ge \aa{B}_{\ga',\vark}\ge \aa{B}_{\cB(Y_{\ga'},Y_{\ga'})}$$
it follows that
$$||Jd  h_+(x)||_{\ga'}\le
\Cte \ab{f^T}_{\begin{subarray}{c}\s,\mu\ \ \\ \ga, \varkappa,\D  \end{subarray}}\max(1,\mu).$$
Finally $Jd^2  h_+(x)$ equals $ JB$ which satisfies the required bound.

The estimates of the derivatives with respect to $\r$  are the same and obtained in the same way.

The functions $S(\theta,r,\zeta)$ and $R(\theta,r,\zeta)$ 
are estimated in the same way.

\end{proof}

\subsection{The non-linear homological equation}

The equation \eqref{eqNlHomEq} can now be solved easily.

\begin{proposition}\label{thm-Eq}
There exists a constant $C$ such that for any
$$h\in\NF_{\varkappa}(\Delta,\de),\quad\delta \le \frac1C c',$$
and for any 
$$N\ge 1,\quad \Delta'\ge \Delta\ge 1,\quad \ka\le\frac1C c'$$ 
there exists a subset $\D'=\D(h, \ka,N)\subset \D$, satisfying 
$$\Leb(\D\setminus {\D'})\le C N^{\exp_1}\big(\frac{\chi+\delta}{\delta_0}\big)
 (\frac{\ka}{\de_0})^{\alpha},$$
and, for any $f\in \Tc_{\ga,\varkappa}(\s,\mu,\D)$, $\mu\le1$,
$$\eps=\ab{f^T}_{\begin{subarray}{c}\s,\mu\ \ \\ \ga, \vark,\D  \end{subarray}}\quad
\textrm{and}\quad 
\xi=\ab{f}_{\begin{subarray}{c}\s,\mu\ \ \\ \ga, \vark,\D  \end{subarray}},$$
there exist real jet-functions 
$S,R\in\Tc_{\ga,\varkappa,\D}(\s,\mu)$ and $h_+$ 
verifying, for $\r\in\D'$,
\be\label{eqNlHomEqbis}
\{ h,S \}+\{ f-f^T,S \}^T+ f^T=  h_++R\ee
and such that
$$h+ h_+\in\NF_{\varkappa}(\Delta',\delta_+)$$
and, for all $ \s'<\s$ and $\mu'<\mu$,

\be\label{estim-B2}
\ab{ h_+}_{\begin{subarray}{c}\s',\mu\ \ \\ \ga, \varkappa,\D  \end{subarray}}\le  CXY\eps\ee

\be\label{estim-S2}
\ab{S}_{\begin{subarray}{c}\s',\mu\ \ \\ \ga, \varkappa,\D  \end{subarray}}\\
\leq C \frac1\ka XY\eps\ee
 
\be\label{estim-R2}
\ab{R}_{\begin{subarray}{c}\s',\mu\ \ \\ \ga', \varkappa,\D  \end{subarray}}\leq 
C\left( e^{-(\s-\s')N}+e^{-(\ga-\ga')\Delta'}\right)XY\eps,\ee
for $\ga_*\le\ga'\le \ga$, 
where
$$X=(\frac{N\Delta' e^{\ga  d_\Delta} }{(\s-\s')(\mu-\mu')})^{\exp_2}$$
and
$$
Y=(\frac{\chi+\delta+\xi}{\ka})^{4{s_*}+3}.$$

The exponent $\exp_1$ only depends  on $\frac{d_*}{\beta_1} $ and $\#\cA$. The exponent $\exp_2$ only depends  
on $d_*,m_*,s_*$  and $\#\cA$. The exponent
$\alpha$ is a positive constant only depending on  $s_*,\frac{d_*}{\vark} ,\frac{d_*}{\beta_3} $.
$C$ is an absolut constant that depends on $c$ and $\sup_\D\ab{\om}$.

\end{proposition}

\begin{remark}
Notice that the ``loss'' of $S$ with respect to $\ka$ is of ``order''  $4{{s_*}}+3$.
However, if $\chi$, $\de$  and 
$\xi= \ab{f}_{\begin{subarray}{c}\s,\mu\ \ \\ \ga', \varkappa,\D  \end{subarray}}$ are of size $\lsim \ka$, then the loss is only of ``order'' 1.
\end{remark}

\begin{proof}
Let $S=S_0+S_1+S_2$ be a jet-function such that $S_1$ starts with terms of degree $1$ in $r,w$
and $S_2$ starts with terms of degree $2$ in $r,w$  --  jet functions are polynomials in $r,w$ and we
give (as is usual) $w$ degree $1$ and $r$ degree $2$.

Let now $\s'=\s_5<\s_4<\s_3<\s_2<\s_1<\s_0=\s$ be a (finite) arithmetic progression, i.e.
$\s_j-\s_{j+1}$ do not depend on $j$,  and let
  and $\mu'=\mu_5<\mu_4<\mu_3<\mu_2<\mu_1<\mu_0=\mu$  be another  arithmetic progressions.

Then $\{ h',S\}+\{ f-f^T,S \}^T+ f^T=  h_++R$ decomposes into three homological equations
$$\{ h',S_0 \}+ f^T= ( h_+)_{0}+R_0,$$
$$\{ h',S_1 \}+f_1^T= ( h_+)_{1}+R_1,\quad f_1=\{ f-f^T,S_0 \},$$
$$\{ h',S_2 \}+ f_2^T= ( h_+)_{2}+R_2,\quad f_2=\{ f-f^T,S_1 \}.$$

By Lemma \ref{thm-homo} we have for the first equation
$$\ab{( h_+)_0}_{\begin{subarray}{c}\s_1,\mu\ \ \\ \ga, \varkappa,\D  \end{subarray}}\le  X\eps,
\quad
\ab{R_0}_{\begin{subarray}{c}\s_1,\mu\ \ \\ \ga', \varkappa,\D  \end{subarray}}\leq 
XZ\eps,
$$

$$\ab{S_0}_{\begin{subarray}{c}\s_1,\mu\ \ \\ \ga, \varkappa,\D  \end{subarray}}\\
\leq \frac{1}\ka X Y\eps
$$
where
$$
X=C\Delta'\big(\frac{5\Delta}{\s-\s' }\big)^{\exp} e^{2\ga_1 d_\Delta}.$$
and where $Y,Z$ are defined by the right hand sides in the estimates  \eqref{estim-S} and  \eqref{estim-R}.
 
By Proposition \ref{lemma:poisson} we have
$$
\xi_1=\ab{f_1}_{\begin{subarray}{c}\s_2,\mu_2\ \ \\ \ga, \varkappa,\D  \end{subarray}}\le
\frac{1}\ka X YW\xi \eps$$
where
$$W=C \big(\frac5{(\sigma-\sigma')}  +   \frac5{ (\mu-\mu') }\big).$$
By Proposition \ref{lemma:jet} 
$\eps_1=\ab{f_1^T}_{\begin{subarray}{c}\s_2,\mu_2\ \ \\ \ga, \varkappa,\D  \end{subarray}}$
satisfies the same bound as $\xi_1$

By Lemma \ref{thm-homo} we have for the second equation
$$\ab{( h_+)_1}_{\begin{subarray}{c}\s_3,\mu_2\ \ \\ \ga, \varkappa,\D  \end{subarray}}\le  X\eps_1,\quad
\ab{R_1}_{\begin{subarray}{c}\s_3,\mu_2\ \ \\ \ga', \varkappa,\D  \end{subarray}}\leq 
XZ\eps_1,
$$

$$\ab{S_1}_{\begin{subarray}{c}\s_3,\mu_2\ \ \\ \ga, \varkappa,\D  \end{subarray}}\\
\leq \frac{1}\ka X Y\eps_1.
$$

By Propositions \ref{lemma:jet} and \ref{lemma:poisson} we have
$$
\xi_2=\ab{f_2}_{\begin{subarray}{c}\s_4,\mu_4\ \ \\ \ga, \varkappa,\D  \end{subarray}}\le
\frac1\ka X YW\xi_1 \eps_1,$$
and $\eps_2=\ab{f_2^T}_{\begin{subarray}{c}\s_4,\mu_4\ \ \\ \ga, \varkappa,\D  \end{subarray}}$
satisfies the same bound.

By Lemma \ref{thm-homo} we have for the third equation
$$\ab{( h_+)_2}_{\begin{subarray}{c}\s_5,\mu_4\ \ \\ \ga, \varkappa,\D  \end{subarray}}\le  X\eps_2,\quad
\ab{R_2}_{\begin{subarray}{c}\s_5,\mu_4\ \ \\ \ga', \varkappa,\D  \end{subarray}}\leq 
XZ\eps_2,
$$

$$\ab{S_2}_{\begin{subarray}{c}\s_5,\mu_4\ \ \\ \ga, \varkappa,\D  \end{subarray}}\\
\leq \frac{1}\ka X Y\eps_2.
$$
 
Putting this together we find that
$$\eps+\eps_1+\eps_2 \le (1+\frac1\ka X YW\xi)^3\eps= T\eps$$
and
$$\ab{h_+}_{\begin{subarray}{c}\s',\mu'\ \ \\ \ga, \varkappa,\D  \end{subarray}}\le  
XT\eps,\quad
\ab{R}_{\begin{subarray}{c}\s',\mu'\ \ \\ \ga', \varkappa,\D  \end{subarray}}\leq 
XZT\eps,
$$

$$\ab{S}_{\begin{subarray}{c}\s',\mu'\ \ \\ \ga, \varkappa,\D  \end{subarray}}\\
\leq \frac1\ka X YT\eps.
$$

Renaming $X$ and $Y$ gives now the estimates. 
\end{proof}

\section{Proof of the KAM Theorem}


Theorem \ref{main} is proved by an infinite sequence of change of variables typical for KAM-theory.
The change of variables will be done by the classical Lie transform method
which is based on a well-known relation between composition of a function with a Hamiltonian flow
$\Phi^t_S$ and Poisson brackets:
$$\frac{d}{d t} f\circ \Phi^t_S=\{f,S\}\circ \Phi^t_S
$$
from which we derive
$$f \circ \Phi^1_S=f+\{f,S\}+\int_0^1 (1-t)\{\{f,S\},S\}\circ \Phi^t_S\ \dd t.$$
Given now three functions $ h,k$ and $f$. Then
\begin{multline*}( h+k+f )\circ \Phi^1_S=\\
 h+k+f +\{ h+k+f ,S\}+\int_0^1 (1-t)\{\{ h+k+f ,S\},S\}\circ \Phi^t_S\ \dd t.
\end{multline*}
If now $S$ is a solution of the equation
\be \label{eq-homobis}
\{  h,S \}+\{ f-f^T,S \}^T+ f^T= h_++R,
\ee
then 
$$
( h+k+f )\circ \Phi^1_S= h+k+h_++f_+$$
with
\begin{multline}\label{f+}
f_+= R+(f-f^T)+\{k+f^T ,S\}+\{f-f^T,S\}-\{f-f^T,S\}^T+\\ +\int_0^1 (1-t)\{\{ h+k+f ,S\},S\}\circ \Phi^t_S\ \dd t
\end{multline}
and
\be \label{f+T}
f_+^T= R+ \{k+f^T ,S\}^T+(\int_0^1 (1-t)\{\{h+k+f ,S\},S\}\circ \Phi^t_S\ \dd t)^T.
\ee

If we assume that $S$ is ``small as''  $f^T$, then  $f_+^T$ is is ``small as'' $k f^T$ --  this is the basis
of a linear iteration scheme with (formally) linear convergence.
\footnote{\ it was first used by Poincar\'e, credited by him to the astronomer Delauney, and it has been used many times since then in
different contexts. } 
But if also $k$ is of the size $f^T$, then $f^+$ is is ``small as'' the square of $f^T$  --  this is the basis
of a quadratic iteration scheme with (formally) quadratic convergence. We shall combine both of them.

First we shall give a rigorous version of the change of variables described above.

\subsection{The basic step}
Let $h\in\NF_{\varkappa}(\Delta,\de)$ and  assume  $\vark>0$ and 
 \be\label{ass1}
 \delta \le \frac{1}{C}c' ,\ee
 where $C$ is to be determined. 
 
Let
$$\ga=(\ga,m_*)\ge \ga_*=(0,m_*)$$
and recall Remark \ref{rAbuse}. Let  $N\ge 1$, $\Delta'\ge \Delta\ge 1$ and
$$\ka\le\frac1C c'.$$ 

Proposition \ref{thm-Eq} then gives, for any $f\in \Tc_{\ga,\varkappa,\D}(\s,\mu)$, $\mu\le1$,
$$\eps=\ab{f^T}_{\begin{subarray}{c}\s,\mu\ \ \\ \ga, \vark,\D  \end{subarray}}\quad
\textrm{and}\quad 
\xi=\ab{f}_{\begin{subarray}{c}\s,\mu\ \ \\ \ga, \vark,\D  \end{subarray}},$$
a set
$\D'=\D'(h, \ka,N)\subset \D$
and functions $S,h_+, R$ satisfying \eqref{estim-B2}+\eqref{estim-S2}+\eqref{estim-R2}
and solving  the equation \eqref{eq-homobis},
$$\{h,S \}+\{ f-f^T,S \}^T+ f^T= h_++R,$$
for any $\r\in\D'$. Let now $0<\s'=\s_4<\s_3<\s_2<\s_1<\s_0=\s$  and 
$0<\mu'=\mu_4<\mu_3<\mu_2<\mu_1<\mu_0=\mu$ be  (finite) arithmetic progressions.

\medskip

{\it The flow $\Phi^t_S$.}

We have, by \eqref{estim-S2},
$$\ab{S}_{\begin{subarray}{c}\s_1,\mu_1\ \ \\ \ga, \varkappa,\D  \end{subarray}}\\
\leq \Cte \frac1\ka XY\eps$$
where $X,Y$ and $\Cte$ are given in Proposition \ref{thm-Eq}, i.e. 
$$X=(\frac{\Delta' e^{\ga  d_\Delta}N}{(\s_0-\s_1)(\mu_0-\mu_1)})^{\exp_2}=
(\frac{4^2\Delta' e^{\ga  d_\Delta}N}{(\s-\s')(\mu-\mu')})^{\exp_2},$$
$$Y=(\frac{\chi+\delta+\xi}{\ka})^{4{{s_*}}+3}$$
--  we can assume without restriction that $\exp_2\ge 1$.

If 
\be\label{hyp-f1}
\eps\leq \frac1C \frac{\ka}{X^2Y}, \ee
and $C$ is sufficiently large, then we can apply Proposition \ref{Summarize}(i).
By this proposition  it follows  that for  any $0\le t\le1$ the Hamiltonian
flow map  $\Phi^t_S$ is a $\cC^{{s_*}}$-map 
$$\O_{\ga'}(\s_{i+1},\mu_{i+1})\times \D\to\O_{\ga'}(\s_i,\mu_i),\quad \forall \ga_*\le\ga'\le\ga,\quad i=1,2,3, $$
 real holomorphic and symplectic for any fixed  $\rho\in \D$.
Moreover,
$$|| \p_\r^j (\Phi^t_S(x,\cdot)-x)||_{\ga'}\le \Cte\frac1\ka XY \eps$$
and
$$\aa{ \p_\r^j (d\Phi^t_S(x,\cdot)-I)}_{\ga',\vark}\le \Cte\frac1\ka XY \eps$$
for any $x\in \O_{\ga'}(\s_2,\mu_2)$, $\ga_*\le \ga'\le\ga$,  and $0\le \ab{j}\le {s_*}$.

\medskip

{\it A transformation.}

Let now $k\in \Tc_{\ga,\varkappa,\D}(\s,\mu)$  and set
$$\eta=\ab{k}_{\begin{subarray}{c}\s,\mu\ \ \\ \ga, \vark,\D  \end{subarray}}.$$
 Then we have
$$(h+k+f )\circ \Phi^1_S= h+k+h_++f_+$$
where $f_+$ is defined by \eqref{f+}, i.e.
\begin{multline*}
f_+= R+(f-f^T)+\{k+f^T ,S\}+\{f-f^T,S\}-\{f-f^T,S\}^T+\\ +\int_0^1 (1-t)\{\{ h+k+f ,S\},S\}\circ \Phi^t_S\ \dd t.
\end{multline*}
The integral term is the sum
$$
\int_0^1 (1-t)\{h_++R-f^T,S\}\circ \Phi^t_S\ \dd t
+\int_0^1 (1-t)\{\{k+f ,S\}-\{f-f^T,S\}^T,S\}\circ \Phi^t_S\ \dd t.$$

\medskip

{\it The estimates of $\{k+f^T,S \}$ and $\{f-f^T,S \}$.} 

By Proposition \ref{lemma:poisson}(i)
$$
\ab{\{k+f^T,S\}}_{\begin{subarray}{c} \s_2,\mu_2 \  \\  \ga, \alpha, \D \end{subarray}}
\leq  \Cte X
\ab{S}_{\begin{subarray}{c}\s_1,\mu_1\ \ \\ \ga, \varkappa,\D  \end{subarray}}
 \ab{k+f^T}_{\begin{subarray}{c} \s_1,\mu_1 \  \\  \ga, \alpha, \D \end{subarray}}.$$
 Hence
 \be
\ab{\{k+f^T,S\}}_{\begin{subarray}{c} \s_2,\mu_2 \  \\  \ga, \alpha, \D \end{subarray}}\leq
\Cte  \frac1\ka X^2Y(\eta+\eps) \eps.\ee

Similarly,
\be
 \ab{\{f-f^T,S\}}_{\begin{subarray}{c} \s_2,\mu_2 \  \\  \ga, \alpha, \D \end{subarray}}\leq
\Cte \frac1\ka X^2Y(\xi+\eps)\eps.\ee

\medskip

{\it The estimate of $\{h_+-f^T,S \}\circ\Phi^t_S$.} 

The estimate of $h_+$  is given by \eqref{estim-B2}:
$$\ab{ h_+}_{\begin{subarray}{c}\s_1,\mu_1\ \ \\ \ga, \varkappa,\D  \end{subarray}}\le  \Cte XY\eps.$$
This gives, again by Proposition \ref{lemma:poisson}(ii),
$$
\ab{\{h_+-f^T,S\}}_{\begin{subarray}{c} \s_2,\mu_2 \  \\  \ga, \alpha, \D \end{subarray}}\leq 
\Cte \frac1\ka X^3Y^2\eps^2.$$

Let now $F=\{h_+-f^T,S \}$.  If $\eps$ verifies \eqref{hyp-f1} for
a  sufficiently large constant $C$, then we can apply Proposition \ref{Summarize}(ii). By this
proposition, for $\ab{t}\le1$, the function
$F\circ \Phi_S^t\in \Tc_{\ga,\vark,\D}(\s_3,\mu_3)$ and
\be
\ab{\{h_+-f^T,S \}\circ \Phi_S^t}_{\begin{subarray}{c} \s_3,\mu_3 \  \\  \ga, \vark, \D \end{subarray}}\leq 
\Cte \frac1\ka X^3Y^2\eps^2.\ee

\medskip

{\it The estimate of $\{R,S \}\circ\Phi^t_S$.} 

The estimate of  $R$ is given by \eqref{estim-R2}:

$$\ab{R}_{\begin{subarray}{c}\s_2,\mu_1\ \ \\ \ga', \varkappa,\D  \end{subarray}}\leq 
\Cte XYZ_{\ga'}\eps,$$
where
$$Z_{\ga'}=\left( e^{-(\s-\s_2)N}+e^{-(\ga-\ga')\Delta'}\right).$$
Then, as in the previous case,
\be
\ab{\{R,S \}\circ \Phi_S^t}_{\begin{subarray}{c} \s_3,\mu_3 \  \\  \ga', \vark, \D \end{subarray}}\leq 
\Cte \frac1\ka X^3Y^2Z_{\ga'}\eps^2.\ee

\medskip

{\it The estimate of $\{\{k+f,S \}-\{f-f^T,S\}^T,S\}\circ\Phi^t_S$.} 

This function is estimated as above. If $F=\{\{k+f,S \}-\{f-f^T,S\}^T,S\}$, then, by 
Proposition \ref{lemma:jet} and Proposition \ref{lemma:poisson}(i),
$$
\ab{F}_{\begin{subarray}{c} \s_3,\mu_3 \  \\  \ga, \alpha, \D \end{subarray}}\leq 
\Cte (\frac1\ka X^2Y)^2(\eta+\xi)\eps^2$$
and by Proposition \ref{Summarize}(ii)
\be
\ab{\{\{k+f,S \}-\{f-f^T\}^T,S\}\circ \Phi_S^t}_{\begin{subarray}{c} \s_4,\mu_4 \  \\  \ga, \vark, \D \end{subarray}}\leq
\Cte (\frac1\ka X^2Y)^2(\eta+\xi)\eps^2 .\ee

Renaming now $X$ and $Y$ and replacing $N$ by $2N$ now gives the following lemma.

\begin{lemma}\label{basic}
There exists an absolute constant $C$ such that, for any
$$h\in\NF_{\varkappa}(\Delta,\de),\quad\vark>0,\quad \delta \le \frac1C c',$$
and for any 
$$N\ge 1,\quad \Delta'\ge \Delta\ge 1,\quad \ka\le\frac1C c',$$ 
there exists a subset $\D'=\D( h, \ka,N)\subset \D$, satisfying 
$$\Leb(\D\setminus {\D'})\le C N^{\exp_1}
\big(\frac{\chi+\delta}{\delta_0}\big) (\frac{\ka}{\de_0})^{\alpha},$$
and, for any $f\in \Tc_{\ga,\varkappa,\D}(\s,\mu)$,  $\mu\le1$,
$$\eps=\ab{f^T}_{\begin{subarray}{c}\s,\mu\ \ \\ \ga, \vark,\D  \end{subarray}}\quad \textrm{and}\quad
\xi=[f]_{\s,\mu,\D}^{\ga,\varkappa},$$
satisfying
$$\eps \leq \frac1C \frac{\ka}{XY}, \qquad 
\left\{\begin{array}{ll}
X=(\frac{N\Delta' e^{\ga  d_\Delta}}{(\s-\s')(\mu-\mu')})^{\exp_2},& \s'<\s,\ \mu'<\mu\\
Y= (\frac{\chi+\delta+\xi}\ka)^{\exp_3},&\ \end{array}\right. 
$$
and for any  $k\in \Tc_{\ga,\varkappa,\D}(\s,\mu)$,
there exists a  $\cC^{{s_*}}$ mapping
$$\Phi:\O_{\ga'}(\s',\mu')\times \D\to\O_{\ga'}(\s-\frac{\s-\s'}2,\mu-\frac{\mu-\mu'}2),\quad \forall \ga_*\le\ga'\le\ga,$$
real holomorphic and symplectic  for each fixed parameter $\r\in\D$, and functions
$f_+,R_+\in \Tc_{\ga,\varkappa,\D}(\s',\mu')$ and
$$h+h_+\in\NF_{\varkappa}(\Delta',\delta_+),$$
such that
$$(h+k+f )\circ \Phi= h+k+ h_++f_++R_+,\quad \forall \r\in\D',$$
and
$$\ab{ h_+}_{\begin{subarray}{c}\s',\mu'\ \ \\ \ga, \varkappa,\D  \end{subarray}}
+\ab{ f_+-f}_{\begin{subarray}{c}\s',\mu'\ \ \\ \ga, \varkappa,\D  \end{subarray}}\le 
 CXY\eps,$$
$$\ab{ f_+^T}_{\begin{subarray}{c}\s',\mu'\ \ \\ \ga, \varkappa,\D  \end{subarray}}\le 
C\frac1\ka XY
(\ab{k}_{\begin{subarray}{c}\s,\mu\ \ \\ \ga, \varkappa,\D  \end{subarray}}+\ka e^{-(\s-\s')N}+\eps )\eps
$$
and
$$\ab{ R_+}_{\begin{subarray}{c}\s',\mu'\ \ \\ \ga', \varkappa,\D  \end{subarray}}\le 
C XY e^{-(\ga-\ga')\Delta'}\eps$$
for any $\ga_*\le\ga'\le\ga$.

Moreover,
$$|| \p_\r^j (\Phi(x,\r)-x)||_{\ga'}+ \aa{ \p_\r^j (d\Phi(x,\r)-I)}_{\ga',\vark} \le C\frac1\ka XY \eps$$
for any $x\in \O_{\ga'}(\s',\mu')$, $\ga_*\le\ga'\le\ga$ and $\ab{j}\le{s_*}$, and for any $\r\in\D$.

Finally, if  $\tilde \r=(0,\r_2,\dots,\r_p)$ and $f^T(\cdot,\tilde \r)=0$ for all $\tilde \r$,  then $f_+-f=R_+=h_+=0$ and $\Phi(x,\cdot)=x$
 for all $\tilde \r$.
 
The exponent $\exp_1$ only depends  on $\frac{d_*}{\beta_1} $ and $\#\cA$. The exponent $\exp_2$ only depends  
on $d_*,m_*,s_*$  and $\#\cA$. The exponent $\exp_3$ only depends on $s_*$. The exponent
$\alpha$ is a positive constant only depending on  $s_*,\frac{d_*}{\vark} ,\frac{d_*}{\beta_3} $.
$C$ is an absolut constant that depends on $c$ and $\sup_\D\ab{\om}$.

\end{lemma}

\subsection{A finite induction}
We shall first make a finite iteration without changing the normal form in order to decrease
strongly the size of the perturbation. We shall restrict ourselves to the case when $N=\Delta'$.

\begin{lemma}\label{Birkhoff}
There exists a constant $C$ such that, for any
$$h\in\NF_{\varkappa}(\Delta,\de),\quad\vark>0,\quad\delta \le \frac1C c',$$
and for any 
$$\Delta'\ge \Delta\ge 1,\quad \ka\le\frac1C c',$$ 
there exists a subset $\D'=\D( h, \ka,\Delta')\subset \D$, satisfying 
$$\Leb(\D\setminus {\D'})\le C (\Delta')^{\exp_1}
\big(\frac{\chi+\delta}{\delta_0}\big) (\frac{\ka }{\de_0})^{\alpha},$$
and,  for any $f\in \Tc_{\ga,\varkappa,\D}(\s,\mu)$,  $\mu\le1$,
$$\eps=\ab{f^T}_{\begin{subarray}{c}\s,\mu\ \ \\ \ga, \vark,\D  \end{subarray}}\quad \textrm{and}\quad
\xi=[f]_{\s,\mu,\D}^{\ga,\varkappa},$$
satisfying
$$\eps \leq \frac1C \frac{\ka}{XY},\quad \left\{\begin{array}{ll}
X=(\frac{\Delta' e^{\ga  d_\Delta}}{(\s-\s')(\mu-\mu')}\log\frac1{\eps})^{\exp_2},& \s'<\s,\ \mu'<\mu\\
Y= (\frac{\chi+\delta+\xi}\ka)^{\exp_3},&\ \end{array}\right. $$
there exists a  $\cC^{{s_*}}$ mapping
$$\Phi:\O_{\ga'}(\s',\mu')\times \D\to\O_{'\ga}(\s-\frac{\s-\s'}{2},\mu-\frac{\mu-\mu'}{2}),
\quad \forall \ga_*\le\ga'\le\ga,$$
real holomorphic and symplectic  for each fixed parameter $\r\in\D$, and  functions
$f'\in \Tc_{\ga,\varkappa,\D}(\s',\mu')$ and
$$h'\in\NF_{\varkappa}(\Delta',\de'),$$
such that  
$$(h+f )\circ \Phi= h'+f',\quad \forall \r\in\D',$$
and
$$\ab{ h'- h}_{\begin{subarray}{c}\s',\mu'\ \ \\ \ga, \varkappa,\D  \end{subarray}}\le C XY \eps,$$
$$\xi'=\ab{ f'}_{\begin{subarray}{c}\s',\mu'\ \ \\ \ga', \varkappa,\D  \end{subarray}} \le  \xi+ CXY \eps$$
and 
$$ \eps'=\ab{ (f')^T}_{\begin{subarray}{c}\s',\mu'\ \ \\ \ga', \varkappa,\D  \end{subarray}} \le 
C XY(e^{-\frac12(\s-\s')\Delta'}+ e^{-\frac12(\ga-\ga')\Delta'})\eps,$$
for any $\ga_*\le\ga'\le \ga$.

Moreover,
$$|| \p_\r^j (\Phi(x,\r)-x)||_{\ga'}+ \aa{ \p_\r^j (d\Phi(x,\r)-I)}_{\ga',\vark} \le C\frac1\ka XY \eps$$
for any $x\in \O_{\ga'}(\s',\mu')$, $\ga_*\le\ga'\le\ga$ and $\ab{j}\le{s_*}$, and for any $\r\in\D$.

Finally, if  $\tilde \r=(0,\r_2,\dots,\r_p)$ and $f^T(\cdot,\tilde \r)=0$ for all $\tilde \r$,  then $f'-f=h'=0$ and $\Phi(x,\cdot)=x$
 for all $\tilde \r$.
 
The exponent $\exp_1$ only depends  on $\frac{d_*}{\beta_1} $ and $\#\cA$. The exponent $\exp_2$ only depends  
on $d_*,m_*,s_*$  and $\#\cA$. The exponent $\exp_3$ only depends on $s_*$. The exponent
$\alpha$ is a positive constant only depending on  $s_*,\frac{d_*}{\vark} ,\frac{d_*}{\beta_3} $.
$C$ is an absolut constant that depends on $c$ and $\sup_\D\ab{\om}$.

\end{lemma}

\begin{proof}
Let $N=\Delta'$. Let $\s_1=\s-\frac{\s-\s'}2$, $\mu_1=\mu-\frac{\mu-\mu'}2$ and  $\s_{K+1}=\s'$, $\mu_{K+1}=\mu'$, and let 
$\{\s_j\}_1^{K+1}$ and $\{\mu_j\}_1^{K+1}$ be arithmetical progressions. We take $K$ such that 
$$\ka e^{-(\s_{j}-\s_{j+1})N}\le \eps,$$
i.e. $K\le (\s-\s')\Delta'(\log\frac\ka{\eps})^{-1}$.

We let $f_1=f$ and $k_1=0$, and we let
$\eps_1= [f_1^T]_{\begin{subarray}{c}\s,\mu\ \ \\ \ga, \vark,\D  \end{subarray}}=\eps$, 
$\xi_1=  [f_1]_{\begin{subarray}{c}\s,\mu\ \ \\ \ga, \vark,\D  \end{subarray}}=\xi$, 
$\delta_1=\delta$ and
$\eta_1=[k_1]_{\begin{subarray}{c}\s,\mu\ \ \\ \ga, \vark,\D  \end{subarray}}=0$. 

Define now
$$\eps_{j+1}=C\frac1\ka X_jY_j(\eta_j+\eps_1+\eps_j)\eps_j,$$
$$\xi_{j+1}=\xi_j+ C X_jY_j \eps_j,\quad \eta_{j+1}=\eta_j+CX_jY_j\eps_j,$$
with
$$X_j=(\frac{N\Delta' e^{\ga d_\Delta}}  {(\s_j-\s_{j+1})(\mu_j-\mu_{j+1})})^{\exp_2},\quad Y_j=(\frac{\chi+\delta+\xi_j}{\ka})^{\exp_3},$$
where $C, \exp_2, \exp_3$  are given in Lemma \ref{basic}. One verifies by an immediate induction that 

\begin{sublem*}
There exists an absolute constant $C'$ such that if
$$\eps_1\le\frac 1{C'}
 \frac\ka{ X_1^2Y_1^2}$$
 then, for all $j\ge1$, 
 $$\eps_j\le \frac1C \frac\ka{ X_j^2Y_j^2}\quad\textrm{and}\quad\eps_{j}\le (\Cte \frac{ X^2_1Y^2_1}\ka  \eps_1)^{j-1}\eps_1 ,$$
$$(\xi_{j} - \xi_1)+(\eta_{j}-\eta_1)  \le C'  X_1 Y_1\eps_1.$$
The constant $\Cte$ only depends on $C$ and $\exp_3$.
\end{sublem*}

We can then  apply Lemma \ref{basic} $K$ times to get
$$\Phi_j:\O_{\ga'}(\s_{j+1},\mu_{j+1})\times \D'\to
\O_{\ga'}(\s_j-\frac{\s_j-\s_{j+1}}2,\mu_j-\frac{\mu_j-\mu_{j+1}}2),\quad \ga_*\le\ga'\le\ga_{j}$$
and $f_{j+1}$ and $R_{j+1}$ such that, for $\r\in\D'$,
$$(h+k_j+f_j +S_j)\circ \Phi_j= h+k_j+ h_{j+1}+f_{j+1}+R_{j+1}+S_j\circ \Phi_j$$
with $k_{j+1}=k_j+ h_{j+1}$, $k_1=0$, $S_{j+1}=R_{j+1}+S_j\circ \Phi_j$, $S_1=0$.

We then take $\Phi=\Phi_1\circ\dots\circ \Phi_K$, $h'= h+k_{K+1}$ and
$f'=f_{K+1}+S_{K+1}$.

Then
$$\ab{ h'-h}_{\begin{subarray}{c}\s',\mu'\ \ \\ \ga, \varkappa,\D  \end{subarray}}\le C' X_1Y_1 \eps,
\qquad  \delta'\le C(\Delta')^{\exp_2}X_1Y_1\eps$$
and
$$\xi'=\ab{ f'}_{\begin{subarray}{c}\s',\mu'\ \ \\ \ga, \varkappa,\D  \end{subarray}} \le  \xi+ C'X_1Y_1 \eps$$
and, for $\r\in\D'$,
$$ \eps'=\ab{ (f')^T}_{\begin{subarray}{c}\s',\mu'\ \ \\ \ga', \varkappa,\D  \end{subarray}} \le 
 (\Cte \frac{ X^2_1Y^2_1}\ka  \eps_1)^K\eps +C X_1Y_1 e^{-(\ga-\ga')\Delta'}\eps.$$

For the estimates of $\Phi$, write $\Psi_j=\Phi_j\circ\dots\circ \Phi_K$ and  $\Psi_{K+1}=id$.
For $(x,\r)\in \O_{\ga'}(\s',\mu')\times \D$ we then have 
$$||\Phi(x,\r)-x||_{\ga'}\le 
\sum_{j=1}^K  ||\Psi_j(x,\r)-\Psi_{j+1}(x,\r)||_{\ga'}.$$
Then
$$
||\Psi_j(x,\r)-\Psi_{j+1}(x,\r)||_{\ga'}=||\Phi_j(\Psi_{j+1}(x,\r),\r)-\Psi_{j+1}(x,\r)||_{\ga'}$$
is
$$
\le \Cte \frac1\ka X_1Y_1 \eps_j\max(\frac1{|\s-\s'  |}, \frac1{|\mu-\mu'  |}),$$
by a Cauchy estimate. Hence
$$||\Phi(x,\cdot)-x||_{\ga'}\le
\Cte \frac1\ka X_1Y_1 \max(\frac1{|\s-\s'  |}, \frac1{|\mu-\mu'  |})\eps.$$

The derivatives with respect to $\r$ are obtained in the same way, as is also the estimates of $d\Phi$.

The result now follows if we take $C'$ sufficiently large and increases the exponent $\exp_2$.

\end{proof}

{\it  Proof of sublemma.}
The estimates are true for $j=1$ so we proceed by induction on $j$. Let us assume the estimates 
hold up to $j$. Then, for $k\le j$,
$$Y_{k}\le (\frac{\chi+\delta+\xi_1+2C' X_1Y_1\eps_1}{\ka})^{\exp_3}= 2^{\exp_3} Y_1$$
and 
$$\eps_{j+1}\le   2^{\exp_3} \frac{ X_1Y_1}\ka  
[2C X_1 Y_1\eps_1+\eps_1+\eps_1]\eps_j\le   \Cte \frac{ X^2_1Y^2_1}\ka  \eps_1\eps_j.$$
Then
$$\xi_{j+1}\le \xi_1+\Cte X_1Y_1(\eps_1+\dots+\eps_{j+1})\le \xi_1+2\Cte  X_1Y_1\eps_1$$
and similarly for $\eta_{j+1}$.

\subsection{The infinite induction}

We are now in position to prove our main result, Theorem \ref{main}. 

Let $h$ be a normal form Hamiltonian in $\NF_{\varkappa}(\Delta,\delta)$ and let 
$f\in \Tc_{\ga,\varkappa,\D}(\s,\mu)$ be a perturbation such that
$$ 0<\eps= \ab{f^T}_{\begin{subarray}{c}\s,\mu\ \ \\ \ga, \vark,\D  \end{subarray}},\quad   
\xi=\ab{f}_{\begin{subarray}{c}\s,\mu\ \ \\ \ga, \vark,\D  \end{subarray}}. $$
We  construct the transformation $\Phi$ as the composition of infinitely many transformations $\Phi$ as in Lemma \ref{Birkhoff}. We first specify the choice of all the parameters for $j\geq 1$. 

Let $C,\exp_1,\exp_2,\exp_3$ and $\alpha$ be the constants given in Lemma \ref{Birkhoff}.

\subsubsection{Choice of parameters}
We assume (to simplify) $\ga,\s,\mu\le1$ and  $\Delta\ge1$.
By decreasing $\ga$ or increasing $\Delta$ we can also assume $\ga=(d_{\Delta})^{-1}$.

We  choose for $j\ge1$
$$ \mu_j=\big(\frac 12 +\frac 1 {2^j}\big)\mu\quad\textrm{and}\quad\s_{j}=\big(\frac 1 2 +\frac 1 {2^j}\big)\s.$$
We define inductively the sequences $\eps_j$, $\Delta_j$, $\de_j$ and $\xi_j$ by
\be\left\{\begin{array}{ll}
\eps_{j+1}= \eps^{K_j}CX_1Y_1\eps & \eps_1=\eps\\
\Delta_{j+1} =4K_j
\max(\frac {1} {\s_j-\s_{j+1}},d_{\Delta_j})\log \frac1{\eps}&\Delta_1=\Delta\\
\ga_{j+1}=(d_{\Delta_{j+1}})^{-1}& \ga_1=\ga\\
\de_{j+1}=\de_j+ C   X_jY_j\eps_j& \de_1=\de\ge0 \\
\xi_{j+1}= \xi_j+CX_jY_j\eps_j&\xi_1=\xi\ge\eps,
\end{array}\right.\ee
where
$$\left\{\begin{array}{ll}
X_j=(\frac{\Delta_{j+1} e^{\ga_j  d_{\Delta_j}}}{(\s_j-\s_{j+1})(\mu_j-\mu_{j+1})}\log \frac1{\eps_j})^{\exp_2}
&=(\frac{K_j\Delta_{j+1} e4^{j+1}}{\s\mu}\log \frac1{\eps})^{\exp_2}\\
Y_j=( \frac{\chi+\delta_j+\xi_j}{\ka_j})^{\exp_3}.&
\end{array}\right.$$
The $\ka_j$ is defined implicitly by
$$\eps_j=\frac 1{C}
 \frac{\ka_j}{ X_jY_j},$$
 
 These sequences depend on the choice of $K_j$. We shall let $K_j$ increase like 
$$K_{j}=K^{j}$$
for some $K$ sufficiently large.

\begin{lemma}\label{numerical2} 
There exist  constants $C'$ and $\exp'$  such that, if
$$
K\ge C' $$
and
$$
\eps\le\frac1{C'}\big( \frac{\s\mu}{K\Delta\log\frac1\eps}\big)^{\exp'}
\big(\frac{c' }{\chi+\delta+\xi}\big)^{\exp_3} c',$$

then 
\begin{itemize} 
\item[(i)] 
$$\sum_{k=1}^\infty CX_kY_k\eps_k  \le 2C X_1Y_1 \eps\le \frac1{2C} c'.$$

\item[(ii)] 
$$\sum_{k=1}^\infty C \max(\frac1{\s_k-\s_{k+1}}, \frac1{\mu_k-\mu_{k+1}})X_kY_k\eps_k 
\le 5C \max(\frac1{\s}, \frac1{\mu}) X_1Y_1   \eps\le \frac1{2C} c'.$$

\item[(iii)]
 $$
  \sum_{j\ge1}C \Delta_{j+1}^{\exp_1}   \big(\frac{\chi+\delta_j}{\delta_0}\big) 
     (\frac{\ka_j}{\de_0})^{\alpha}\le
 C'(\frac{K\Delta\log\frac1\eps}{\s\mu})^{\exp'} ( \frac{\chi+\delta+\xi}{\delta_0})^{\beta}
 (\frac{\eps}{\de_0})^{\alpha'}.
 $$
    
\end{itemize}

 $C'$ is an absolut constant that only depends on  $\beta, \vark,c$ and $\sup_\D\ab{\om}$.
The exponent $\exp'$ is an absolute constant that only depends  on $\beta$ and $\vark$. 
The exponents $\alpha'$ and $\beta$ are positive constants only depending on  $s_*,\frac{d_*}{\vark} ,\frac{d_*}{\beta_3} $.

\end{lemma}

 Notice that (i) implies that
$$ C X_jY_j(e^{-\frac12(\s_j-\s_{j+1})\Delta_{j+1}}+ e^{-\frac12(\ga_j-\ga_{j+1})\Delta_{j+1}}  )\eps_j\le \eps_{j+1},$$

\begin{proof}
$\Delta_{j+1} $ is equal to
$$4K_j\max(\frac {1} {\s_j-\s_{j+1}},d_{\Delta_j})\log \frac1{\eps}
\le
(\Cte \frac {1} {\s} \log\frac1{\eps})(2K)^{j^2}\Delta_j^{a}=A(2K)^{j^2}\Delta_j^{a},$$
which, by an induction, is seen to be, by assumption on $\eps$,
$$\le(A(2K)^a\Delta)^{a^j}\le (\frac1\eps)^{a^j}$$
if $a$ is, say,  at least $6$. In the same way one sees that
$$
X_j\le  (\frac1\eps)^{2\exp_2 a^j}.$$

\medskip

(i). For $j=1$, (i) holds by assumption.  Indeed, by definition
$$(CX_1Y_1\eps_1)^{1+\exp_3}=
\ka_1^{1+\exp_3}=CX_1Y_1\ka_1^{\exp_3}\eps_1 \le \Cte
 \big(\frac{K\Delta\log\frac1\eps}{\s\mu}\big)^{\exp_2'}
(\chi+\delta+\xi)^{\exp_3}\eps,$$
which is 
$$\le \big(\frac1{4C} c'\big)^{\exp_3+1}$$
by assumption on $\eps$. 

Assume now (i) holds up to $j-1\ge1$. Then $\de_j\le\de+2CX_1Y_1\eps$ and $\xi_j\le\xi+2CX_1Y_1\eps$, and hence
$$
Y_j\le ( \frac{\chi+\delta+\xi+ 4CX_1Y_1 \eps}{\ka_j})^{\exp_3}\le
\Cte Y_1(\frac{\ka_1}{\ka_j})^{\exp_3},$$
and, by the definition of $\ka_j$,
$$\ka_j^{1+\exp_3}=CX_jY_j\eps_j \ka_j^{\exp_3} \le\Cte Y_1 \ka_1^{\exp_3}X_j\eps_j \le X_j\eps^{K_{j}}$$
by assumption on $\eps$. Hence
$$CX_jY_j\eps_j=\ka_j\le X_j\eps^{2b K_{j}}\le \eps^{2b K_{j}-2\exp_2 a^j}\le  \eps^{b K_{j}}   ,\quad b=\frac1{2(\exp_3+1)},$$
if $K$ is large enough   --  notice that $j\ge2$. This implies that
$$\sum_{k=2}^j C X_kY_k\eps_k  \le  2\eps^{b K_{2}}\le \eps\le CX_1Y_1\eps_1$$
if $K$ is large enough.

The proof of (ii) is similar. To see (iii)  we have for $j\ge 2$
   $$ \Delta_{j+1}^{\exp_1}\ka_j^{\alpha}= 
  \big(\Delta_{j+1}^{\exp'} \ka_j\big)^{\alpha}\le  \big(X_j^{\exp'} \ka_j\big)^{\alpha}
  \le \big(\eps^{2b K_{j}-2(\exp'+1)\exp_2 a^j}\big)^{\alpha}
  $$
  which is
  $$\le \eps^{b K_{j}\alpha } $$
  if $K$ is large enough.

 Therefore
 $$
 \sum_{j\ge1}\Delta_{j+1}^{\exp_1}\ka_j^{\alpha}=\Delta_{2}^{\exp_1}\ka_1^{\alpha}+  2\eps^{b K_{2}\alpha }\le  2\Delta_{2}^{\exp_1}\ka_1^{\alpha}$$
  if $K$ is large enough.
 
\end{proof}

\subsubsection{The iteration}

\begin{proposition}
There exist positive constants $C'$, $\alpha'$ and  $\exp'$ such that, for any
$h\in\NF_{\varkappa}(\Delta,\de)$ and for any 
$f\in \Tc_{\ga,\varkappa,\D}(\s,\mu)$, $0<\ga,\s,\mu\le 1$,
$$ \eps=\ab{f^T}_{\begin{subarray}{c}\s,\mu\ \ \\ \ga, \varkappa,\D  \end{subarray}},\quad 
 \xi=\ab{f}_{\begin{subarray}{c}\s,\mu\ \ \\ \ga, \varkappa,\D  \end{subarray}},$$
if
$$\delta \le \frac1{C'} c'$$
and
$$
\eps(\log \frac1\eps)^{\exp'}\le\frac1{C}\big( \frac{ \max(\ga^{-1} ,d_{\Delta})}{\s\mu}\big)^{-\exp'}
\big(\frac{c'}{\chi+\delta+\xi}\big)^{\exp_3} c' $$

then there exist a set $\D'=\D'(h, f)\subset \D$,
$$\Leb (\D\setminus \D')\leq 
C\big(\log\frac1{\eps} \frac{ \max(\ga^{-1} ,d_{\Delta})}{\s\mu} \big)^{\exp'}
( \frac{\chi+\delta+\xi}{\delta_0})^{\beta}
(\frac{\eps}{\de_0})^{\alpha'},$$
and a $\cC^{{s_*}}$ mapping 
$$\Phi :\O_{\ga_*}(\s/2,\mu/2)\times \D \to\O_{\ga_*}(\s,\mu),$$
real holomorphic and symplectic for given parameter $\rho\in\D$,
and
$$h'\in \NF_{\varkappa}(\infty,\de'),\quad \de'\le \frac{c'}2,$$
such that
$$(h+f)\circ \Phi=h'+f'$$
verifies
$$\ab{ f'-f}_{\begin{subarray}{c}\s/2,\mu/2\ \ \\ \ga_*, \varkappa,\D  \end{subarray}} \le  C'$$
and, for $\r\in\D '$, $(f')^T=0$.

Moreover,
$$\ab{ h'- h}_{\begin{subarray}{c}\s/2,\mu/2\ \ \\ \ga_*, \varkappa,\D  \end{subarray}}\le C'$$
and
$$|| \p_\r^j (\Phi(x,\cdot)-x)||_{\ga_*}+ \aa{ \p_\r^j (d\Phi(x,\cdot)-I)}_{\ga_*,\vark} \le C'$$
for any $x\in \O_{(0,m_*)}(\s',\mu')$ and $\ab{j}\le{s_*}$, and for any $\r\in\D$.

Finally, if  $\tilde \r=(0,\r_2,\dots,\r_p)$ and $f^T(\cdot,\tilde \r)=0$ for all $\tilde \r$,  then $h'=h$ and $\Phi(x,\cdot)=x$
for all $\tilde \r$.

 $C'$ is an absolut constant that only depends on  $\beta, \vark,c$ and $\sup_\D\ab{\om}$.
The exponent $\exp'$ is an absolute constant that only depends  on $\beta$ and $\vark$. 
The exponent $\exp_3$ only depends on $s_*$.
The exponent $\alpha'$ is a positive constant only depending on  $s_*,\frac{d_*}{\vark} ,\frac{d_*}{\beta_3} $.

\end{proposition}

\begin{proof} Assume first that $\ga=d_\Delta^{-1}$.


Choose the number $\mu_j,\s_j,\eps_j,\Delta_j,\ga_j,\de_j,\xi_j,X_j,Y_j,\ka_j$ as 
above in Lemma \ref{numerical2} with $K= C' $.
By the assumption on $\eps$ we can apply Lemma \ref{numerical2}. 

Let $h_1=h$, $f_1=f$ and $\D_1=\D$. Lemma \ref{numerical2} now implies that
we can apply Lemma \ref{Birkhoff} iteratively to get for all $j\ge1$
a set $\D_{j+1}\subset\D_j$ such that
$$\Leb (\D_j\setminus \D_{j+1})\leq   C \Delta_{j+1}^{\exp_2}
\big(\frac{\chi+\delta_j}{\delta_0}\big) (\frac{\ka_j}{\de_0})^{\alpha},$$
a $\cC^{{s_*}}$ mapping
$$\Phi_{j+1} :\O^{\ga'}(\s_{j+1},\mu_{j+1})\times \D_{j+1}\to
\O^{\ga'}(\s_j-\frac{\s_j-\s_{j+1}}{2},\mu_j-\frac{\mu_j-\mu_{j+1}}{2}),\quad \forall \ga_*\le\ga'\le\ga_{j+1},$$
real holomorphic and symplectic  for each fixed parameter $\r$, 
and  functions
$f_{j+1}\in \Tc_{\ga,\varkappa,\D}(\s_{j+1},\mu_{j+1})$ and
$$h_{j+1}\in \NF_{\varkappa}(\Delta_{j+1},\de_{j+1})$$
such that
$$(h_j+f_j)\circ \Phi_{j+1}=h_{j+1}+f_{j+1},\quad\forall \r\in \D_{j+1},$$
with

$$\ab{ f_{j+1}^T}_{\begin{subarray}{c}\s_{j+1},\mu_{j+1}\ \ \\ \ga_{j+1}, \varkappa,\D  \end{subarray}} 
\le  \eps_{j+1}$$
and
$$\ab{ f_{j+1}}_{\begin{subarray}{c}\s_{j+1},\mu_{j+1}\ \ \\ \ga_{j+1}, \varkappa,\D  \end{subarray}} 
\le  \xi_{j+1}. $$

Moreover,
$$\ab{ h_{j+1}- h_j}_{\begin{subarray}{c}\s_{j+1},\mu_{j+1}\ \ \\ \ga_{j+1}, \varkappa,\D  \end{subarray}}
\le C X_jY_{j} \eps_j$$
and
$$|| \p_\r^l (\Phi_{j+1}(x,\cdot)-x)||_{\ga'}+ \aa{ \p_\r^l (d\Phi_{j+1}(x,\cdot)-I)}_{\ga',\vark} \le C\frac1{\ka_j} X_jY_j \eps_j$$
for any $x\in \O_{\ga'}(\s_{j+1},\mu_{j+1})$, $\ga_*\le \ga'\le\ga_{j+1}$ and $\ab{l}\le{s_*}$.

We let $h'=\lim h_j$, $f'=\lim f_j$ and  $\Phi=\Phi_2\circ\dots\circ \Phi_3\circ\dots $.
Then $(h+f)\circ \Phi=h'+f'$ and $h'$ and $f'$ verify the statement. The convergence of $\Phi$
and its estimates follows by Cauchy estimates as in the proof of Lemma \ref{Birkhoff}.

The last statement is obvious.

If  $\ga>(d_{\Delta})^{-1}$, then we can just decrease $\ga$ and we obtain the same result. If 
$\ga<(d_{\Delta})^{-1}$, then we increase $\Delta$ and we obtain the same result.
\end{proof}

Theorem \ref{main} now follows from this proposition.

\section{Examples}

\subsection{Beam equation with a convolutive potential
}\label{s4.1}
Consider the $d_*$ dimensional beam equation on the torus
\be \label{beamm}u_{tt}+\Delta^2 u+V\star u + \eps g(x,u)=0 ,\quad   x\in \T^{d_*}.
\ee
 Here  $g$ is a real analytic function on $\T^{d_*}\times I$, where  $I$ is a 
  neighborhood of the origin in $\R$, and the 
   convolution potential $V:\ \T^{d_*}\to \R$ is supposed to be analytic with
   real  Fourier coefficients $\hat V(a)$, $a\in\Z^{d_*}$. 
   
     Let $\cA$ be any subset of cardinality $n$ in $\Z^{d_*}$. We set  $\L=\Z^{d_*}\setminus \cA$, 
$
\rho=(\hat V_a)_{a\in\cA},
$
and treat $\rho$ as a parameter of the equation,
$$
\rho=(\rho_{a_1},\dots,\rho_{a_n})
\in\D=[\rho_{a_1'},\rho_{a_1^{''}}]\times\dots \times[\rho_{a_n'},\rho_{a_n^{''}}]
$$
 (all other Fourier coefficients are fixed). We denote $\mu_a=|a|^4+ \hat V(a)$, $a\in\Z^{d_*}$, 
 and assume 
 that $\mu_a >0$ for all $a\in \cA$, i.e. $|a|^4+\rho_a'>0$ if $a\in\cA$. We also suppose that 
 $$
 \mu_l\ne0,\quad \mu_{l_1}\ne\mu_{l_2} \qquad \forall \, l, l_1, l_2 \in\L,\ l_1\ne l_2. 
 $$
 
 Denote 
$$
\F=\{a\in\L:  \mu_a<0\}, \;\; |\F|=: {{N}},\quad \L_\infty = \L\setminus \F\,,
$$
consider the operator
$$
\Lambda=|\Delta^2 +V\star\ |^{1/2}=\diag \{\lambda_a, a\in \Z^{d_*}\}\,, \quad \lambda_a= \sqrt{|\mu_a|}\,,
$$
and
the following operator $\Lambda^{\#}$, 
linear over real numbers:
$$
\LL(ze^{i \langle   a, x\rangle })
= \left\{\begin{array}{ll}\
\  z\la e^{i  \langle  a, x\rangle }, \;\; a\in\L_{\infty}\,,\\
 -\bar z \la e^{i \langle a,x\rangle  }, \;\; a\in\F,
\end{array}\right.
$$
Introducing the complex variable 
 $$
 \psi= \frac 1{\sqrt 2}(\Lambda^{1/2}u- i\Lambda^{-1/2}\dot u)=    (2\pi)^{-d/2}   \sum_{a\in\Z^{d_*}}\psi_a e^{i \langle   a, x\rangle }\,,
 $$
 we get for it   the equation (cf.    \cite[Section 1.2]{EGK})
\be\label{k1}
\dot \psi=i \big( \Lambda^{\#}\psi+ \eps \frac1{\sqrt2}\Lambda^{-1/2} g\left(x,\Lambda^{-1/2} \left(\frac{\psi+\bar\psi}{\sqrt 2}\right)\right)\,.
\ee
Writing $\psi_a=(u_a+iv_a)/\sqrt2$ we see that eq.~\eqref{k1} is a 
Hamiltonian system with respect to the symplectic form $\sum dv_s\wedge du_s$ and the Hamiltonian $h=h_{\textrm up}+\eps P$, where
$$
 P= \int_{\T^{d_*}}  G\left(x,\Lambda^{-1/2} \left(\frac{\psi+\bar\psi}{\sqrt 2}\right)\right) \dd x\,,\qquad \p_u G(x,u)=g(x,u)\,,
$$
and $h_{\textrm up}$ is  the quadratic Hamiltonian 
$$
h_{\textrm up} ({u}, {v}) = \sum_{a\in\cA} \lambda_a |{\psi}_a |^2 
+\Big\langle 
\left(\begin{array}{c} {u}_\F \\ {v}_F \end{array}\right), H
\left(\begin{array}{c} {u}_F \\ {v}_F \end{array}\right)
\Big\rangle
+   \sum_{a\in\L_\infty}  \lambda_a |{\psi}_a |^2 \,.
$$
Here 
 ${u}_\F={}^t({u}_a, a\in\F)$ and 
 $H$ is a symmetric $2{{N}}\times2{{N}}\,$-matrix.
 The $2{{N}}$  eigenvalues  of the Hamiltonian operator with 
 the matrix $H$ are the real  numbers $\{\pm\la, a\in\F\}$.
  So the linear system \eqref{beamm}${}\mid_{\eps=0}$  is stable if and only if $n=0$.

Let us fix any $n$ vector $I=\{I_a>0,a\in\cA\}$ with positive components. 
The $n$-dimensional torus 
\ben \left\{\begin{array}{ll}
     |\psi_a|^2 =I_a,\quad &a\in \cA\\
\psi_a=0,\quad & a\in \L=\Z^{d_*}\setminus \cA,
\end{array}\right.
\een
is invariant for the unperturbed linear equation; it is linearly stable if and only if ${{N}}=0$. 
In the linear
space span$\{\psi_a, a\in\cA\}$ we  introduce the action-angle variables $(r_a,\theta_a)$ through  the relations 
$
\psi_a=\sqrt{(I_a+r_a)}e^{i\theta_a}$, $  a\in\cA. $
The  unperturbed  Hamiltonian  becomes
$$
h_{\textrm up}= \text{const} +\langle r,\om(\r)\rangle
+\Big\langle 
\left(\begin{array}{c} {u}_\F \\ {v}_\F \end{array}\right), H
\left(\begin{array}{c} {u}_\F \\ {v}_\F \end{array}\right)
\Big\rangle
 +
 \sum_{a\in\L_\infty}\lambda_a |\psi_a|^2\,,
$$
with  $ \om(\r)=( \omega_a=\lambda_a, \,{a\in\cA} )$,  and the perturbation  becomes 
$$
P=\eps \int_{\T^{d_*}}G
\left(x, \hat u(r,\theta;\zeta)(x)
\right)\dd x, \quad \hat u(r,\theta;\zeta)(x) = \Lambda^{-1/2}\Big(\frac{\psi+\bar\psi}{\sqrt2}\Big),
$$
 i.e. 
$$
\hat u
=\sum_{a\in\cA}   \frac 
{ \sqrt{(I_a+r_a)} \,(e^{i\theta_a}\phi_a +  e^{-i\theta_a}\phi_{-a}   )} 
 {\sqrt{ 2\la}}
 +  \sum_{a\in\L}\frac{\psi_a\phi_a +\bar\psi_a\phi_{-a}}{\sqrt{ 2\la}}.
$$
In the  symplectic coordinates $((u_a, v_a), a\in \L)$ the Hamiltonian $h_{\textrm up}$ has
 the form \eqref{equation1.1}, 
  and we wish to apply to the  Hamiltonian $h=h_{\textrm up}+\eps P$
  Theorem \ref{main} and Corollary \ref{cMain}. The assumption A1
  with constants $c,c'$ of order one, $\beta_1=\beta_2=\beta_3=2$ holds trivially. The
  assumption A2 also holds since for each case (i)-(iii) the second alternative 
  with $\omega(\rho)=\rho$ 
  is fulfilled for some  $\delta_0\sim1$. Finally, the assumptions R1 and R2 with $\varkappa=1$ and suitable 
  constants $\gamma_1, \sigma, \mu>0$ and $\gamma_2=m_*$ 
    are valid in view of Lemma~3.2 in \cite{EGK}. 
    More exactly, the validity of the assumption R1 is a part of the
    lemma's assertion. The lemma also states that the second differential  $Jd^2 f$ defines holomorphic mappings
    $$
    Jd^2f:   \O_{\ga'}(\s,\mu) \to M_{\ga'}^D\,,\qquad\ga'\le\ga\,,
    $$
    where $ M_{\ga}^D$ is the space of matrices $A$, formed by $2\times 2$-blocs $A_a^b$, such that
    $$
    |A|_\ga^D := \sup_{a,b} \langle a\rangle \langle b\rangle |A|_a^b \max([a-b], 1)^{\ga_2}e^{\ga_1 [a-b]}<\infty\,. 
    $$
    It is easy to see that $M_{\ga}^D\subset  \mathcal M^b_{\ga, \varkappa}$ if $\varkappa=1$ and $m_*$, entering the definition of 
    $ \mathcal M^b_{\ga, \varkappa}$,     is sufficiently big.
    So R2 also holds.

\medskip

 Let us set $ u_0(\theta,x) = \hat u(0,\theta;0)(x) $.
 Then  for every $I\in\R_+^n$ and $\theta_0\in\T^{d_*}$ 
 the function $(t,x)\mapsto u_0(\theta_0+t\om,x)$ is a solution of \eqref{beamm} with
 $\eps=0$. Application of 
  Theorem \ref{main} and Corollary \ref{cMain} gives us the following result:
 
\begin{theorem}\label{t72}
For $\eps$ sufficiently small there is a Borel subset
${\D_\eps}\subset \D$, 
$\, \meas(\D\setminus{\D_\eps})\leq C\eps^\alpha$, $\alpha>0$,  
such that for  $\rho\in{\D_\eps}$ there is a function 
$ u_1(\theta,x)$, analytic in $\theta\in\T^n_{\frac\s 2}$ and $H^{d^*}$-smooth in $x\in\T^{d_*}$, satisfying 
$$\sup_{|\Im\theta|<\frac\s 2}\|u_1(\theta,\cdot)-u_0(\theta,\cdot)\|_{H^{d^*}(\T^{d_*})}
\leq \beta\eps,$$
and there is a mapping  $\om':{\D_\eps}\to \R^n$,
$\ \|\om'-\om\|_{C^1({\D_\eps})}\leq \beta\eps,$
such that for  $\r\in {\D_\eps}$ the function 
$\ 
u(t,x)=u_1(\theta+t\om'(\r),x)
$
is a   solution of the beam equation \eqref{beamm}.
Equation \eqref{k1}, linearised around its solution $\psi(t)$, corresponding to
the solution $u(t,x)$ above, has exactly ${{N}}$ unstable and ${{N}}$ stable directions.
 \end{theorem}
 
 The last assertion of this theorem follows from the item (iii) of Theorem \ref{main} which 
 implies that the linearised equation, in the directions, corresponding to $\L$, reduces 
 to a linear equation with a  coefficient matrix which can be written as $B=B_\F \oplus B_\infty$. 
 The operator $B_\F $ is close to the Hamiltonian operator with the matrix $H$, so it has ${{N}}$ stable
 and ${{N}}$ unstable directions, while the matrix $B_\infty$ is skew-symmetric, so it has imaginary
 spectrum.

\begin{remark} This result was  proved by Geng and You  \cite{GY06a}
 for  the case when the perturbation 
 $g$ does not depend on $x$ and the unperturbed linear equation is stable.
\end{remark}

\subsection{NLS equation with a smoothing nonlinearity
}\label{s4.2}
Consider the NLS equation with the Hamiltonian 
$$
g(u)=\tfrac12\int|\nabla u|^2\,dx+\frac{m}2\int|u(x)|^2\,dx 
+\eps\int f(t,\De^{-\alpha}u(x),x)\,dx,
$$
where $m\ge0, \ \alpha>0, 
\ u(x)$ is a complex function on the torus $\T^{d_*}$ and $f$ is a real-analytic 
function on $\R\times  \R^2\times\T^{d_*}$
 (here we regard $\C$ as $\R^2$).  The corresponding 
Hamiltonian equation is 
\begin{equation}\label{-2.1}
\dot u=i \big(-\Delta+mu+\eps \De^{-\alpha}\nabla_2 f(t,\De^{-\alpha}u(x),x)\big)\,,
\end{equation}
where $\nabla_2$ is the gradient with respect to the second variable, $u\in\R^2$. 
 We have to introduce in this equation a vector-parameter $\rho\in\R^n$.
 To do this we can either assume that $f$ is time-independent and 
  add a convolution-potential term $V(x,\rho)*u$ (cf. \eqref{beamm}), 
 or assume that $f$ is a quasiperiodic function of time, $f=F(\rho t,u(x),x)$, where $\rho\in\D\Subset\R^n$.
 Cf.  \cite{BB}. 

Let us discuss the second option. In this case the 
 non-autonomous equation \eqref{-2.1} can be written as an autonomous system on the 
 extended phase-space $\O\times\T^n\times L_2=\{(r,\theta,u(\cdot))\}$, where 
 $ L_2=L_2(\T^{d_*};\R^2)$ and  $\O$ is a ball in $\R^n$, with the Hamiltonian 
 \begin{equation*}
 \begin{split}
&g(r,u,\rho)=h_{\textrm up}(r,u,\rho)+\eps\int F(\theta,\De^{-\alpha}u(x),x)\,dx,\\
&h_{\textrm up}(r,u, \rho)=\langle \rho, r\rangle +
\tfrac12\int|\nabla u|^2\,dx+\frac{m}2\int|u(x)|^2\,dx. 
\end{split}
\end{equation*}
Assume that $m>0$\footnote{\ if undesirable, the term $imu$  can be removed from eq.~\eqref{-2.1}
by means of the substitution $u(t,x)=u'(t,x)e^{imt}$.}
 and take for $A_{\textrm up}$ the operator $-\Delta+m$ with the eigenvalues 
$\lambda_a=|a|^2+m$. Then the Hamiltonian $g(r,u,\rho)$ has the form, required by  Theorem \ref{main}
 with
 $$
 \L=\Z^{d_*},\quad \F=\emptyset, \quad \varkappa=\min(2\alpha,1), \quad  
  \beta_1=2, \quad \beta_2=0, \quad \beta_3=2  $$
 (any $\beta_3$ will do here in fact)
 and suitable $\sigma, \mu, \gamma_1>0$ and $\ga_2=m_*$. The 
   theorem applies and implies  that,  for a typical $\rho$,  equation \eqref{-2.1}
   has time-quasiperiodic solutions of order $\eps$. The equation, linearised 
about these solutions, reduces to constant coefficients and all its Lyapunov exponents are zero.

If $\alpha=0$, equations  \eqref{-2.1}   become significantly more complicated. Still the 
assertions above remain true since they follow from the KAM-theorem  in \cite{EK10}. 
Cf.  \cite{EK09}, where is considered nonautonomous linear Schr\"odinger equation, which is
equation \eqref{-2.1} with the perturbation $\eps (-\Delta)^{-\alpha}\nabla_2f$ replaced by $\eps V(\rho t,x)u$,
and it is proved that this equation reduces to an autonomous equation by means of a time-quasiperiodic linear 
change  of  variable $u$. In \cite{BB}  equation \eqref{-2.1} with $\alpha=0$ and $f=F(\rho t,\De^{-\alpha}u(x),x)$
 is considered for the case 
when the constant-potential term $mu$ is replaced by  $V(x)u$ with arbitrary sufficiently smooth 
potential 
$V(x)$. It is proved that for a typical $\rho$ the equation has small time-quasiperiodic solutions, but not that the linearised equations are reducible to constant coefficients.

\appendix

\section{}

\subsubsection{Transversality}\label{ssTransversality}
\

Let $\D$ be the unit ball in $\R^p$. For any matrix-valued function
$$f:\D\to gl(\dim,\C),$$
let
$$\Sigma(f,\eps)=\{\r\in\D: \aa{f(\r)^{-1}}>\frac1{\eps}\},$$
where $\aa{\ \ }$ is the operator norm.
 
\begin{lemma}\label{lTransv1} Let $f:\D\to\C$ be a $\cC^{{s_*}}$-function which is  $({\mathfrak z},j,\de_0)$-transverse,
$1\le j\le {s_*}$.

Then, 
$$
\Leb \{\r\in\D: \ab{f(\r)}<\eps\}\le C \frac{|\nabla_\r f |_{\cC^{{{s_*}}-1}(\D)}}{\de_0}(\frac\eps{\de_0})^{\frac1j}.$$

$C$ is a constant that only depends on ${s_*}$ and $p$.
\end{lemma}

\begin{proof} It is enough to prove this for ${\mathfrak z}=(1,0,\dots,0)$, i.e. for a scalar $\r$. It is a well-known result, see for example Lemma B.1 in \cite{E02}, that
$$
\Leb(\Sigma(f,\eps))\le C \frac{|f |_{\cC^{{{s_*}}}(\D)}}{\de_0}(\frac\eps{\de_0})^{\frac1j}.$$
This implies the claim.

\end{proof}

\subsubsection{Extension}\label{ssExtension}
\

\begin{lemma}\label{lExtension} Let $X\subset Y$ be subsets of $\D_0$ such that
$$\dist(\D_0\setminus Y,X)\ge \eps,$$
then there exists a $\cC^\infty$-function $g:\D_0\to\R$, being $=1$ on $X$ and $=0$ outside
$Y$ and such that for all $j\ge 0$
$$| g |_{\cC^j(\D_0)}\le C(\frac C{\eps})^j.$$
$C$ is an absolute constant.

\end{lemma}

\begin{proof}
This is  a classical result obtained by convoluting the characteristic function of $X$
with a $\cC^\infty$-approximation of the Dirac-delta supported in a ball of radius $\le  \frac{\eps}2$.
\end{proof}

\end{document}